\documentclass{article}
\usepackage{amsmath}
\usepackage{amsthm}

\usepackage{textcomp, arcs}
\usepackage{times}
\usepackage{xcolor}
\usepackage{enumitem}
\usepackage{xparse,etoolbox}
\usepackage{caption}
\usepackage{chngcntr}
\counterwithin{figure}{section}
\counterwithin{table}{section}
\usepackage{tikz}
\usepackage{graphicx}
\usepackage[mathscr]{eucal}
\usepackage{scrextend}
\usepackage{wasysym}
\usepackage{etoolbox}
\makeatletter

\usepackage[all]{xy}

\usepackage{etoolbox}
\makeatletter
\patchcmd{\@makechapterhead}{50\p@}{\chapheadtopskip}{}{}

\patchcmd{\@makeschapterhead}{50\p@}{\chapheadtopskip}{}{}
\makeatother
\newlength{\chapheadtopskip}
\usepackage[parfill]{parskip} 
\usepackage{amssymb}
\usepackage{imakeidx}
\usepackage{enumitem}
\usepackage[mathlines]{lineno}
\DeclareMathAlphabet{\mathpzc}{OT1}{pzc}{m}{it}
\makeatletter
\newcommand{\mylabel}[2]{#2\def\@currentlabel{#2}\label{#1}}
\makeatother
\usepackage{fancyhdr}
\usepackage{enumitem}
\usepackage{xpatch}
\newtheorem{theorem}{Theorem}[section]
\newtheorem{lemma}[theorem]{Lemma}

\newtheorem{defn}[theorem]{Definition}
\newtheorem{prop}[theorem]{Proposition}
\newtheorem{cor}[theorem]{Corollary}

\newtheorem{Claim}[theorem]{Claim}

\newtheorem{subclaim}{Subclaim}[theorem]
\newcounter{claimlevel}[theorem]

\ExplSyntaxOn
\NewDocumentEnvironment{claim}{O{=}}
 {
  \str_case:nn { #1 }
   {
    {=}  { }
    {+}  { \stepcounter{claimlevel} }
    {-}  { \addtocounter{claimlevel}{-1} }
   }
  \begin{ Claim \int_to_Roman:n { \value{claimlevel} } }
 }
 {
  \end{ Claim \int_to_Roman:n { \value{claimlevel} } }
 }
\ExplSyntaxOff

\newenvironment{claimproof}[1]{\par\noindent\underline{Proof:}\space#1}{\hspace{1mm}$\blacksquare$}

\usepackage[left=2.54cm,top=2.54cm,right=2.54cm,bottom=2.54cm]{geometry}
\usepackage{setspace}
\newlength\FHoffset
\fancyheadoffset{\FHoffset}
\newlength\FHright
\makeatletter
\usepackage{framed}

 \newtheoremstyle{TheoremNum}
        {\topsep}{\topsep}              
        {\itshape}                      
        {}                              
        {\bfseries}                     
        {.}                             
        { }                             
        {\thmname{#1}\thmnote{ \bfseries #3}}
    \theoremstyle{TheoremNum}
    \newtheorem{thmn}{Theorem}

\newtheoremstyle{PropNum}
        {\topsep}{\topsep}              
        {\itshape}                      
        {}                              
        {\bfseries}                     
        {.}                             
        { }                             
        {\thmname{#1}\thmnote{ \bfseries #3}}
    \theoremstyle{PropNum}

\newtheoremstyle{LemmaNum}
        {\topsep}{\topsep}              
        {\itshape}                      
        {}                              
        {\bfseries}                     
        {.}                             
        { }                             
        {\thmname{#1}\thmnote{ \bfseries #3}}
    \theoremstyle{LemmaNum}
    
\makeatletter
\renewcommand\subitem{\@idxitem\nobreak\hspace*{20\p@}}
\renewcommand\subsubitem{\@idxitem\nobreak\hspace*{20\p@}}
\makeatother

\date{}
\makeatother
\title{Extensions of Thomassen's Theorem to Paths of Length At Most Four: Part II}
\author{Joshua Nevin}
\begin{document}
\maketitle

\begin{center}\textbf{Abstract}\end{center} Let $G$ be a planar embedding with list-assignment $L$ and outer cycle $C$, and let $P$ be a path of length at most four on $C$, where each vertex of $G\setminus C$ has a list of size at least five and each vertex of $C\setminus P$ has a list of size at least three. This is the second paper in a sequence of three papers in which we prove some results about partial $L$-colorings $\phi$ of $C$ with the property that any extension of $\phi$ to an $L$-coloring of $\textnormal{dom}(\phi)\cup V(P)$ extends to $L$-color all of $G$, and, in particular, some useful results about the special case in which $\textnormal{dom}(\phi)$ consists only of the endpoints of $P$. We also prove some results about the other special case in which $\phi$ is allowed to color some vertices of $C\setminus\mathring{P}$ but we avoid taking too many colors away from the leftover vertices of $\mathring{P}\setminus\textnormal{dom}(\phi)$. We use these results in a later sequence of papers to prove some results about list-colorings of high-representativity embeddings on surfaces. 

\section{Introduction}\label{IntroMotivSec}

Given a graph $G$, a \emph{list-assignment} for $G$ is a family of sets $\{L(v): v\in V(G)\}$ indexed by the vertices of $G$, such that $L(v)$ is a finite subset of $\mathbb{N}$ for each $v\in V(G)$. The elements of $L(v)$ are called \emph{colors}. A function $\phi:V(G)\rightarrow\bigcup_{v\in V(G)}L(v)$ is called an \emph{$L$-coloring of} $G$ if $\phi(v)\in L(v)$ for each $v\in V(G)$, and, for each pair of vertices $x,y\in V(G)$ such that $xy\in E(G)$, we have $\phi(x)\neq\phi(y)$. Given a set $S\subseteq V(G)$ and a function $\phi: S\rightarrow\bigcup_{v\in S}L(v)$, we call $\phi$ an \emph{ $L$-coloring of $S$} if $\phi(v)\in L(v)$ for each $v\in S$ and $\phi$ is an $L$-coloring of the induced graph $G[S]$. A \emph{partial} $L$-coloring of $G$ is a function of the form $\phi:S\rightarrow\bigcup_{v\in S}L(v)$, where $S$ is a subset of $V(G)$ and $\phi$ is an $L$-coloring of $S$. Likewise, given a set $S\subseteq V(G)$, a \emph{partial $L$-coloring} of $S$ is a function $\phi:S'\rightarrow\bigcup_{v\in S'}L(v)$, where $S'\subseteq S$ and $\phi$ is an $L$-coloring of $S'$. Given an integer $k\geq 1$, a graph $G$ is called \emph{$k$-choosable} if, for every list-assignment $L$ for $G$ such that $|L(v)|\geq k$ for all $v\in V(G)$, $G$ is $L$-colorable.

This is the second in a sequence of three papers. The motivation for this sequence of papers is outlined in the introduction of Paper I (\cite{JNHolepunchPaperICitation}). We use the main result of this sequence of papers in a later sequence of papers to prove a result about high-representativty embeddings on surfaces with distant precolored components.  We now recall some notation from Paper I. Given a graph $G$ with list-assignment $L$, we very frequently analyze the situation where we begin with a partial $L$-coloring $\phi$ of a subgraph of $G$, and then delete some or all of the vertices of $\textnormal{dom}(\phi)$ and remove the colors of the deleted vertices from the lists of their neighbors in $G\setminus\textnormal{dom}(\phi)$. We thus make the following definition. 

\begin{defn}\label{GeneralListDefn1}\emph{Let $G$ be a graph, let $\phi$ be a partial $L$-coloring of $G$, and let $S\subseteq V(G)$. We define a list-assignment $L^S_{\phi}$ for $G\setminus (\textnormal{dom}(\phi)\setminus S)$ as follows.}
$$L^S_{\phi}(v):=\begin{cases} \{\phi(v)\}\ \textnormal{if}\ v\in\textnormal{dom}(\phi)\cap S\\ L(v)\setminus\{\phi(w): w\in N(v)\cap (\textnormal{dom}(\phi)\setminus S)\}\ \textnormal{if}\ v\in V(G)\setminus \textnormal{dom}(\phi) \end{cases}$$ \end{defn}

If $S=\varnothing$, then $L^{\varnothing}_{\phi}$ is a list-assignment for $G\setminus\textnormal{dom}(\phi)$ in which the colors of the vertices in $\textnormal{dom}(\phi)$ have been deleted from the lists of their neighbors in $G\setminus\textnormal{dom}(\phi)$. The situation where $S=\varnothing$ arises so frequently that, in this case, we simply drop the superscript and let $L_{\phi}$ denote the list-assignment $L^{\varnothing}_{\phi}$ for $G\setminus\textnormal{dom}(\phi)$. In some cases, we specify a subgraph $H$ of $G$ rather than a vertex-set $S$. In this case, to avoid clutter, we write $L^H_{\phi}$ to mean $L^{V(H)}_{\phi}$. Our main result for this set of three papers, which we prove in Paper III, is Theorem \ref{MainHolepunchPaperResulThm} below. The work of Paper I and this paper build up to the proof of this result. 

\begin{theorem}\label{MainHolepunchPaperResulThm} (Holepunch Theorem) Let $G$ be a planar embedding with outer cycle $C$. Let $P:=p_0q_0zq_1p_1$ be a subpath of $C$ whose three internal vertices have no common neighbor in $C\setminus P$, and let $L$ be a list-assignment for $V(G)$ such that the following hold.
\begin{enumerate}[label=\arabic*)] 
\itemsep-0.1em
\item $|L(p_0)|+|L(p_1)|\geq 4$ and each of $L(p_0)$ and $L(p_1)$ is nonempty; AND
\item For each $v\in V(C\setminus P)$, $|L(v)|\geq 3$; AND
\item For each $v\in\{q_0, z, q_1\}\cup V(G\setminus C)$, $|L(v)|\geq 5$.
\end{enumerate}
Then there is a partial $L$-coloring $\phi$ of $V(C)\setminus\{q_0, q_1\}$, where $p_0, p_1, z\in\textnormal{dom}(\phi)$, such that each of $q_0, q_1$ has an $L_{\phi}$-list of size at least three, and furthermore, any extension of $\phi$ to an $L$-coloring of $\textnormal{dom}(\phi)\cup\{q_0, q_1\}$ also extends to $L$-color all of $G$. 
 \end{theorem}

Theorem \ref{MainHolepunchPaperResulThm} is a statement about paths of lengths four in facial cycles of planar graphs. In order to prove this, we need some intermediate facts about paths of lengths 2, 3, and 4 in facial cycles of planar graphs. We proved some of these results in Paper I. We need two more results of this form, which we prove in this paper. These results are Theorem \ref{CornerColoringMainRes} and Lemma \ref{EndLinked4PathBoxLemmaState}, whose proofs make up all of the new content of this paper. In Sections \ref{NotationFromPISec}-\ref{PIBlackBoxSec}, we recall the notation that we introduced in Paper I and the statements of the results that we proved in Papers I. It is not necessary to have read Paper I in order to read this paper (more generally, any of the three papers can be read independently of the other two). All of the notation introduced in those papers is re-introduced below, and we restate the results from Papers I in the same language. In Section \ref{CorColSecRes}, we prove Theorem \ref{CornerColoringMainRes} and, in Sections \ref{BoxLemmaFor4PathSecPropIntermediate}-\ref{Main4PathLemmaSec}, we prove Lemma \ref{EndLinked4PathBoxLemmaState}. 

\section{Background}\label{BackgroundSect}

In 1994, Thomassen demonstrated in \cite{AllPlanar5ThomPap} that all planar graphs are 5-choosable, settling a problem that had been posed in the 1970's. Actually, Thomassen proved something stronger. 

\begin{theorem}\label{thomassen5ChooseThm}
Let $G$ be a planar graph with facial cycle $C$ and let $xy\in E(C)$. Let $L$ be a list assignment for $G$, where  vertex of $G\setminus C$ has a list of size at least five and each vertex of $V(C)\setminus\{x,y\}$ has a list of size at least three, where $xy$ is $L$-colorable. Then $G$ is $L$-colorable.
\end{theorem}

Theorem \ref{thomassen5ChooseThm} has the following two useful corollaries, which we use frequently. 

\begin{cor}\label{CycleLen4CorToThom} Let $G$ be a planar graph with outer cycle $C$ and let $L$ be a list-assignment for $G$ where each vertex of $G\setminus C$ has a list of size at least five.  If $|V(C)|\leq 4$ then any $L$-coloring of $V(C)$ extends to an $L$-coloring of $G$. \end{cor}

\begin{cor}\label{2ListsNextToPrecEdgeCor}
Let $G$ be a planar graph with facial cycle $C$ and $p_0p_1\in E(C)$ and, for each $i=0,1$, let $u_i$ be the unique neighbor of $p_i$ on the path $C\setminus\{p_0, p_1\}$. Let $L$ be a list assignment for $G$, where each vertex of $G\setminus C$ has a list of size at least five and each vertex of $C\setminus\{p_0, p_1, u_0, u_1\}$ has a list of size at least three.  Let $\phi$ be an $L$-coloring of $p_0p_1$, where $|L(u_i)\setminus\{\phi(p_i)\}|\geq 2$ for each $i=0,1$. Then $\phi$ extends to an $L$-coloring of $G$. \end{cor}

Because of Corollary \ref{CycleLen4CorToThom}, planar embeddings which have no separating cycles of length 3 or 4 play a special role in our analysis, so we give them a name.

\begin{defn} \emph{Given a planar graph $G$, we say that $G$ is \emph{short-separation-free} if, for any cycle $F\subseteq G$ with $|V(F)|\leq 4$, either $V(\textnormal{Int}_G(F))=V(F)$ or $V(\textnormal{Ext}_G(F))=V(F)$.} \end{defn}

Lastly, we rely on the following very simple result, which is a useful consequence of Theorem 7 from \cite{lKthForGoBoHm6} which characterizes the obstructions to extending a precoloring of a cycle of length at most six in a planar graph. 

\begin{theorem}\label{BohmePaper5CycleCorList} Let $G$ be a short-separation-free graph with facial cycle $C$. Let $L$ be a list-assignment for $G$, where $|L(v)|\geq 5$ for all $v\in V(G\setminus C)$. Suppose that $|V(C)|\leq 6$ and $V(C)$ is $L$-colorable, but $G$ is not $L$-colorable. Then $5\leq |V(C)|\leq 6$, and the following hold.
\begin{enumerate}[label=\arabic*)]
\itemsep-0.1em
\item If $|V(C)|=5$, then $G\setminus C$ consists of a lone vertex which is adjacent to all five vertices of $C$; AND
\item If $|V(C)|=6$, then $G\setminus C$ consists of at most three vertices, each of which has at least three neighbors in $G\setminus C$. Furthermore, $G\setminus C$ consists of one of the following.
\begin{enumerate}[label=\roman*)]
\itemsep-0.1em
\item A lone vertex adjacent to at least five vertices of $C$; OR
\item An edge $x_1x_2$ such that, for each $i=1,2$, $G[N(x_i)\cap V(C)]$ is a path of length three; OR
\item A triangle $x_1x_2x_3$ such that, for each $i=1,2,3$, $G[N(x_i)\cap V(C)]$ is a path of length two.
\end{enumerate}
\end{enumerate}
\end{theorem}

\section{Notation from Paper I}\label{NotationFromPISec}

In this section, we restate some terminology introduced in Paper I. We first have the following definition, which is our main object of study for all three papers.

\begin{defn} \emph{A \emph{rainbow} is a tuple $(G, C, P, L)$, where $G$ is a planar graph with outer cycle $C$, $P$ is a path on $C$ of length at least two, and $L$ is a list-assignment for $V(G)$ such that $|L(v)|\geq 3$ for each $v\in V(C\setminus P)$ and $|L(v)|\geq 5$ for each $v\in V(G\setminus C)$. We say that a rainbow is \emph{end-linked} if, letting $p ,p^*$ be the endpoints of $P$, each of $L(p)$ and $L(p^*)$ is nonempty and $|L(p)|+|L(p^*)|\geq 4$. }
 \end{defn}

We also recall the following definitions. 

\begin{defn} \emph{Given a graph $G$ with list-assignment $L$, a subgraph $H$ of $G$, and a partial $L$-coloring $\phi$ of $G$, we say that $\phi$ is \emph{$(H,G)$-sufficient} if any extension of $\phi$ to an $L$-coloring of $\textnormal{dom}(\phi)\cup V(H)$ extends to $L$-color all of $G$.} \end{defn}

\begin{defn}\label{GeneralAugCrownNotForLink} \emph{Let $G$ be a planar graph with outer cycle $C$, let $L$ be a list-assignment for $G$, and let $P$ be a path in $C$ with $|V(P)|\geq 3$. Let $pq$ and $p'q'$ be the terminal edges of $P$, where $p, p'$ are the endpoints of $P$. We let $\textnormal{Crown}_{L}(P, G)$ be the set of partial $L$-colorings $\phi$ of $V(C)\setminus\{q, q'\}$ such that}
\begin{enumerate}[label=\arabic*)] 
\itemsep-0.1em
\item $V(P)\setminus\{q, q'\}\subseteq\textnormal{dom}(\phi)$ and, for each $x\in\{q, q'\}$, $|L_{\phi}(x)|\geq |L(x)|-2$; AND
\item\emph{$\phi$ is $(P, G)$-sufficient}
\end{enumerate}
 \end{defn}

With the definitions above in hand, we have the following compact restatement of Theorem \ref{MainHolepunchPaperResulThm}.

\begin{thmn}[\ref{MainHolepunchPaperResulThm}] Let $(G, C, P, L)$ be an end-linked rainbow, where $P$ is a path of length four whose three internal vertices have no common neighbor in $C\setminus P$. Suppose further that each internal vertex of $P$ has an $L$-list of size at least five. Then $\textnormal{Crown}_L(P, G)\neq\varnothing$. \end{thmn} 

We also recall the following notation from Paper I and the standard notation of Definition \ref{StandQD}, which we also use frequently throughout this paper. 

\begin{defn} \emph{Given a planar graph $G$ with outer face $C$ and an $H\subseteq G$, we let $C^H$ denote the outer face of $H$.} \end{defn}

\begin{defn}\label{StandQD} \emph{Given a path $Q$ in a graph $G$, we let $\mathring{Q}$ denote the subpath of $Q$ consisting of the internal vertices of $Q$. In particular, if $|E(Q)|\leq 2$, then $\mathring{Q}=\varnothing$. Furthermore, for any $x,y\in V(Q)$, we let $xQy$ denote the unique subpath of $Q$ with endpoints $x$ and $y$. } \end{defn}

As we frequently deal with colorings of paths, we also use the following natural notation. 

\begin{defn} \emph{Let $G$ be a graph with list-assignment $L$. Given an integer $n\geq 1$, a path $P:=p_1\ldots p_n$, and a partial $L$-coloring $\phi$ of $G$ with $V(P)\subseteq\textnormal{dom}(\phi)$, we denote the $L$-coloring $\phi|_P$ of $P$ as $(\phi(p_1), \phi(p_2),\ldots, \phi(p_n))$.} \end{defn}

\begin{defn}\label{DefnForColorSetsonVertex} \emph{Let $G$ be a graph with list-assignment $L$. Given a set $\mathcal{F}$ of partial $L$-colorings of $G$ and a vertex $x\in V(G)$ with $x\in\textnormal{dom}(\phi)$ for each $\phi\in\mathcal{F}$, we define $\textnormal{Col}(\mathcal{F}\mid x)=\{\phi(x): \phi\in\mathcal{F}\}$.} \end{defn}

\section{Black Boxes from Paper I}\label{PIBlackBoxSec}

The first black box from Paper I that we have is a result about broken wheels.

\begin{defn}\emph{A \emph{broken wheel} is a graph $G$ with a vertex $p\in V(G)$ such that $G-p$ is a path $q_1\ldots q_n$ with $n\geq 2$, where $N(p)=\{q_1, \ldots, q_n\}$. The subpath $q_1pq_n$ of $G$ is called the \emph{principal path} of $G$.} \end{defn}

Note that, if $|V(G)|\leq 4$, then the above definition does not uniquely specify the principal path, although in practice, whenever we deal with broken wheels, we specify the principal path beforehand so that there is no ambiguity. 

\begin{defn} \emph{Let $G$ be a graph and let $L$ be a list-assignment for $V(G)$. Let $P:=p_1p_2p_3$ be a path of length two in $G$. For each $(c, c')\in L(p_1)\times L(p_3)$, we let $\Lambda_{G,L}^P(c, \bullet, c')$ be the set of $d\in L(p_2)$ such that there is an $L$-coloring of $G$ which uses $c,d,c'$ on the respective vertices $p_1, p_2, p_3$. Likewise,  given a pair $(c, c')$ in either $L(p_1)\times L(p_2)$ or $L(p_2)\times L(p_3)$ respectively, we define the sets $\Lambda_{G,L}^P(c, c', \bullet)$  and $\Lambda_{G,L}^P(\bullet, c, c')$  analogously.}  \end{defn}

In the setting above, we have $\Lambda_{G,L}^P(c, c, \bullet)=\varnothing$ for any $c\in L(p_1)\cap L(p_2)$. Likewise, $\Lambda_{G,L}(\bullet, d, d)=\varnothing$ for any $d\in L(p_2)\cap L(p_3)$. The use of the notation above always requires us to specify an ordering of the vertices of a given 2-path. That is, whenever we write $\Lambda_{G,L}^P(\cdot, \cdot, \cdot)$, where two of the coordinates are colors of two of the vertices of $P$ and one is a bullet denoting the remaining uncolored vertex of $P$, we have specified beforehand which vertices the first, second, and third coordinates correspond to. Sometimes we make this explicit by writing $\Lambda^{p_1p_2p_3}_{G,L}(\cdot, \cdot, \cdot)$. Whenever any of $P, G, L$ are clear from the context, we drop the respective super- or subscripts from the notation above. 

\begin{defn}\label{GUniversalDefinition} \emph{Let $G$ be a graph and let $P:= p_1p_2p_3$ be a subpath of $G$ of length two. Let $L$ be a list-assignment for $V(G)$. Given an $a\in L(p_3)$,}
\begin{enumerate}[label=\emph{\alph*)}] 
\itemsep-0.1em
\item\emph{we say that $a$ is \emph{$(G, P)$-universal} if, for each $b\in L(p_2)\setminus\{a\}$, we have $\Lambda_G^P(\bullet, b, a)=L(p_1)\setminus\{b\}$.}
\item\emph{We say that $a$ is \emph{almost $(G,P)$-universal}, if, for each $b\in L(p_2)\setminus\{a\}$, we have $|\Lambda_G^P(\bullet, b, a)|\geq |L(p_1)|-1$.}
\end{enumerate}
 \end{defn}

In the setting above, if $a$ is $(G,P)$-universal, then it is clearly also almost $(G,P)$-universal. Furthermore, if $p_1p_3\in E(G)$ and $L(p_3)\subseteq L(p_1)$, then there is no $(G, P)$-universal color in $L(p_3)$. That is, $a$ being a $(G, P)$-universal color of $L(p_3)$ is a stronger property than the property that any $L$-coloring of $V(P)$ using $a$ on $p_3$ extends to an $L$-coloring of $G$, unless either $a\not\in L(p_1)$ or $p_1p_3\not\in E(G)$. Iff the 2-path $P$ is clear from the context, then, given an $a\in L(p_3)$, we just say that $a$ is $G$-universal.  Our first and second black boxes from Paper I are the following results.

\begin{theorem}\label{BWheelMainRevListThm2} Let $G$ be a broken wheel with principal path $P=pp'p''$ and let $L$ be a list-assignment for $V(G)$ in which each vertex of $V(G)\setminus\{p, p'\}$ has a list of size at least three. Let $G-p'=pu_1\ldots u_tp''$ for some $t\geq 0$. 
\begin{enumerate}[label=\arabic*)]
\item Let $\phi_0, \phi_1$ be a pair of distinct $L$-colorings of $pp'$. For each $i=0,1$, let $S_i:=\Lambda_G(\phi_i(p), \phi_i(p'), \bullet)$, and suppose that $|S_0|=|S_1|=1$. Then the following hold. 
\begin{enumerate}[label=\alph*)]
\itemsep-0.1em
\item If $\phi_0(p)=\phi_1(p)$ and $S_0=S_1$, then $|E(G-p')|$ is even; AND
\item If $\phi_0(p)=\phi_1(p)$ and $S_0\neq S_1$, then $|E(G-p')|$ is odd and, for each $i=0,1$, $S_i=\{\phi_{1-i}(p')\}$; AND
\item If $\phi_0(p)\neq\phi_1(p)$ and $S_0=S_1$, then $|E(G-p')|$ is odd and $(\phi_0(p), \phi_0(p'))=(\phi_1(p'), \phi_1(p))$
\end{enumerate}
\item Let $\mathcal{F}$ be a family of $L$-colorings of $pp'$ and let $q\in\{p, p'\}$. Suppose that $|\mathcal{F}|\geq 3$ and $\mathcal{F}$ is constant on $q$. Then there exists a $\phi\in\mathcal{F}$ such that $|\Lambda_G(\phi(p), \phi(p'), \bullet)|\geq 2$ and in particular, if $G$ is not a triangle, then $L(p'')\setminus\{\phi(p')\}\subseteq\Lambda_G(\phi(p), \phi(p'), \bullet)$.
\item If $|V(G)|>4$ and there is an $a\in L(p)$ with $L(u_1)\setminus\{a\}\not\subseteq L(u_2)$, then $a$ is $G$-universal. 
\item If $|V(G)|\geq 4$, then, letting $x$ be the unique vertex of distance two from $p$ on the path $G-p'$, the following holds: For any $a\in L(p)$ with $L(u_1)\setminus\{a\}\not\subseteq L(x)$, $a$ is almost $G$-universal. 
\end{enumerate}
 \end{theorem}

\begin{lemma}\label{PartialPathColoringExtCL0}
Let $(G, C, P, L)$ be a rainbow and let $\phi$ be a partial $L$-coloring of $V(P)$ which includes the endpoints of $P$ in its domain and does not extend to an $L$-coloring of $G$. Then at least one of the following holds.
\begin{enumerate}[label=\arabic*)]
\itemsep-0.1em
\item There is a chord of $C$ with one endpoint in $\textnormal{dom}(\phi)$ and the other endpoint in $C\setminus P$; OR
\item There is a $v\in V(G\setminus C)\cup (V(\mathring{P})\setminus\textnormal{dom}(\phi))$ with $|L_{\phi}(v)|\leq 2$.
\end{enumerate} 
\end{lemma}

Our third black box from Paper I is the following result and its corollary.  

\begin{theorem}\label{EitherBWheelOrAtMostOneColThm} Let $(G, C, P, L)$ be a rainbow, where $P=p_1p_2p_3$ is a 2-path. Suppose further that $G$ is short-separation-free and every chord of $C$ has $p_2$ as an endpoint. Then,
\begin{enumerate}[label=\arabic*)]
\itemsep-0.1em
\item either $G$ is a broken wheel with principal path $P$, or there is at most one $L$-coloring of $V(P)$ which does not extend to an $L$-coloring of $G$; AND
\item If $\phi$ is an $L$-coloring of $V(P)$ which does not extend to an $L$-coloring of $G$, then, for each $p\in\{p_1, p_3\}$, letting $u$ be the unique neighbor of $p$ on the path $C-p_2$, we have $L(u)\setminus\{\phi(p)\}|=2$.
\end{enumerate}
\end{theorem}

\begin{cor}\label{CorMainEitherBWheelAtM1ColCor} Let $(G, C, P, L)$ be a rainbow, where $P:=p_1p_2p_3$ is a 2-path. Suppose further that $G$ is short-separation-free and every chord of $C$ has $p_2$ as an endpoint. Then all three of the following hold.
\begin{enumerate}[label=\arabic*)]
\itemsep-0.1em
\item If $|V(C)|>4$ and $|L(p_3)|\geq 2$, then, letting $xyp_3$ be the unique 2-path of $C-p_2$ with endpoint $p_3$, either there is a $G$-universal $a\in L(p_3)$ or $G$ is a broken wheel with principal path $P$ such that $L(p_3)\subseteq L(x)\cap L(y)$; AND
\item  If $|L(p_3)|\geq 3$ and either $|V(C)|>3$ or $G=C$, then
\begin{enumerate}[label=\roman*)]
\itemsep-0.1em
\item Either $G$ is a broken wheel with principal path $P$, or $|\Lambda_G^P(\phi(p_1), \phi(p_2), \bullet)|>1$ for any $L$-coloring $\phi$ of $p_1p_2$; AND
\item For each $a\in L(p_1)$, there are at most two $b\in L(p_2)\setminus\{a\}$ such that $|\Lambda_G^P(a, b, \bullet)|=1$; AND
\end{enumerate}
\item If $\phi$ is an $L$-coloring of $\{p_1, p_3\}$ and $S\subseteq L_{\phi}(p_2)$ with $|S|\geq 2$ and $S\cap\Lambda_G^P(\phi(p_1), \bullet, \phi(p_3))=\varnothing$, then
\begin{enumerate}[label=\roman*)]
\itemsep-0.1em
\item $|S|=2$ and $G$ is a broken wheel with principal path $P$, and $G-p_2$ is a path of odd length; AND
\item For any $L$-coloring $\psi$ of $V(P)$ which does not extend to an $L$-coloring of $G$, either $\psi(p_1)=\psi(p_3)=s$ for some $s\in S$, or $\psi$ restricts to the same coloring of $\{p_1, p_3\}$ as $\phi$. 
\end{enumerate}
\end{enumerate} \end{cor}

We now recall the following notation from Paper I. 

\begin{defn}\label{EndNotationColor} \emph{Let $G$ be a graph and let $L$ be a list-assignment for $G$. Let $P$ be a path in $G$ with $|V(P)|\geq 3$, let $H$ be a subgraph of $G$ and let $\{p, p'\}$ be the endpoints of $P$. We let $\textnormal{End}_L(H, P, G)$ be the set of $L$-colorings $\phi$ of $\{p, p'\}\cup V(H)$ such that $\phi$ is $(P,G)$-sufficient.} \end{defn}

We usually drop the subscript $L$ in the case where it is clear from the context. Furthemore, if $H=\varnothing$, then we just write $\textnormal{End}_L(P,G)$. Our fourth black box from Paper I is the following result, together with its corollary. 

\begin{theorem}\label{SumTo4For2PathColorEnds} Let $(G, C, P, L)$ be an end-linked rainbow, where $P$ is a 2-path. Then $\textnormal{End}_{L}(P,G)\neq\varnothing$.  \end{theorem}

\begin{cor}\label{GlueAugFromKHCor} Let $(G, C, P, L)$ be an end-linked rainbow and $pq$ be a terminal edge of $P$, where $p$ is an endpoint of $P$, and let $p'$ be the other endpoint of $P$. Let $qx$ be an edge with $x\not\in V(C-p')$. Let $H, K$ be subgraphs of $G$ bounded by respective outer faces $x(C\setminus\mathring{P})p'Pq$ and $p(C\setminus\mathring{P})xq$. Let $\mathcal{F}$ be a nonempty family of partial $L$-colorings of $H$, where each element of $\mathcal{F}$ has $x$ in its domain. Suppose further that either $x=p$ or $|\textnormal{Col}(\mathcal{F}\mid x)|\geq |L(p')|$. Then there is an $L$-coloring $\phi$ of $\{p, x\}$ and a $\psi\in\mathcal{F}$ such that $\phi(x)=\psi(x)$ and $\phi$ is $(q, K)$-sufficient. \end{cor}

Our fifth and sixth black boxes from Paper I are the following results. 

\begin{theorem}\label{ThmFirstLink3PathForUseInHolepunch} Let $(G, C, P, L)$ be an end-linked rainbow, where $P:=p_1p_2p_3p_4$ is a 3-path. Then
\begin{enumerate}[label=\arabic*)] 
\itemsep-0.1em
\item $\textnormal{Crown}_L(P, G)\neq\varnothing$. Actually, something stronger holds. There is a subgraph $H$ of $G$ with $V(H)\subseteq V(C\setminus P)\cap N(p_1)\cup N(p_4))$, where $|V(H)\cap N(p)|\leq 1$ for each endpoint $p$ of $P$ and $\textnormal{End}_L(H,P,G)\neq\varnothing$; AND
\item If there is no chord of $C$ incident to a vertex of $\mathring{P}$, then $\textnormal{End}_{L}(P, G)\neq\varnothing$. 
\end{enumerate}
 \end{theorem}

\begin{theorem}\label{3ChordVersionMainThm1} Let $(G, C, P, L)$ be an end-linked rainbow, where $P:=p_1p_2p_3p_4$ is a 3-path, and the following additional conditions are satisfied.
\begin{enumerate}[label=\arabic*)]
\itemsep-0.1em
\item $|L(p_1)|\geq 1$ and $|L(p_4)|\geq 3$; AND
\item $N(p_3)\cap V(C)=\{p_2, p_4\}$. 
\end{enumerate} 
Then $\textnormal{End}_{L}(P, G)\neq\varnothing$. \end{theorem}

\section{Corner Colorings}\label{CorColSecRes}

In Theorem \ref{3ChordVersionMainThm1} in Paper I, we proved a result about coloring the endpoints of a 3-path in such a way that any extension of this precoloring to the rest of the 3-path also extends to the rest of the graph. That is, we proved an analogue to Theorem \ref{SumTo4For2PathColorEnds}, but for paths of length three rather than two. The conditions on our path $P:=p_1p_2p_3p_4$ in the statement of Theorem \ref{3ChordVersionMainThm1} require that $p_3$ is incident to no chord of $C$. Theorem \ref{3ChordVersionMainThm1} is not true if we drop this condition, and it is also not true if we weaken the condition on the lists of $p_1, p_4$ so that they only satisfy the condition that $L(p_1), L(p_4)$ are nonempty sets with $|L(p_1)|+|L(p_4)|\geq 4$. However, with the conditions on $G, C, P, L$ weakened in the way described above, we can obtain a useful result that is similar to Theorem \ref{3ChordVersionMainThm1} but requires us to precolor more vertices than just the endpoints of $P$. We prove such a result in this section This result is Theorem \ref{CornerColoringMainRes}. Note that, in the statement of Theorem \ref{3ChordVersionMainThm1} if $|L(p_3)|=5$, then the result of Theorem \ref{3ChordVersionMainThm1} states that there is an $L$-coloring $\phi$ of $\{p_1, p_4\}$ which extends to four (or rather, $|L(p_3)\setminus\{\phi(p_3)\}|$, so possibly five) different elements of $\textnormal{End}_{L}(p_3, P, G)$. If we weaken drop either of Conditions 1) or 2) of Theorem \ref{3ChordVersionMainThm1}, this is not true, but it is ``almost" true in a way which is make precise by the result below, whose proof makes up the remainder of Section \ref{CorColSecRes}.

\begin{theorem}\label{CornerColoringMainRes} Let $(G, C, P, L)$ be an end-linked rainbow, where $p_1p_2p_3p_4$ be a subpath of $C$ of length three. Suppose further that $|L(p_3)|\geq 5$. Then
\begin{enumerate}[label=\arabic*)] 
\item there is an $L$-coloring $\psi$ of $\{p_1, p_4\}$ which extends to $|L_{\psi}(p_3)|-2$ elements of $\textnormal{End}_{L}(p_3, P, G)$; AND
\item If $p_3$ is incident to no chord of $C$, then there is an $L$-coloring $\psi$ of $\{p_1, p_4\}$ which extends to $|L_{\psi}(p_3)|-1$ elements of $\textnormal{End}_{L}(p_3, P, G)$.
\end{enumerate}
\end{theorem}

\begin{proof} Suppose not and let $G$ be a vertex-minimal counterexample to the theorem. Let $C, P, L$ be as above, where $P:=p_1p_2p_3p_4$. By adding edges to $G$ if necessary, we suppose further that every face of $G$, except possibly $C$, is bounded by a triangle .By removing some colors from some lists if necessary, we suppose further that $|L(p_1)|+|L(p_4)|=4$, an furthermore, we suppose that, for each $v\in V(C\setminus P )$, $|L(v)|=3$. We break the proof of Theorem \ref{CornerColoringMainRes} into four subsections, which are organized as follows.
\begin{enumerate}[label=\arabic*)] 
\itemsep-0.1em
\item In Subsection \ref{PrelimRestricCornerColMain}, we gather a few preliminary facts about $G$. In particular, we show that there is a chord of $C$ incident to $p_2$. 
\item In Subsection \ref{DealWithComNbrSubSecCorner}, we show that $p_2$ and $p_4$ have no common neighbor in $G\setminus C$. In Subsections \ref{OuterplanarPartISub}-\ref{GOuterPlanarFinalSubSec}, we show that $p_2, p_3$ have a common neighbor in $C\setminus P$ and that $G$ is outerplanar. 
\item In Subsection \ref{GOuterPlanarFinalSubSec}, we complete the proof of Theorem \ref{CornerColoringMainRes}.  
\end{enumerate}

\makeatletter
\renewcommand{\thetheorem}{\thesubsection.\arabic{theorem}}
\@addtoreset{theorem}{subsection}
\makeatother

\subsection{Preliminary restrictions}\label{PrelimRestricCornerColMain}

Applying Theorem \ref{thomassen5ChooseThm} and Corollary \ref{CycleLen4CorToThom}, we immediately have the following by the minimality of $|V(G)|$. 

\begin{Claim}\label{MidPathPChordCCl01} $G$ is short-separation-free, and furthermore, every chord of $C$ is incident to one of $\{p_2, p_3\}$. \end{Claim}

We also have the following.

\begin{Claim}\label{CVertexSetNotPSizeAtLeast5} $V(C)\neq V(P)$. In particular, $p_1p_4\not\in E(G)$. \end{Claim}

\begin{claimproof} Suppose that $V(P)=V(C)$. By Corollary \ref{CycleLen4CorToThom}, every $L$-coloring of $V(P)$ extends to an $L$-coloring of $G$. In particular, for any $L$-coloring of $\{p_1, p_4\}$, $\phi$ extends to $|L_{\phi}(p_3)|$ different elements of $\textnormal{End}(p_3, P, G)$, contradicting the fact that $G$ is a counterexample. Thus, $V(P)\neq V(C)$, so $|V(C)|>4$. Since every chord of $C$ is incident to a vertex of $\mathring{P}$, we have $p_1p_4\not\in E(G)$. \end{claimproof}

Since $V(P)\neq V(C)$, we let $C\setminus\mathring{P}:=p_1u_1\ldots u_tp_4$ for some $t\geq 1$. 

\begin{Claim}\label{EvChordCEndP2P3OtherEndC-PClaim} $p_2p_4, p_1p_3\not\in E(G)$. In particular, every chord of $C$ has one endpoint in $\{p_2, p_3\}$ and the other endpoint in $\{u_1, \ldots, u_t\}$. \end{Claim}

\begin{claimproof} Suppose toward a contradiction that $E(G)$ contains at least one (and thus, precisely one) of $p_2p_4, p_1p_3$. Thus, there is a $q\in\{p_2, p_3\}$ such that $G$ contains the 2-path $p_1qp_4$. Let $q'$ be the other vertex of $\{p_2, p_4\}$. Since $G$ is short-separation-free, $G-q'$ is bounded by outer face $p_1qp_4u_t\ldots u_1$. Since $|L(p_1)|+|L(p_4)|\geq 4$, it follows from Theorem \ref{SumTo4For2PathColorEnds} that there is an $L$-coloring $\psi$ of $\{p_1, p_4\}$ which is $(q, G-q')$-sufficient. Since $N(q')\subseteq V(P)$, it follows that $\psi$ is $(P, G)$-sufficient. In particular, $\psi$ extends to $|L_{\psi}(p_3)|$ different elements of $\textnormal{End}(p_3, P, G)$, so we contradict the fact that $G$ is a counterexample. Thus, $p_2p_4, p_1p_3\not\in E(G)$, so it follows from Claim \ref{MidPathPChordCCl01} that every chord of $C$ has one endpoint in $\{p_2, p_3\}$ and the other endpoint in $\{u_1, \ldots, u_t\}$. \end{claimproof}

We now introduce the following notation.

\begin{defn} \emph{For each $c\in L(p_1)$ and $d\in L(p_4)$, we let $\mathcal{B}_{cd}$ be the set of  $L$-colorings of $V(P)$ which do not extend to $L$-color $G$.} \end{defn}

By Claim \ref{CVertexSetNotPSizeAtLeast5}, $p_1p_4\not\in E(G)$, so, for any $c\in L(p_1)$ and $d\in L(p_4)$, there is an $L$-coloring of $\{p_1, p_4\}$ using $c,d$ on the respective vertices $p_1, p_4$. Applying the fact that $p_1p_3\not\in E(G)$ by Claim \ref{EvChordCEndP2P3OtherEndC-PClaim}, together with the fact that $G$ is a counterexample to Theorem \ref{CornerColoringMainRes}, we immediately have the following.

\begin{Claim}\label{ObvObstructionColorSetS} For any $L$-coloring $\sigma$ of $\{p_1, p_4\}$, $L_{\sigma}(p_3)=L(p_3)\setminus\{\sigma(p_4)\}$. Furthermore, for any $c\in L(p_1)$ and $d\in L(p_4)$, we have $|\textnormal{Col}(\mathcal{B}_{cd}\mid p_3)|\geq 2$, and, if $p_3$ is incident to a chord of $C$, then $|\textnormal{Col}(\mathcal{B}_{cd}\mid p_3)|\geq 3$. \end{Claim}

We now have the following.

\begin{Claim}\label{p2isIncidenttoAtLeastOneChordCorner} $p_2$ is incident to a chord of $C$. \end{Claim}

\begin{claimproof} Suppose not. By 1) of Theorem \ref{ThmFirstLink3PathForUseInHolepunch}, there is a partial $L$-coloring $\psi$ of $V(C\setminus\{p_2, p_3\})$, where $p_1, p_4\in\textnormal{dom}(\psi)$ and $\psi$ is $(P, G)$-sufficient, and furthermore, $|N(p_3)\cap\textnormal{dom}(\psi)|\leq 2$. In particular, $|L_{\psi}(p_3)|\geq |L(p_3)|-2$. Let $c=\psi(p_1)$ and $d=\psi(p_4)$. Since $p_3$ is adjacent to at most one vertex of $\textnormal{dom}(\psi)\setminus\{p_4\}$, it follows that $L_{\psi}(p_4)$ consists of all but at most one color of $L(p_4)\setminus\{d\}$. Thus, by Claim \ref{ObvObstructionColorSetS}, there is an $\tau\in\mathcal{B}_{cd}$ with $\tau(p_3)\in L_{\psi}(p_4)$. Since $p_2$ is incident to no chords of $C$, the union $\psi\cup\tau$ is a proper $L$-coloring of its domain. It follows from our choice of $\psi$ that $\psi\cup\tau$ extends to $L$-color $G$, contradicting the fact that $\tau\in\mathcal{B}_{cd}$. \end{claimproof}

By Claim \ref{EvChordCEndP2P3OtherEndC-PClaim}, $p_2p_4\not\in E(G)$, so, by Claim \ref{p2isIncidenttoAtLeastOneChordCorner}, $p_2$ has a neighbor in $\{u_1, \ldots, u_t\}$. We now define the following.

\begin{defn}\label{umMaxIndexKTSubgraphDefnCorner}
\textcolor{white}{aaaaaaaaaaaaaaaaaaaaaaaaaaaaaaaaaaaaaaa}
\begin{enumerate}[label=\emph{\arabic*)}]
\itemsep-0.1em
\item\emph{Let $u_m$ be the neighbor of $p_2$ of maximal index on the path $u_1\ldots u_t$.}
\item\emph{Let $K$ be the subgraph of $G$ bounded by outer cycle $p_1u_1\ldots u_mp_2$ and let $T$ be the 2-path $p_1p_2u_m$.}
\end{enumerate}
\end{defn}

We now have the following. 

\begin{Claim}\label{IfNotTriThenIntersecBdCL} At least one of the following holds.
\begin{enumerate}[label=\arabic*)]
\itemsep-0.1em
\item $K$ is not a broken wheel with principal path $T$; OR
\item $K$ is a triangle; OR
\item $L(p_1)\not\subseteq L(u_2)$.  
\end{enumerate}
\end{Claim}

\begin{claimproof} Suppose that that $K$ is a broken wheel with principal path $T$, but $K$ is not a triangle. Suppose toward a contradiction that $L(p_1)\subseteq L(u_2)$. Let $G^{\dagger}:=G\setminus\{p_1, u_1\}$ and let $P^{\dagger}:=u_2p_2p_3p_4$. Since $K$ is not a triangle, $G^{\dagger}$ is bounded by outer cycle $C^{G^{\dagger}}=u_2p_2p_3p_4u_t\ldots u_2$. Thus, $|L(u_2)\cap L(p_1)|+|L(p_4)|\geq 4$ and $L(u_2)\cap L(p_1)$ is nonempty, as $L(p_1)$ is nonempty. Consider the following cases. 

\textbf{Case 1:} $p_3u_2\in E(G)$

In this case, since $G$ is short-separation-free, $G^{\dagger}-p_2$ is bounded by outer cycler $(u_2\ldots u_t)p_4p_3$, and this cycle contains the 2-path $u_2p_3p_4$. By Theorem \ref{SumTo4For2PathColorEnds}, there is an $L$-coloring $\phi$ of $\{u_2, p_4\}$ such that any extension of $\phi$ to an $L$-coloring of $\{u_2, p_3, p_4\}$ extends to $L$-color all of $G^{\dagger}-p_2$, where $\phi(u_2)\in L(u_2)\cap L(p_1)$. Let $\phi(u_2)=c$ and $\phi(p_4)=d$. By Claim \ref{ObvObstructionColorSetS}, since $c\in L(p_1)$, there is a $\sigma\in\mathcal{B}_{cd}$, where $\sigma(p_3)\neq c$. Thus, $\sigma\cup\phi$ is a proper $L$-coloring of its domain. Since $N(u_1)=\{p_1, p_2, u_2\}$ and two neighbors of $u_1$ are colored with the same color by $\sigma\cup\phi$,  it follows from our choice of $\phi$ that $\sigma\cup\phi$ extends to $L$-color $G$, so $\sigma$ extends to $L$-color $G$, which is false. 

\textbf{Case 2:} $p_3u_2\not\in E(G)$

Note that $G^{\dagger}$ has a chord of its outer face incident to $p_3$ if and only if $G$ has a chord of its outer face incident to $p_3$. Let $r$ be an integer, where $r=2$ if $G$ has a chord of $C$ incident to $p_3$ and otherwise $r=1$. By the minimality of $G$, there is an $L$-coloring $\phi$ of $\{u_2, p_4\}$ such that $\phi(u_2)\in L(u_2)\cap L(p_1)$, where $\phi$ extends to $|L_{\phi}(p_3)|-r$ different elements of $\textnormal{End}(p_3, P^{\dagger}, G^{\dagger})$. Let $k=|L_{\phi}(p_3)|-r$ and let $\phi^1, \ldots, \phi^k$ be $k$ distinct extensions of $\phi$ to elements of $\textnormal{End}(p_3, P^{\dagger}, G^{\dagger})$, each of which uses a different color on $p_3$.  

Let $c=\phi(u_2)$ and $d=\phi(p_4)$. Let $\sigma$ be the $L$-coloring of $\{p_1, p_4\}$ using $c,d$ on the respective vertices $p_1, p_4$. By Claim \ref{EvChordCEndP2P3OtherEndC-PClaim}, $p_1p_3\not\in E(G)$. Since $\sigma(p_4)=\phi(p_4)$ and $p_3u_2\not\in E(G)$, we have $L_{\sigma}(p_3)=L_{\phi}(p_3)$. Since $\phi(u_2)\in L(p_1)$, it follows that, for  $j=1, \ldots, k$, there is an $L$-coloring $\psi^j$ of $\{p_1, p_3, p_4\}$ using $c, \phi^j(p_3), d$ on the respective vertices $p_1, p_3, p_4$. Since $L_{\sigma}(p_3)=L_{\phi}(p_3)$ and $G$ is a counterexample to Theorem \ref{CornerColoringMainRes}, there is a $j\in\{1, \ldots,k\}$ such that $\psi^j\not\in\textnormal{End}(p_3, P, G)$. Thus, $\psi^j$ extends to an $L$-coloring $\psi^j_*$ of $V(P)$ which does not extend to an $L$-coloring of $G$. Since $\psi^j_*(p_2)\neq c$, $\psi^j_*$ extend to an $L$-coloring of $V(P)\cup\{u_2\}$ which uses $c$ on both of $p_1, u_2$. But then, since the same color is used on two neighbors of $u_1$, and $\phi^j\in\textnormal{End}(p_3, P^{\dagger}, G^{\dagger})$, it follows that $\psi^j_*$ extends to an $L$-coloring of $G$, a contradiction. We conclude that $L(p_1)\not\subseteq L(u_2)$.\end{claimproof}

\subsection{Dealing with a common neighbor of $p_2$ and $p_4$ in $G\setminus C$}\label{DealWithComNbrSubSecCorner}

This subsection consists of the following claim. 

\begin{Claim}\label{CornerCommonNbrCaseP2P4NotAdjw} $p_2, p_4$ have no common neighbor in $V(G\setminus C)$. \end{Claim}

\begin{claimproof} Suppose toward a contradiction that there is a $w\in V(G\setminus C)$ which is adjacent to each of $p_2, p_4$. Let $G^*:=G-p_3$. Since $p_2p_4\not\in E(G)$ and $G$ is short-separation-free, it follows from our triangulation conditions that $wp_3\in E(G)$ as well. In particular, $N(w)=\{p_2, p_3, p_4\}$, so no chord of $C$ is incident to $p_3$. Note that $H$ is bounded by outer cyle $C^H=u_m\ldots u_tp_4p_3p_2$ and $C^H$ contains the 3-path $P^H=u_mp_2p_3p_4$. Furthermore, $H$ contains no chords of its outer cycle which are incident to  $p_3$. Note that $P^H$ is an induced path of $G$, or else $u_mp_4\in E(G)$ and there is a 4-cycle separating $w$ from $p_1$. 

\vspace*{-8mm}
\begin{addmargin}[2em]{0em}
\begin{subclaim}\label{SetSObstructedLambda} 
At least one of the following holds.
\begin{enumerate}[label=\emph{\arabic*)}]
\itemsep-0.1em
\item $u_mw\in E(G)$; OR
\item For each $\pi\in\textnormal{Crown}(P^H, H)$ and each $c\in L(p_1)$, there is an $S\subseteq L(p_2)\setminus\{c\}$ such that $|S|\geq 2$ and $\Lambda_K^T(c, \bullet, \pi(u_m))\cap S=\varnothing$.  
\end{enumerate}
\end{subclaim}

\begin{claimproof} Suppose that $u_mw\not\in E(G)$ and suppose toward a contraditction that there is a $\pi\in\textnormal{Crown}(P^H, H)$ and a $c\in L(p_1)$ violating 2) of Subclaim \ref{SetSObstructedLambda}. Let $d=\pi(p_4)$. By Claim \ref{ObvObstructionColorSetS}, there exist $\tau_0, \tau_1\in\mathcal{B}_{cd}$, where $\tau_0(p_3)\neq\tau_1(p_3)$. Since $p_3$ is incident to no chords of $C$, $p_3$ has no neighbors in $C^H$ except for $p_2$ and $p_4$. It follows that, for each $i=0,1$, if $\tau_i\cup\pi$ is not a proper $L$-coloring of its domain in $G$, then either $\tau_i(p_2)=\pi(u_m)$ or $K$ is a triangle with $c=\pi(u_m)$. 

Let $S=\{\tau_0(p_2), \tau_1(p_2)\}$. If $S\cap\Lambda_K^T(c, \bullet, \pi(u_m)\neq\varnothing$, then, since $\pi\in\textnormal{Crown}(P^H, H)$, it follows that one of $\tau_0, \tau_1$ extends to an $L$-coloring of $G$, which is false. Since $c\not\in S$ and 2) of Subclaim \ref{SetSObstructedLambda} is violated, we have $|S|<2$, so $\tau_0(p_2)=\tau_1(p_2)$. Thus, we let $\tau_0(p_2)=\tau_1(p_2)=f$ for some color $f$. Note that $f\neq c$. We let $f^*\in\Lambda_K^T(c, f, \bullet)$. Let $\phi$ be the $L$-coloring of $\{u_m, p_2, p_4\}$ obtained by coloring $u_m, p_2, p_4$ with the respective colors $r^*, f, d$. This is a proper $L$coloring of its domain, since $P^H$ is an induced subgraph of $G$. Now, $f, d\not\in\{\tau_0(p_3), \tau_1(p_3)\}$. Thus, if $\phi$ extends to $L$-color $H-p_3$, then at least one color of $\{\tau_0(p_3), \tau_1(p_3)\}$ is distinct from the color used on $w$, and one of $\tau_0, \tau_1$ extends to an $L$-coloring of $G$, which is false. It follows that $\phi$ does not extend to $L$-color $H-p_3$. The outer cycle of $H-p_3$ contains the path $u_mp_2wp_4$, and no chord of $C^{H-p_3}$ is incident to any of $u_m, p_2, p_4$. Furthermore, no vertex of $V(H-p_3)\setminus V(C^{H-p_3})$ is adjacent to all three of $u_m, p_2, p_4$, or else $G$ contains a 4-cycle separating $w$ from $p_0$. Thus, it follows from Lemma \ref{PartialPathColoringExtCL0} that $|L_{\phi}(w)|\leq 2$, so $u_mw\in E(G)$, contradicting our assumption that $u_mw\not\in E(G)$. \end{claimproof}\end{addmargin}

Applying Subclaim \ref{SetSObstructedLambda}, we have the following

\vspace*{-8mm}
\begin{addmargin}[2em]{0em}
\begin{subclaim}\label{piStarObstructionForCrown} There is a $\pi^*\in\textnormal{Crown}(P, G)$ with $\textnormal{dom}(\pi^*)\cap N(u_2)=\{p_1, u_m\}$. Furthermore, $u_mw\in E(G)$. \end{subclaim}

\begin{claimproof} Since $1\leq |L(p_4)|\leq 3$ and $|L(u_m)|=3$, it follows from 1) of Theorem \ref{ThmFirstLink3PathForUseInHolepunch} that there exist $|L(p_4)|$ distinct elements of $\textnormal{Crown}(P^H, H)$, each of which uses a different color on $L(u_m)$. By Corollary \ref{GlueAugFromKHCor}, there is a $c\in L(p_1)$ and a $\pi\in\textnormal{Crown}(P^H, H)$ and a proper $L$-coloring $\phi$ of $\{p_1, p_4\}$ using $c, \pi(u_m)$ on the respective vertices $p_1, u_m$, where $\phi\in\textnormal{End}(T, K)$. Let $\pi^*=\phi\cup\pi$. Note that $\textnormal{dom}(\pi^*)\cap N(u_2)=\{p_1, u_m\}$, so $|L_{\pi^*}(p_2)|\geq 3$. Since $\phi\in\textnormal{End}(T, K)$ and $\pi\in\textnormal{Crown}(P^H, H)$, it follows that $\pi^*\in\textnormal{Crown}(P, G)$. Now we just need to show that $u_mw\in E(G)$. Suppose not. Since $\phi\in\textnormal{End}(T, K)$, we have $\Lambda_K^T(\phi(p_1), \bullet, \phi(u_m))=L(p_2)\setminus\{\phi(p_1), \phi(u_m)\}$, contradicting Subclaim \ref{SetSObstructedLambda}. \end{claimproof}\end{addmargin}

Let $J$ be the subgraph of $G$ bounded by outer cycle $u_m\ldots u_tp_4w$. Note that $J$ contains the 2-path $Q=u_mwp_4$. Let $\pi^*\in\textnormal{Crown}(P, G)$ be as in Subclaim \ref{piStarObstructionForCrown}. By Claim \ref{ObvObstructionColorSetS}, there exist two $L$-colorings $\pi_0, \pi_1$ of $V(P)$, neither of which extends to an $L$-coloring of $G$, where each of $\pi_0, \pi_1$ uses $\pi^*(p_1), \pi^*(p_4)$ on the respective vertices $p_1, p_4$, and $\pi_0(p_3)\neq\pi_1(p_4)$. For each $i=0,1$, the union $\pi^*\cup\pi_i$ is not a proper $L$-coloring of its domain, or else $\pi_i$ extends to an $L$-coloring of $G$. Let $f=\pi^*(u_m)$. Since $\pi^*\in\textnormal{Crown}(P, G)$ and $N(p_2)\cap\textnormal{dom}(\pi^*)=\{p_1, u_m\}$, it follows that $\pi_0(p_2)=\pi_1(p_2)=f$. Let $r\in\Lambda_K^T(c, f, \bullet)$. If $\Lambda_J^Q(r, \bullet, \pi^*(p_4))$ contains a color of $L(w)\setminus\{f\}$, then, since $\pi_0(p_3)\neq\pi_1(p_3)$, one of these colors is left for $p_3$ and one of $\pi_0, \pi_1$ extends to an $L$-coloring of $G$, which is false. Let $X=L(w)\setminus\{r, f, \pi^*(p_4)\}$. As indicated above, we have $\Lambda_J(r, \bullet, \pi^*(p_4))\cap X=\varnothing$. By 3) of Corollary \ref{CorMainEitherBWheelAtM1ColCor}, we immediately have the following.

\vspace*{-8mm}
\begin{addmargin}[2em]{0em}
\begin{subclaim}\label{PathLengthOddForJwSubH} $|X|=2$ and $J$ is a broken wheel with principal path $u_mwp_4$, where $J-w$ is a path of odd length. \end{subclaim}\end{addmargin}

We have an analogous result for the other side. 

\vspace*{-8mm}
\begin{addmargin}[2em]{0em}
\begin{subclaim}\label{DistinctColorA0A1FlipSubMa} $K$ is a broken wheel with principal path $T$, where $K-p_2$ has odd length. \end{subclaim}

\begin{claimproof} Let $G^*=G-p_3$. Note that $G^*$ is bounded by outer cycle $C^{G^*}=p_1p_2wp_4u_t\ldots u_1$, and this cycle contains the path $R=p_1p_2wp_4$. Since $|L(w)|\geq 5$, it follows from the minimality of $G$ that there exists an $L$-coloring $\phi$ of $\{p_1, p_4\}$ such that $\phi$ extends to two different elements of $\textnormal{End}(w, R, G^*)$. Let $\phi_0, \phi_1$ be two elements of $\textnormal{End}(w, R, G)$, where $\phi_0(w)\neq\phi_1(w)$ and each of $\phi_0, \phi_1$ restricts to $\phi$. For each $k=0,1$, let $a_k=\phi_k(w)$. 

By Claim \ref{ObvObstructionColorSetS}, there exist $\sigma_0, \sigma_1\in\mathcal{B}_{cd}$ of $V(P)$, where $\sigma_0(p_3)\neq\sigma_1(p_3)$. For each $k=0,1$, since $\sigma_k$ does not extend to $L$-color $G$, we have $\{a_0, a_1\}\cap L_{\sigma_k}(w)=\varnothing$, so each of $\sigma_0, \sigma_1$ colors the edge $p_2p_3$ with the colors $\{a_0, a_1\}$. Let $B=L(w)\setminus\{a_0, a_1, \phi(p_4)\}$. Note that, for each $k=0,1$ and $s\in\Lambda(\phi(p_1), a_k, \bullet)$, we have $\Lambda_J^Q(s, \bullet, \phi(p_4))\cap B=\varnothing$, or else there exists a $k=0,1$, such that the $\sigma_k$ extends to $L$-color $G$, which is false. Now let $B'=\Lambda(\phi(p_1), a_0, \bullet)\cup \Lambda(\phi(p_1), a_1, \bullet)$. 

To finish the proof of Subclaim \ref{DistinctColorA0A1FlipSubMa}, we first show that $|B'|=1$. Suppose not. Thus, $|B'|>1$. Let $b_0, b_1$ be distinct colors of $B$. Since $|B|\geq 2$, it follows that there exist distinct $r_0, r_1\in B'$, where, for each $i=0,1$, $r_i\neq b_i$, and the $L$-coloring $(r_i, b_i, \phi(p_4))$ of $u_mwp_4$ does not extend to an $L$-coloring of $J$. Recall that $|X|=2$ and  $\Lambda_J(r, \bullet, \pi^*(p_4))\cap X=\varnothing$. Since $r_0\neq r_1$ and at least one of $r_0, r_1$ is distinct from $\phi(p_4)$, we contradict 3) of Subclaim \ref{CorMainEitherBWheelAtM1ColCor}. Thus, $|B'|=1$. Since $a_0\neq a_1$, it follows from 1) of Theorem \ref{EitherBWheelOrAtMostOneColThm} that $K$ is a broken wheel with principal path $T$. Furthermore, by 1a) of Theorem \ref{BWheelMainRevListThm2}, $K-p_2$ is a path of even length. \end{claimproof}\end{addmargin}

\vspace*{-8mm}
\begin{addmargin}[2em]{0em}
\begin{subclaim}\label{ListP4P1Sizes1} $|L(p_1)|>1$. \end{subclaim}

\begin{claimproof} Suppose not. Thus, $|L(p_1)|=1$ and $|L(p_4)|=3$. Since no chord of $C$ is incident to $p_3$, it follows from Theorem \ref{3ChordVersionMainThm1} that there is a $\pi\in\textnormal{End}(P, G)$. Since $\psi$ is an $L$-coloring of $\{p_1, p_4\}$, we have $|L_{\pi}(p_3)|\geq 3$, and since any extension of $\pi$ to an $L$-coloring of $V(P)$ extends to $L$-color all of $G$, we contradict our assumption that $G$ is a counterexample to Theorem \ref{CornerColoringMainRes}. \end{claimproof}\end{addmargin}

Applying Subclaim \ref{ListP4P1Sizes1}, we have the following

\vspace*{-8mm}
\begin{addmargin}[2em]{0em}
\begin{subclaim}\label{SubForShowEachP2Map2Col} There is a $c'\in L(p_1)$ such that $|\Lambda_K^T(\phi(p_1), \phi(p_2), \bullet)|\geq 2$ for any $L$-coloring $\phi$ of $p_1p_2$ using $c'$ on $p_1$. \end{subclaim}

\begin{claimproof} By Subclaim \ref{DistinctColorA0A1FlipSubMa}, $K-p_2$ is a path of even length, so we have $|V(K)|\geq 4$. By Claim \ref{IfNotTriThenIntersecBdCL}, we have $L(p_1)\not\subseteq L(u_2)$. By Subclaim \ref{ListP4P1Sizes1}, $|L(p_1)|>1$. If there is a $K$-universal color of $L(p_1)$, then we are done, so suppose that no such color of $L(p_1)$ exists. Thus, $L(p_1)\subseteq L(u_1)$, and, since $K$ is outerplanar, it follows from 1) of Corollary \ref{CorMainEitherBWheelAtM1ColCor} that $|V(K)|\leq 4$, so $|V(K)|=4$. But since $|L(p_1)|>1$, there is a $c'\in L(p_1)$ with $L(u_1)\setminus\{c'\}\not\subseteq L(u_2)$. Thus, by 4) of Theorem \ref{BWheelMainRevListThm2}, we are done. \end{claimproof}\end{addmargin}

Let $c'\in L(p_1)$ be as in Subclaim \ref{SubForShowEachP2Map2Col}. Since $G$, there is an $L$-coloring $\tau$ of $V(P)$ which does not extend to an $L$-coloring of $G$, where $\tau(p_1)=c'$. Let $S'=L_{\tau'}(w)$. Note that $|S'|\geq 2$. By our choice of $c'$, we have $|\Lambda_K^T(c', \tau(p_2), \bullet)|\geq 2$. For each $s'\in S'$, we have $\Lambda_K^T(c', \tau(p_2), \bullet)\cap\Lambda_J^Q(\bullet, s', \tau(p_4))=\varnothing$, since $\tau$ does not extend to an $L$-coloring of $G$. Thus, there is a lone color $r'$ such that $\Lambda_J^Q(\bullet, s', \tau(p_4))=\{r'\}$ for each $s'\in S'$. Since $|S'|\geq 2$, it follows from 1a) of Theorem \ref{BWheelMainRevListThm2} that $J-w$ is a path of even length, contradicting Subclaim \ref{PathLengthOddForJwSubH}. This proves Claim \ref{CornerCommonNbrCaseP2P4NotAdjw}. \end{claimproof}

\subsection{Showing that $G$ is a outerplanar: part I}\label{OuterplanarPartISub}

Over the course of Subsections \ref{OuterplanarPartISub}-\ref{GOuterPlanarFinalSubSec}, we first show that $p_3$ is also adjacent to $u_m$, and then we show that $G-p_2p_3$ consists of a pair of broken wheels which intersect on the vertex $u_m$. In particular, we show that $G$ is outerplanar. We begin by introducing the following notation analogous to that of Definition \ref{umMaxIndexKTSubgraphDefnCorner}.

\begin{defn}
\textcolor{white}{aaaaaaaaaaaaaaaaaaaaaaaaaaaaaaaaaaaaaaa}
\begin{enumerate}[label=\emph{\arabic*)}]
\itemsep-0.1em
\item\emph{We let $u^{\dagger}$ be the unique neighbor of $p_3$ in $V(C\setminus\mathring{P})$ which is farthest from $p_4$ on the path $p_1u_1\ldots u_tp_4$. Possibly $u^{\dagger}=p_4$.}
\item\emph{We let $K'$ be the subgraph of $G$ with outer walk $(u^{\dagger}\ldots p_4)p_3$. That is, either $u^{\dagger}=p_4$ and $K'$ is an edge, or $K'$ is bounded by an outer cycle.}
\item\emph{We let $T'$ be the path $up_3p_4$ on the outer face of $K'$. That is, either $u=p_4$ and $T'$ is an edge, or $u^{\dagger}\neq p_4$ and $T'$ is the 2-path $u^{\dagger}p_3p_4$.}
\item\emph{Let $H$ be the subgraph of $G$ bounded by outer cycle $C^H=(u_m(C\setminus\mathring{P}) u^{\dagger})p_3p_2u_m$.} 
\end{enumerate}
\end{defn}

Note that $G=K\cup H\cup K'$, and the outer face of $H$ is an induced cycle of $G$.

\begin{Claim}\label{ForEachdSetSAtLeastTwoLambdaCL} If $|E(K')|>1$, then, for each $d\in L(p_4)$, there is an $S\subseteq L(p_3)\setminus\{d\}$  such that
\begin{enumerate}[label=\arabic*)] 
\itemsep-0.1em
\item $|S|\geq 2$ and $|\Lambda_{K'}^{T'}(\bullet, s, d)|\geq 2$ for each $s\in S$; AND
\item For each $a\in L(p_1)$, there is an $L$-coloring of $V(P)$ which does not extend to an $L$-coloring of $G$, where $\tau$ uses $a, d$ on the respective vertices $p_1, p_4$ and $\tau(p_3)\in S$. 
\end{enumerate}
 \end{Claim}

\begin{claimproof} Suppose that $|E(K')|>1$. Thus, $u^{\dagger}\neq p_4$ and $T'$ is a 2-path. In particular, there is a chord of $C$ incident to $p_3$. Fix a $d\in L(p_4)$ and consider the following cases.

\textbf{Case 1:} Either $d\not\in L(u_t)$ or $K'$ is not a broken wheel with principal path $T'$

In this case, we let $S$ be the set of $s\in L(p_3)\setminus\{d\}$ such that $|\Lambda_{K'}^{T'}(\bullet, s, d)|\geq 2$. If $d\not\in L(u_t)$, then, by 2) of Theorem \ref{EitherBWheelOrAtMostOneColThm}, any $L$-coloring of $V(T')$ using $d$ on $p_3$ extends to $L$-color $K'$, and thus, $S=L(p_3)\setminus\{d\}$. Note that this is true even if $K'$ is a triangle, since $d\not\in L(u_t)$. On the other hand, if $K'$ is not a broken wheel with principal path $T'$, then $p_1u^{\dagger}\not\in E(G)$ and it follows from 1) of Theorem \ref{EitherBWheelOrAtMostOneColThm} that $S$ consists of all but at most one color of $L(p_3)\setminus\{d\}$. In any case, we have $|S|\geq |L(p_3)\setminus\{d\}|-1$. Since $|L(p_3)|\geq 5$, we have $|S|\geq 2$. Now let $a\in L(p_1)$. Since $|S|\geq |L(p_3)\setminus\{d\}|-1$, it follows from Claim \ref{ObvObstructionColorSetS} that there is a $\tau\in\mathcal{B}_{ad}$ with $\tau(p_3)\in S$, so we are done in this case. 

\textbf{Case 2:} $K'$ is a broken wheel with principal path $T'$ and $d\in L(u_t)$

In this case, we let $S:=L(p_3)\setminus L(u_t)$. Since $|L(p_3)|\geq 5$, we have $|S|\geq 2$. We have $d\not\in S$, and, for each $s\in S$, we have $|\Lambda_{K'}^{T'}(\bullet, s, d)|\geq 2$. Now let $a\in L(p_1)$. We have $|L(u_t)|=3$, and since $d\in L(u_t)$, it follows that $|S|\geq |L(p_3)\setminus\{d\}|-2$. Since $p_3$ is incident to a chord of $C$ by assumption, it follows from Claim \ref{ObvObstructionColorSetS} that there is a $\tau\in\mathcal{B}_{ad}$ with $\tau(p_3)\in S$, so again, we are done. \end{claimproof}

\begin{Claim}\label{OneSideOrOtherSideMapstoOneColorClaim}
 For any $L$-coloring $\sigma$ of $V(P)$ which does not extend to an $L$-coloring of $G$, at least one of the following holds. 
\begin{enumerate}[label=\arabic*)]
\itemsep-0.1em
\item $|\Lambda_K^T(\sigma(p_1), \sigma(p_2), \bullet)|=1$; OR
\item $|E(K')|>1$ and $|\Lambda_{K'}^{T'}(\bullet, \sigma(p_3), \sigma(p_4))|=1$
\end{enumerate}
\end{Claim}

\begin{claimproof} Let $\sigma$ be an $L$-coloring of $V(P)$ which does not extend to an $L$-coloring of $G$ and suppose that $|\Lambda_K^T(\sigma(p_1), \sigma(p_2), \bullet)|=1$. Let $X:=\Lambda_K^T(\sigma(p_1), \sigma(p_2), \bullet)$. Consider the following cases.

\textbf{Case 1:} $|E(K')|=1$

In this case, we have $G=K\cup H$, and we just need to show that $|X|=1$. Suppose not. Thus, by Theorem \ref{thomassen5ChooseThm}, $|X|>1$. Let $L^*$ be a list-assignment for $V(H)$ in which each of $p_2, p_3, p_4$ is precolored by the respective colors $\sigma(p_2), \sigma(p_3), \sigma(p_4)$, and $L^*(u_m)=X\cup\{\sigma(p_2)\}$, and otherwise $L^*=L$. Any $L^*$-coloring of $H$ uses a color of $X$ on $u_m$. Since $\sigma$ does not extend to an $L$-coloring of $G$, it follows that $H$ is not $L^*$-colorable. Since $\sigma(p_2)\not\in X$, we have $|L^*(u_m)|\geq 3$. Since $H$ is not $L^*$-colorable, it follows from Lemma \ref{PartialPathColoringExtCL0} that the three vertices $p_2, p_3, p_4$ have a common neighbor in $H\setminus C^H$, and thus a common neighbor in $G\setminus C$, contradicting Claim \ref{CornerCommonNbrCaseP2P4NotAdjw}. 

\textbf{Case 2:} $|E(K')|\neq 1$

In this case, $u^{\dagger}\neq p_4$ and $|E(K')|>1$. Thus, we have $u^{\dagger}=u_n$ for some $n\in\{m, \ldots, t\}$. Let $X'=\Lambda_{K'}^{T'}(\bullet, \sigma(p_3), \sigma(p_4))$. We just need to show that at least one of $X, X'$ is a singleton. Suppose not. By Theorem \ref{thomassen5ChooseThm}, each of $X, X'$ is nonempty, so each of $X, X'$ has size at least two. Consider the following cases.

\textbf{Subcase 2.1:} $u^{\dagger}=u_m$

In this case, each of $X, X'$ is a subset of $L(u_m)$, and, since $G$ is short-separation-free, $H$ is a triangle, and $G=(K\cup K')+p_2p_3$. Since $m=n$ and $\sigma$ does not extend to an $L$-coloring of $G$, we have $X\cap X'=\varnothing$, which is false, since $|L(u_m)|=3$ and each of $X, X'$ has size at least two.

\textbf{Subcase 2.2:} $u^{\dagger}\neq u_m$

In this case, $n\in\{m+1, \ldots, t\}$. Since $\sigma(p_2)\not\in X$ and $\sigma(p_3)\not\in X$, it follows from Corollary \ref{2ListsNextToPrecEdgeCor} that the $L$-coloring $(\sigma(p_2), \sigma(p_3))$ of $p_2p_3$ extends to an $L$-coloring of $H$ which uses a color of the respective sets $X, X'$ on the respective vertices, $u_m, u_n$, so $\sigma$ extends to an $L$-coloring of $G$, a contradiction.    \end{claimproof}

Applying Claim \ref{OneSideOrOtherSideMapstoOneColorClaim}, we have the following.

\begin{Claim}\label{KBWheelCornerColorClaimM0}
\textcolor{white}{aaaaaaaaaaaaa}
\begin{enumerate}[label=\arabic*)] 
\itemsep-0.1em
\item $K$ is a broken wheel with principal path $T$, and $L(p_1)\subseteq L(u_1)$; AND
\item If $K$ is not a triangle, then $|E(K')|>1$ and $L(p_1)\not\subseteq L(u_2)$ and $|L(p_1)|=1$.
\end{enumerate}
\end{Claim}

\begin{claimproof} Since $G$ is a counterexample to Theorem \ref{CornerColoringMainRes}, it follows from Claim \ref{ForEachdSetSAtLeastTwoLambdaCL} that there is an $L$-coloring $\sigma$ of $V(P)$ which does not extend to an $L$-coloring of $G$, where either $|E(K')|=1$ or $|\Lambda_{K'}^{T'}(\bullet, \sigma(p_3), \sigma(p_4))|\geq 2$. By Claim \ref{OneSideOrOtherSideMapstoOneColorClaim}, we have $|\Lambda_K^T(\sigma(p_1), \sigma(p_2), \bullet)|=1$. It follows from 2) of Corollary \ref{CorMainEitherBWheelAtM1ColCor} applied to $K$ that $K$ is a broken wheel with principal path $T$. Now we show that $L(p_1)\subseteq L(u_1)$. Suppose not. Since $G$ is a counterexample to Theorem \ref{CornerColoringMainRes}, it again follows from Claim \ref{ForEachdSetSAtLeastTwoLambdaCL} that there an $L$-coloring $\phi$ of $V(P)$ which does not extend to an $L$-coloring of $G$, where $\phi(p_1)\not\in L(u_1)$ and either $|E(K')|=1$ or $|\Lambda_{K'}^{T'}(\bullet, \phi(p_3), \phi(p_4))|\geq 2$. Since $K$ is a broken wheel with principal path $T$ and $\phi(p_1)\not\in L(u_1)$, we have $|\Lambda_K(\phi(p_1), \phi(p_2), \bullet)|\geq 2$ (this is true even if $K$ is a triangle), contradicting Claim \ref{OneSideOrOtherSideMapstoOneColorClaim}. This proves 1). 

Now we prove 2). Suppose that $K$ is not a triangle. Since $K$ is a broken wheel with principal path $T$, it follows from Claim \ref{IfNotTriThenIntersecBdCL} that $|L(p_1)\not\subseteq L(u_2)$. Suppose toward a contradiction that $|L(p_1)|>1$. In that case, since $L(p_1)\subseteq L(u_1)$ and $L(p_1)\not\subseteq L(u_2)$, there is an $a\in L(p_1)$ with $L(u_1)\setminus\{a\}\not\subseteq L(u_2)$. Since $G$ is a counterexample to Theorem \ref{CornerColoringMainRes}, it follows from Claim \ref{ForEachdSetSAtLeastTwoLambdaCL} that there is an $L$-coloring $\phi$ of $V(P)$ which does not extend to an $L$-coloring of $G$, where $\phi(p_1)=a$ and either $|E(K')|=1$ or $|\Lambda_{K'}^{T'}(\bullet, \phi(p_3), \phi(p_4))|\geq 2$. By Claim \ref{OneSideOrOtherSideMapstoOneColorClaim}, we have $|\Lambda_K^T(\phi(p_1), \phi(p_2), \bullet)|=1$. Since $K$ is not a triangle, we contradict 4) of Theorem \ref{BWheelMainRevListThm2}.  \end{claimproof}

\begin{Claim}\label{nm+1DistAtL1SubCL} $u^{\dagger}\neq u_{m+1}$. \end{Claim}

\begin{claimproof} Suppose that $u^{\dagger}=u_{m+1}$. Thus, $p_2u\not\in E(G)$, and $u_mu^{\dagger}p_3p_2$ is a 4-cycle. Since $G$ is short-separation-free, it follows from our triangulation conditions that $p_3u_m\in E(G)$, so $u^{\dagger}=u_m$ and $p_2u\in E(G)$, a contradiction.  \end{claimproof}

We now show that, if $p_3u_m\not\in E(G)$, then there is a unique vertex of $G\setminus C$ adjacent to all three of $u_m, p_2, p_3$. Actually, we show something slightly stronger. 

\begin{Claim}\label{UniquewAdjacentAllThreeUmP2P3MainCL} Either $p_3u_m\in E(G)$ or there exists a $w\in V(H\setminus C^H)$ adjacent to all of $u_m, p_2, p_3$ such that, for any $L$-coloring $\sigma$ of $V(P)$ which does not extend to an $L$-coloring of $G$, and any $r\in\Lambda_K^T(\sigma(p_1), \sigma(p_2), \bullet)$, at least one of the following holds. 
\begin{enumerate}[label=\arabic*)]
\itemsep-0.1em
\item $|L(w)\setminus\{r, \sigma(p_2), \sigma(p_3)\}|=2$; OR
\item $|E(K')|>1$ and $|\Lambda_{K'}^{T'}(\bullet, \sigma(p_3), \sigma(p_4))|=1$, and and $w$ is adjacent to all four vertices of $Q$. 
\end{enumerate}
 \end{Claim}

\begin{claimproof} Suppose that $p_3u_m\not\in E(G)$. Thus, $u^{\dagger}\neq u_m$ and $G$ contains the 3-path $Q=u_mp_2p_3u^{\dagger}$. Any common neighbor to $u_m, p_2, p_3$ in $V(H\setminus C^H)$, if it exists, is the unique vertex of $H\setminus C^H$ with at least three neighbors on $Q$. Since $G$ is a counterexample to Theorem \ref{CornerColoringMainRes}, there is an $L$-coloring of $V(P)$ which does not extend to an $L$-coloring of $G$. Now suppose toward a contradiction that Claim \ref{UniquewAdjacentAllThreeUmP2P3MainCL} does not hold. Thus, there is an $L$-coloring $\sigma$ of $V(P)$ and an $r\in\Lambda_K^T(\sigma(p_1), \sigma(p_2), \bullet)$, where $\sigma$ does not extend to an $L$-coloring of $G$ and least one of the following holds.
\begin{enumerate}[label=\alph*)]
\itemsep-0.1em
\item $u_m, p_2, p_3$ have no common neighbor in $V(H\setminus C^H)$; \emph{OR}
\item $u_m, p_2, p_3$ have a common neighbor $w\in V(H\setminus C^H)$, but $|L(w)\setminus\{r, \sigma(p_2), \sigma(p_3)\}|\geq 3$, and furthermore, either $u\not\in N(w)$ or $|E(K')|=1$ or $T'$ is a 2-path with $|\Lambda_{K'}^{T'}(\bullet, \sigma(p_3), \sigma(p_4))|\geq 2$. 
\end{enumerate}

We now have the following.

\vspace*{-8mm}
\begin{addmargin}[2em]{0em}
\begin{subclaim}\label{SubClSigmaRPairK'>1} $|E(K')|>1$. \end{subclaim}

\begin{claimproof} Suppose not. Thus $G=K\cup H$, and $u^{\dagger}=p_4$, so $Q=u_mp_2p_3p_4$. We let $\sigma'$ be the $L$-coloring $(r, \sigma(p_2), \sigma(p_3), \sigma(p_4))$ of $u_mp_2p_3p_4$. Possibly $r=\sigma(p_4)$, but since $u_m\not\in N(p_3)$ by assumption, $\sigma'$ does not extend to an $L$-coloring of $H$. Since the outer face of $H$ has no chords, it follows from Lemma \ref{PartialPathColoringExtCL0} that there is a vertex $w\in V(H\setminus C^H)$ with at least three neighbors on $Q$, where $|L_{\sigma'}(w)|<3$. In particular, $w\in V(G\setminus C)$. We have $N(w)\cap V(Q)=\{u_m, p_2, p_3\}$, or else we contradict Claim \ref{CornerCommonNbrCaseP2P4NotAdjw}. Thus, we get $L(w)\setminus\{r, \sigma(p_2), \sigma(p_3)\}|=2$, contradicting our assumption that Claim \ref{UniquewAdjacentAllThreeUmP2P3MainCL} is violated. \end{claimproof}\end{addmargin}

Applying Subclaim \ref{SubClSigmaRPairK'>1}, $T'$ is a 2-path of $K'$. Now consider the following cases.

\textbf{Case 1:} $|\Lambda_{K'}^{T'}(\bullet, \sigma(p_3), \sigma(p_4))|\geq 2$. 

Let $L^*$ be a list-assignment for $V(H)$ where the vertices $u_m, p_2, p_3$ are precolored with the respective colors $r, \sigma(p_2), \sigma(p_3)$, and furthermore, $L^*(u^{\dagger})=\{\sigma(p_3)\}\cup\Lambda_{K'}^{T'}(\bullet, \sigma(p_3), \sigma(p_4))$, and otherwise $L^*=L$. Since $\sigma(p_3)\not\in\Lambda_{K'}^{T'}(\bullet, \sigma(p_3), \sigma(p_4))$, we have $|L^*(u^{\dagger})|\geq 3$. Note that any $L^*$-coloring of $V(H)$ uses a color of $\Lambda_{K'}^{T'}(\bullet, \sigma(p_3), \sigma(p_4))$ on $u^{\dagger}$, since $p_3$ is precolored by $\sigma(p_3)$. Since $\phi$ does not extend to an $L$-coloring of $G$, it follows that $H$ is not $L^*$-colorable. Possibly $r=\sigma(p_3)$, but since $u_mp_3\not\in E(G)$ by assumption, it follows that $u_mp_2p_3$ is $L^*$-colorable. Since $|L^*(u^{\dagger})|\geq 3$ and $C^H$ is an induced cycle of $G$, it follows from Lemma \ref{PartialPathColoringExtCL0} that there is a $w\in V(H\setminus C^H)$ adjacent to all three of $u_m, p_2, p_3$ and that $|L(w)\setminus\{r, \sigma(p_2), \sigma(p_3)\}|=2$, contradicting our assumption that the pair $\sigma, r$ is a counterexample to Claim \ref{UniquewAdjacentAllThreeUmP2P3MainCL}. 

\textbf{Case 2:} $|\Lambda_{K'}^{T'}(\bullet, \sigma(p_3), \sigma(p_4))|=1$.

In this case, we let $r'$ be the lone color of $|\Lambda_{K'}^{T'}(\bullet, \sigma(p_3), \sigma(p_4))|=1$ and let $\sigma'$ be the $L$-coloring $(r, \sigma(p_2), \sigma(p_3), r')$ of $u_mp_2p_3u^{\dagger}$. Possibly $r=r'$, but, by Claim \ref{nm+1DistAtL1SubCL}, $u^{\dagger}\neq u_{m+1}$, so $\sigma'$ is a proper $L$-coloring of its domain. Since $\sigma$ does not extend to an $L$-coloring of $G$, $\sigma'$ does not extend to an $L$-coloring of $H$. Since $C^H$ has no chords, it follows from Lemma \ref{PartialPathColoringExtCL0} that there is a $w\in V(H\setminus C^H)$ such that $|L_{\sigma'}(w)|<3$. In particular, $w$ has at least three neighbors on $Q$. If $u^{\dagger}\not\in N(w)$, then $N(w)\cap V(Q)=\{u_m, p_2, p_3\}$ and $L(w)\setminus\{r, \sigma(p_2), \sigma(p_3)\}|=2$, contradicting our assumption that Claim \ref{UniquewAdjacentAllThreeUmP2P3MainCL} is violated. Thus, we have $u^{\dagger}\in N(w)$. 

Since Claim \ref{UniquewAdjacentAllThreeUmP2P3MainCL} is violated, $w$ is not adjacent to all four vertices of $Q$, so $w$ has precisely three neighbors on $Q$. In particular, $w$ is at adjacent to at least one of $p_2, p_3$. If $w\in N(u_m)$, then, since $G$ is short-separation-free and there are no chords of $C^H$, it follows from our triangulation conditions that $w$ is adjacent to all four vertices of $Q$, which is false. Thus, $N(w)\cap V(Q)$ consists precisely of $p_2, p_3, u^{\dagger}$, and $w$ is the unique vertex of $G\setminus C$ with at least three neighbors on $Q$. Thus, there is no vertex of $G\setminus C$ adjacent to all three of $u_m, p_2, p_3$. By Claim \ref{ForEachdSetSAtLeastTwoLambdaCL}, since $|E(K')|>1$, there is an $L$-coloring $\tau$ of $V(P)$ which does not extend to an $L$-coloring of $G$, where $|\Lambda_{K'}^{T'}(\bullet, \tau(p_3), \tau(p_4))|\geq 2$. Since no vertex of $G\setminus C$ is adjacent to all three of $u_m, p_2, p_3$, it follows that, for any $s\in\Lambda_K^T(\tau(p_1), \tau(p_2), \bullet)$, the pair $\tau, s$ violates Claim \ref{UniquewAdjacentAllThreeUmP2P3MainCL} as well. But now we are back to Case 1 with the pair $\sigma, r$ replaced by the pair $\tau, s$. This completes the proof of Claim \ref{UniquewAdjacentAllThreeUmP2P3MainCL}. \end{claimproof}

Applying Claim \ref{UniquewAdjacentAllThreeUmP2P3MainCL}, we have the following. 

\begin{Claim}\label{Lp1StrictlyLarger1AndKTrianMnCL} Either $p_3u_m\in E(G)$ or, letting $w$ be the unique vertex of $H\setminus C^H$ adjacent to all of $u_m, p_2, p_3$, both of the following hold.
\begin{enumerate}[label=\arabic*)]
\itemsep-0.1em
\item $|L(p_1)|>1$ and $K$ is a triangle; AND
\item $L(u_1)\subseteq L(w)$.
\end{enumerate}
 \end{Claim}

\begin{claimproof} Suppose that $p_2u_m\not\in E(G)$. To prove 1), it suffices to show that $|L(p_1)|>1$. If that holds, then,  By 2) of Claim \ref{KBWheelCornerColorClaimM0}, $K$ is a triangle. Suppose toward a contradiction that $|L(p_1)|=1$. Since $|L(p_1)|+|L(p_4)|=4$, it follows that $|L(u^{\dagger})|=3$, even if $u^{\dagger}=p_4$ and $K'$ is an edge. Consider the graph $K\cup H$. Now, $K\cup H$ is bounded by outer cycle $p_1p_2p_3(u^{\dagger}\ldots u_1)$, and $K\cup H$ has no chord of its outer face which is incident to $p_3$. Furthermore, the outer face of $K\cup H$ contains the path $R=p_1p_2p_3u^{\dagger}$. Since $|L(u^{\dagger})|=3$, it follows from Theorem \ref{3ChordVersionMainThm1} that there is a $\psi\in\textnormal{End}(p_1p_2p_3u^{\dagger}, K\cup H)$. 

If $|E(K')|=1$ then $R=P$ and $G=K\cup H$. In particular, $\psi$ is an $L$-coloring of $\{p_1, p_4\}$, and any extension of $\psi$ to an $L$-coloring of $V(P)$ extends to $L$-color all of $G$, contradicting our assumption that $G$ is a counterexample to Theorem \ref{CornerColoringMainRes}. Thus, $|E(K')|>1$, and $T'$ is a 2-path. In particular, $p_3$ is incident to a chord of $C$. Since $|L(p_1)|=1$, we have $|L(p_4)|=3$. Thus, by Theorem \ref{SumTo4For2PathColorEnds}, there is a $\phi\in\textnormal{End}(T', K')$ with $\phi(u^{\dagger})=\psi(u^{\dagger})$, so $\psi\cup\phi$ is a proper $L$-coloring of its domain $\{p_1, u^{\dagger}, p_4\}$. Let $c,d,f$ be the respective colors used on $p_1, u^{\dagger}, p_4$ by $\psi\cup\phi$. It follows from Claim \ref{ObvObstructionColorSetS} that there is a $\sigma\in\mathcal{B}_{cf}$ with $\sigma(p_3)\in L(p_3)\setminus\{d, f\}$. Possibly $\sigma(p_2)=d$, but since $u^{\dagger}\neq u_m$ by assumption, $p_2\not\in N(u^{\dagger})$, the union $\psi\cup\phi\cup\sigma$ is a proper $L$-coloring of its domain $V(P)\cup\{u^{\dagger}\}$. But then it follows from our choice of colors $c, d, f$ that $\psi\cup\phi\cup\sigma$ extends to an $L$-coloring of $G$, and thus $\sigma$ extends to an $L$-coloring of $G$, a contradiction. This proves 1) of Claim \ref{Lp1StrictlyLarger1AndKTrianMnCL}. 

Now we prove 2). Suppose toward a contradiction that there is a $d\in L(u_1)\setminus L(w)$. By 1), $|L(p_1)|>1$, so there is a $c\in L(p_1)\setminus\{d\}$. Since $G$ is a counterexample to Theorem \ref{CornerColoringMainRes}, it follows from Claim \ref{ForEachdSetSAtLeastTwoLambdaCL} that there is an $L$-coloring $\tau$ of $V(P)$ which does not extend to an $L$-coloring of $G$, where $\tau(p_1)=c$ and either $|E(K')|=1$ or $|\Lambda_{K'}^{T'}(\bullet, \tau(p_3), \tau(p_4))|\geq 2$. By Claim \ref{OneSideOrOtherSideMapstoOneColorClaim}, $|\Lambda_K^T(c, \tau(p_2), \bullet)|=1$. Let $r$ be the lone color of $\Lambda_K^T(c, \tau(p_2), \bullet)$. Since $K$ is a triangle, we have either $r=d$ or $\tau(p_2)=d$. In either case, we get $|L(w)\setminus\{r, \tau(p_2), \tau(p_3)\}|\geq 3$, which contradicts Claim \ref{UniquewAdjacentAllThreeUmP2P3MainCL}.  \end{claimproof}

\subsection{Showing that $G$ is outerplanar: part II}\label{ShowingK'NotEdgeSubSec}

This subsection consists of the following result. 

\begin{Claim}\label{K'NotJustAnEdgeCorner} $|E(K')|>1$. \end{Claim}

\begin{claimproof} Suppose toward a contradiction that $|E(K')|=1$. In particular, $u^{\dagger}=p_4$ and $u_mp_3\not\in E(G)$, so $G=K\cup H$, and,  by Claim \ref{UniquewAdjacentAllThreeUmP2P3MainCL}, there is a $w\in V(H\setminus C^H)$ adjacent to $u_m, p_2, p_3$. Furthermore, $G$ contains the 3-path $Q=u_mp_2p_3p_4$. Since $p_3$ is incident to no chords of $C$, it follows from Claim \ref{ObvObstructionColorSetS} that, for any $c\in L(p_1)$ and $d\in L(p_4)$, $\textnormal{Col}(\mathcal{B}_{cd}\mid p_3)|\geq 2$. By Claim \ref{Lp1StrictlyLarger1AndKTrianMnCL}, $K$ is a triangle and $|L(p)|>1$. In particular, $m=1$. We now let $z$ be the unique neighbor of $w$ in $V(C\setminus\{p_2, p_3\}))$ which is defined as follows. If $N(w)\cap V(C)=\{u_1, p_3\}$, then $z=u_1$, and otherwise $z$ is the unique vertex of $N(w)\cap\{u_2, \ldots, u_t, p_4\}$ which is closest to $u_1$ on the path $u_2\ldots u_tp_4$. By Claim \ref{CornerCommonNbrCaseP2P4NotAdjw}, $z\neq p_4$. Let $J$ be the subgraph of $G$ with outer face $(u_1(C\setminus P)z)w$. If $z=u_1$ then $J$ is an edge, otherwise the outer face of $J$ is a cycle. Let $H^*$ be the subgraph of $G$ bounded by outer cycle $zwp_3p_4u_t\ldots z$. Note that, if $J$ is an edge, then, since $G$ is short-separation-free, we have $H^*=H-p_2$. 

\vspace*{-8mm}
\begin{addmargin}[2em]{0em}
\begin{subclaim}\label{subCLForColorBreakJ} Suppose that $|E(J)|>1$ and let $\tau$ be an $L$-coloring of $V(P)$ which does not extend to $L$-color $G$. Then, for any $\sigma\in\textnormal{Crown}(zwp_3p_4, H^*)$ with $\sigma(p_4)=tau(p_4)$, any $f\in L(u_1)\setminus\{\tau(p_1), \tau(p_2)\}$, and any $s\in L_{\sigma}(w)\setminus\{\tau(p_2), \tau(p_3)\}$, there is no $L$-coloring of $J$ which uses $f, s, \tau(z)$ on the respective vertices $u_1, w, z$. \end{subclaim}

\begin{claimproof} Suppose not. Thus, $(f, s, \tau(z))$ is a proper $L$-coloring of $u_1wz$ which extends to an $L$-coloring $\psi$ of $J$. By definition, $\sigma$ is a partial $L$-coloring of $V(C^{H^*})\setminus\{w, p_3\}$ with $z, p_4\in\textnormal{dom}(\sigma)$, and, since no chord of $C$ is incident to $p_3$, the union $\sigma\cup\tau$ is a proper $L$-coloring of its domain. Since $s\in L_{\sigma}(w)\setminus\{\tau(p_2), \tau(p_3)\}$ and $f\in L(u_1)\setminus\{\tau(p_1), \tau(p_2)\}$, the union $\tau\cup\sigma\cup\psi$ is a proper $L$-coloring of its domain, which is $V(K\cup J\cup P)$. Since $\sigma\in\textnormal{Crown}(zwp_3p_4, H^*)$ and $\tau\cup\sigma\cup\psi$ is a proper $L$-coloring of its domain, it follows that $(\sigma(z), s, \tau(p_3), \sigma(p_4))$ is a proper $L$-coloring of the 3-path $zwp_3p_4$ which extends to $L$-color all of $H^*$, so $\tau$ extends to $L$-color $G$< contradicting our assumption. \end{claimproof}\end{addmargin}

Applying Subclaim \ref{subCLForColorBreakJ}, we now show that $J$ is just an edge, i.e $z=u_1$. 

\vspace*{-8mm}
\begin{addmargin}[2em]{0em}
\begin{subclaim}\label{HminP2ChordNotIncident} No chord of the outer face of $H-p_2$ is incident to $w$. \end{subclaim}

\begin{claimproof} Suppose toward a contradiction that there is a chord of $C^{H-p_2}$ incident to $w$. Thus, we have $z=u_n$ for some $n\in\{2, \ldots, t\}$, and, by definition, since $p_4\not\in N(w)$, $z$ is the unique vertex of $N(w)\cap\{u_2, \ldots, u_t\}$ which is closest to $u_1$ on the path $u_2\ldots u_t$. Let $A:=\textnormal{Col}(\textnormal{Crown}(u_nwp_3p_4, H^*)\mid u_n)$. Since $|L(u_n)|=3$ and $1\leq |L(p_4)|\leq 3$, it follows from 1) of Theorem \ref{ThmFirstLink3PathForUseInHolepunch} that $|A|\geq |L(p_4)|$. In particular, $|A|+|L(p_1)|=4$. Now consider the following cases.

\textbf{Case 1:} $J$ is a triangle

In this case, we have $n=2$. Now, since $|A|+|L(p_1)|=4$, there is a $c\in L(p_1)$ and an $f\in L(u_2)$ such that $L(u_1)\setminus\{c, f\}|\geq 2$. Possibl;y $c=f$. In any case, since $f\in A$, there is a $\sigma\in\textnormal{Crown}(u_nwp_3p_4, H^*)$ using $f$ on $u_n$. By definition, $\sigma$ is a partial $L$-coloring of $C^{H^*}\setminus\{w, p_3\}$ which includes $u_n, p_4$ in its domain, and $|L_{\sigma}(w)|\geq 3$.

 Let $d=\sigma(p_4)$. Since $|L_{\sigma}(w)|\geq 3$, there is an $s\in L_{\sigma}(w)$ such that $L(u_1)\setminus\{c, f, s\}|\geq 2$. Since $J$ is a triangle and neither element of $\mathcal{B}_{cd}$ extends to an $L$-coloring of $G$, we have $s\in\{\tau(p_2), \tau(p_3)\}$ for each $\tau\in\mathcal{B}_{cd}$, or else we contradict Subclaim \ref{subCLForColorBreakJ}. Since $|\textnormal{Col}(\mathcal{B}_{cd}\mid p_3)|\geq 2$, there exists a $\tau\in\mathcal{B}_{cd}$ such that $\tau(p_2)=s$. Since $|L_{\sigma}(w)|\geq 3$, there is a color left in $L_{\sigma}(w)\setminus\{\tau(p_2), \tau(p_3)\}$. We color $w$ with this leftover color, and since $L(u_1)\setminus\{c, f, s\}|\geq 2$, there is also a color left for $u_1$, contradicting Subclaim \ref{subCLForColorBreakJ}. 

\textbf{Case 2:} $J$ is not a triangle

Since no chord of $C$ has both endpoints in $C\setminus\{p_2, p_3\}$, every chord of the outer face of $J$ is incident to $w$. Since $J$ is not a triangle, it follows from the minimality of $n$ that $J$ is not a broken wheel with principal path $u_1wu_n$, and, in particular, by 1) of Theorem \ref{EitherBWheelOrAtMostOneColThm}, there is at most one $L$-coloring of $\{u_1, w, u_n\}$ which does not extend to $L$-color $J$. Furthermore, $u_1u_n\not\in E(G)$, so $u_1wu_n$ is an induced path. We break Case 2 into two subcases.

\textbf{Subcase 2.1} $|L(p_1)|=2$

In this case, we have $|L(p_4)|=2$ as well. Recall that $|A|\geq |L(p_4)|$, so there is an $f\in A$ such that any $L$-coloring of $u_1wu_n$ using $f$ on $u_n$ extends to $L$-color all of $J$. That is, $f$ is a $J$-universal color of $L(u_n)$. Since $f\in A$, there is a $\sigma\in\textnormal{Crown}(u_nwp_3p_4, H^*)$ with $\sigma(u_n)=f$. Since $\sigma\in\textnormal{Crown}(u_nwp_3p_4, H^*)$, we have $|L_{\sigma}(w)|\geq 3$. Now we fix a $c\in L(p_1)$. Since $|L_{\sigma}(w)|\geq 3$, there is an $s\in L_{\sigma}(w)$ such that $L(u_1)\setminus\{c, s\}|\geq 2$. In particular, for each $\tau\in\mathcal{B}_{cd}$, we have $L(u_1)\setminus\{c, s, \tau(p_2)\}|\geq 1$. Since $\tau$ does not extend to an $L$-coloring of of $G$ and $f$ is a $J$-universal color of $L(u_n)$, it follows that $s\in\{\tau(p_2), \tau(p_3)\}$, or else we contradict Subclaim \ref{subCLForColorBreakJ}. Since $|\textnormal{Col}(\mathcal{B}_{cd}\mid p_3)|\geq 2$, there is a $\tau\in\mathcal{B}_{cd}$ such that $\tau(p_2)=s$. Now we choose an $s'\in L_{\sigma}(w)\setminus\{\tau(p_2), \tau(p_3)\}$, which exists, since $|L_{\sigma}(w)|\geq 3$. Since $L(u_1)\setminus\{c, s\}|\geq 2$, we color $w$ with $s'$ and $u_n$ with $f$ and obtain an $L$-coloring of $J$ using $f$ on $u_n$, $s'$ on $w$, and using a color of $L(u_1)\setminus\{c,s\}$ on $u_1$, contradicting Subclaim \ref{subCLForColorBreakJ}. 

\textbf{Subcase 2.2} $|L(p_1)|\neq 2$

In this case, since $|L(p_1)|>1$, we have $|L(p_1)|=3$, and $L(p_4)=\{d\}$ for some color $d$. Since $|L(u_n)|+|\{d\}|=4$, it follows from 1) of Theorem \ref{ThmFirstLink3PathForUseInHolepunch} that there is a $\sigma\in\textnormal{Crown}(u_nwp_3p_4, H^*)$. Since $|L(u_1)|=3$ and at most one $L$-coloring of $u_1wu_n$ does not extend to $L$-color $J$, there exists a set $X\subseteq L(u_1)$ with $|X|=2$, where, for each $a\in X$, any $L$-coloring of $u_1wu_n$ using $a$ on $u_1$ extends to $L$-color all of $J$. Since $|L(p_1)|=3$, there is a $c\in L(p_1)\setminus X$. Now we simply choose an arbitrary $k\in\{0,1\}$ and consider $\tau_{cd}^k$. Since $\sigma\in\textnormal{Crown}(u_nwp_3p_4, H^*)$, we have $|L_{\sigma}(w)|\geq 3$, so there is an $s\in L_{\sigma}(w)\setminus X$. 

If there exists a $\tau\in\mathcal{B}_{cd}$ such that $s\not\in\{\tau(p_2), \tau(p_3\}$, then $L(u_1)\setminus\{c, \tau(p_2), s\}$ contains at least one color of $X$, and since $u_1u_n\not\in E(G)$ and each color of $X$ is a $J$-universal color of $L(u_1)$, we contradict Subclaim \ref{subCLForColorBreakJ}. Thus, we have $s\in\{\tau(p_2), \tau(p_3)\}$ for each $\tau\in\mathcal{B}_{cd}$. Since $|\textnormal{Col}(\mathcal{B}_{cd}\mid p_3)|\geq 2$, there is a $\tau\in\mathcal{B}_{cd}$ such that $\tau(p_2)=s$. Now we choose an arbitrary $s'\in L_{\sigma}(w)\setminus\{\tau(p_2), \tau(p_3)\}$, which exists since $|L_{\sigma}(w)|\geq 3$, and since $p_2$ is colored with $s$, there is a color of $X$ left in $L(u_1)\setminus\{c, \tau(p_2), s'\}$. Possibly this color is also used by $\sigma$ on $u_n$, but since $J$ is not a triangle and $u_1u_n\not\in E(G)$, we contradict Subclaim \ref{subCLForColorBreakJ}. This completes the proof of Subclaim \ref{HminP2ChordNotIncident}. \end{claimproof}\end{addmargin}

Applying Subclaim \ref{HminP2ChordNotIncident} we have the following.

\vspace*{-8mm}
\begin{addmargin}[2em]{0em}
\begin{subclaim}\label{SecGenObstructionVStarVertK} There exists a $v^*\in V(H\setminus C^H)$ such that $v^*$ is adjacent to $w, p_3, p_4$. Furthermore, for any $L$-coloring $\tau$ of $V(P)$ which does not extend to $L$-color $G$, any $r\in L_{\tau}(u_1)$, and any $s\in L(w)\setminus\{\tau(p_2), \tau(p_3), r\}$, we have $L(v^*)\setminus\{s, \tau(p_3), \tau(p_4)\}|=2$. \end{subclaim}

\begin{claimproof} Since every chord of $C$ is incident to $p_2$, it follows from Subclaim \ref{HminP2ChordNotIncident} that the outer cycle of $H-p_2$ has no chords. Let $\tau$ be an $L$-coloring of $V(P)$ which does not extend to $L$-color $G$. Such a $\tau$ exists, since $G$ is a counterexample to \ref{CornerColoringMainRes}. Let $r\in L_{\tau}(u_1)$ and let $s\in L(w)\setminus\{\tau(p_2), \tau(p_3), r\}$, and let $\phi$ be the $L$-coloring $(r, s, \tau(p_3), \tau(p_4))$ of $u_1wp_3p_4$. Note that $\phi$ does not extend to $L$-color $H-p_2$. Since there are no chords of the outer cycle of $H-p_2$, it follows from Lemma \ref{PartialPathColoringExtCL0} that there is a vertex $v^*\in V(H-p_2)$, where $v^*$ is not on the outer face of $H-p_2$ and $L_{\phi}(v^*)|<3$. Thus, $v^*\in V(H\setminus C^H)$, and, since $G$ is short-separation-free, $v^*$ is adjacent to at most one of $u_1, p_3$. If all three of $u_1, w, p_3$ are adjacent to $v^*$, then, since $G$ is short-separation-free and there are no chords of $C^{H_*}$, it follows from our triangulation conditions that $v^*$ is also adjacent to $p_3$, which is false. Thus, we have $N(v^*)\cap\{u_1, w, p_3, p_4\}=\{w, p_3, p_4\}$ and $v^*$ is the unique vertex of $H\setminus C^H$ with more than two neighbors in $\{u_1, w, p_3, p_4\}$, and $L(v^*)\setminus\{s, \tau(p_3), \tau(p_4)\}|=2$. \end{claimproof}\end{addmargin}

Let $v^*$ be as in Subclaim \ref{SecGenObstructionVStarVertK}. We now have the following.

\vspace*{-8mm}
\begin{addmargin}[2em]{0em}
\begin{subclaim}\label{CrownColOnU1Blocked} For any $\pi\in\textnormal{Crown}(Q, H)$ and $c\in L(p_1)\setminus\{\pi(u_1)\}$, letting $d=\pi(p_4)$, we have $\tau(p_2)=\pi(u_1)$ for all $\tau\in\mathcal{B}_{cd}$. \end{subclaim}

\begin{claimproof} Suppose there is a $\tau\in\mathcal{B}_{cd}$ such that $\tau(p_2)\neq\pi(u_1)$. By definition, $\pi$ is a partial $L$-coloring of $C^H\setminus\{p_2, p_3\}$ with $u_1, p_4\in\textnormal{dom}(\pi)$. Since $p_2u_1$ is the only chord of $C$ incident to $p_2$, and, by assumption, no chord of $C$ is incident to $p_3$, it follows that $\pi\cup\tau$ is a proper $L$-coloring of its domain in $G$, which is $\textnormal{dom}(\pi)\cup\{p_1\}$. But then, since $K$ is a triangle and $\pi\in\textnormal{Crown}(Q, H)$, it follows that $\pi\cup\tau$ extends to an $L$-coloring of $G$, and thus $\tau$ extends to an $L$-coloring of $G$, which is false. \end{claimproof}\end{addmargin} 

\vspace*{-8mm}
\begin{addmargin}[2em]{0em}
\begin{subclaim}\label{EachPiNotSame4ListSubCM} For each $\pi\in\textnormal{Crown}(Q, H)$ with $\pi(u_1)\in L(p_1)$, there exists an $s\in L(w)\setminus\{\pi(u_1)\}$ with $L(v^*)\setminus\{\pi(p_4), s\}|\geq 4$ \end{subclaim}

\begin{claimproof} Let $d=\pi(p_4)$ and $f=\pi(u_1)$. By assumption, $f\in L(p_1)$. Suppose toward a contradiction that $\pi$ does not satisfy Subclaim \ref{EachPiNotSame4ListSubCM}. Thus, $L(w)\setminus\{f\}=L(v^*)\setminus\{d\}$ and $|L(w)\setminus\{f\}|=|L(v^*)\setminus\{d\}|=4$, so $f\in L(w)$ and $|L(v^*)\setminus\{d, f\}|\geq 4$. Since $f\in L(p_1)$, we consider the set $\mathcal{B}_{df}$. We have $|\textnormal{Col}(\mathcal{B}_{df}\mid p_3)|\geq 2$, so there is a $\tau\in\mathcal{B}_{df}$ with $\tau(p_3)\neq f$. Since $f$ is used on $p_1$, we also have $\tau(p_2)$, and there is an $r\in L(u_1)\setminus\{\tau(p_1), \tau(p_2)\}$ with $r\neq f$. But then $f\in L(w)\setminus\{r, \tau(p_2), \tau(p_3)\}$, contradicting Subclaim \ref{SecGenObstructionVStarVertK}.  \end{claimproof} \end{addmargin}

\vspace*{-8mm}
\begin{addmargin}[2em]{0em}
\begin{subclaim} $|L(p_1)|=|L(p_4)|=2$. \end{subclaim}

\begin{claimproof} Suppose not. Since $|L(p_1)|>1$, we have $|L(p_4)|=3$ and $|L(p_1)|=3$. Since $|L(u_1)|=3$, it follows from 1) of Theorem \ref{ThmFirstLink3PathForUseInHolepunch} that there is a $\pi\in\textnormal{Crown}(Q, H)$. Let $f=\pi(u_1)$ and $d=\pi(p_4)$. Since $|L(p_1)|=3$, there exist distinct $c^0, c^1\in L(p_1)\setminus\{f\}$. For each $k=0,1$ and $\tau\in\mathcal{B}_{c^kd}$, we have $\tau(p_2)=f$ by Subclaim \ref{CrownColOnU1Blocked}. Recall that $L(p_1)\subseteq L(u_1)$ by 1) of Claim \ref{KBWheelCornerColorClaimM0}, so $L(p_1)=L(u_1)=\{c^0, c^1, f\}$ in this case. Thus, it follows from Subclaim \ref{EachPiNotSame4ListSubCM} that there exists an $s\in L(w)\setminus\{f\}$ with $L(v^*)\setminus\{d, s\}|\geq 4$. At least one of $c^0, c^1$ is distinct from $s$, so suppose without loss of generality that $c^1\neq s$. Since $|\textnormal{Col}(\mathcal{B}_{c^0d}\mid p_3)|\geq 2$, there is a $\tau\in\mathcal{B}_{c^0d}$ such that $\tau(p_3)\neq s$. As indicated above, we have $\tau(p_2)=f$. Since $c^1\in L(u_1)\setminus\{c^0, f\}$ and $s\in L(w)\setminus\{f, c^1, \tau(p_3)\}$, we contradict Subclaim \ref{SecGenObstructionVStarVertK}. \end{claimproof}\end{addmargin}

Since $|L(p_4)|=2$ and $|L(u_1)|=3$, it follows from 1) of Theorem \ref{ThmFirstLink3PathForUseInHolepunch} that there exist two distinct elements $\pi_0, \pi_1$ of $\textnormal{Crown}(Q, H)$, where $\pi_0(u_1)\neq\pi_1(u_1)$. For each $k=0,1$, let $d_k=\pi_k(p_4)$ and $f_k:=\pi_k(u_1)$. Note that we may suppose that $d_0\neq d_1$ as well. To see this, suppose that $d_0=d_1=d$ for some color $d$. Since $|L(p_4)|=2$, there is a $d'\in L(p_4)\setminus\{d\}$. Since $|L(u_1)|+|\{d'\}|=4$, it follows from 1) of Theorem \ref{ThmFirstLink3PathForUseInHolepunch} that there is a $\pi^*\in\textnormal{Crown}(Q, H)$ with $\pi^*(p_4)=d$. Since $f_0\neq f_1$, there exists a $k\in\{0,1\}$ such that $\pi^*(u_1)\neq d_k$, so we replace the pair $\pi_k, \pi_{1-k}$ with the pair $\pi_k, \pi^*$, each of which uses a different color on $u_1$ and a different color on $p_4$. Thus, we suppose that $d_0\neq d_1$ as well.Recall that, by Claim \ref{Lp1StrictlyLarger1AndKTrianMnCL}, $L(u_1)\subseteq L(w)$, so $\{f_0, f_1\}\subseteq L(w)$. For each $j\in\{0,1\}$, let $Y_j$ be the set of colors of $s\in L(w)\setminus\{f_j\}$ such that $|L(v^*)\setminus\{d_j, s\}|\geq 4$. 

\vspace*{-8mm}
\begin{addmargin}[2em]{0em}
\begin{subclaim}\label{TwoYjSizeAtMost1} For each $j\in\{0,1\}$ with $f_j\in L(p_1)$ and each $c\in L(p_1)\setminus\{f_j\}$, we have $Y_j=L(u_1)\setminus\{c, f_j\}$ and $|Y|=1$. \end{subclaim}

\begin{claimproof} Since $f_j\in L(p_1)$ by assumption, we have $Y_j\neq\varnothing$ by Subclaim \ref{EachPiNotSame4ListSubCM}. Since $|L(p_1)|=2$, we have $L(p_1)\setminus\{f_j\}\neq\varnothing$. Let $c\in L(p_1)\setminus\{f_j\}$. By Subclaim \ref{CrownColOnU1Blocked}, we have $\tau(p_2)=f_j$ for each $\tau\in\mathcal{B}_{cd_j}$. Since $L(p_1)\subseteq L(u_1)$ and $f_j\in L(u_1)$, there is one color $r$ left in $L(u_1)\setminus\{c, f_0\}$. If $\{r\}\neq Y_j$, then there is an $s\in Y_j\setminus\{r\}$ and, since $|\textnormal{Col}(\mathcal{B}_{cd_0}\mid p_3)|\geq 2$, there is a $\tau\in\mathcal{B}_{cd_0}$ with $\tau(p_3)\neq s$. Since $\tau(p_2)=f_j$ and $s\neq f_j$, we contradict Subclaim \ref{SecGenObstructionVStarVertK}. Thus, $\{r\}=Y_j$. \end{claimproof}\end{addmargin}

\vspace*{-8mm}
\begin{addmargin}[2em]{0em}
\begin{subclaim}\label{p1p4EndpointsMainPathDisjoint} $L(p_1)\cap L(p_4)=\varnothing$ and $L(p_4)\subseteq L(v^*)$ \end{subclaim}

\begin{claimproof} Suppose there is a $d\in L(p_1)\cap L(p_4)$. Since $L(p_1)\subseteq L(u_1)\subseteq L(w)$, we have $d\in L(w)$ as well. Now we consider the set $\mathcal{B}_{dd}$. For any $\tau\in\mathcal{B}_{dd}$, we have $\tau(p_2), \tau(p_3)\neq d$. Since $\tau(p_1)=d$ and $|L_{\tau}(u_1)|\geq 1$, there is a color $a\in L_{\tau}(u_1)$ with $a\neq d$. Thus, $d\in L(w)\setminus\{a, \tau(p_2), \tau(p_3)\}$, and since $\tau(p_4)=d$ as well, we contradict Subclaim \ref{SecGenObstructionVStarVertK}. Now suppose there is a $d\in L(p_4)\setminus L(v^*)$. Choose an arbitrary $c\in L(p_1)$ and $\tau\in\mathcal{B}_{cd}$. Since $\tau$ extends to an $L$-coloring of $\textnormal{dom}(\tau)\cup\{u_1, w\}$, we again contradict Subclaim  \ref{SecGenObstructionVStarVertK}. \end{claimproof}\end{addmargin}

\vspace*{-8mm}
\begin{addmargin}[2em]{0em}
\begin{subclaim}\label{Lp1NotEqualToSetF0F1} $L(p_1)\neq\{f_0, f_1\}$. \end{subclaim}

\begin{claimproof} Suppose toward a contradiction that $L(p_1)=\{f_0, f_1\}$. Letting $r$ be the lone color of $L(u_1)\setminus L(p_1)$, it follows from Subclaim \ref{TwoYjSizeAtMost1} that $Y_0=Y_1=\{r\}$. Thus, for each $j=0,1$, we have $L(w)\setminus\{f_j, r\}\subseteq L(v^*)\setminus\{d_j\}$ and $|L(v^*)\setminus\{d_j\}|=4$. Since $f_0\neq f_1$ and $d_0\neq d_1$, it follows that there exist three colors $a,b,s$ such that $L(w)=\{a, b, f_0, f_1, r\}$ and $L(v^*)=\{a, b, d_0, d_1, s\}$. Now, for each $j=0,1$, we have $f_{1-j}\not\in Y_j$, which means $L(v^*)\setminus\{d_j, f_{1-j}\}|=3$, so we have $f_{1-j}\in\{s, d_{1-j}\}$. Since $f_0\neq f_1$, there exists a $k\in\{0,1\}$ such that $f_k=d_k=f$ for some color $f$. Since $\{f_0, f_1\}\subseteq L(p_1)$, we contradict Subclaim \ref{p1p4EndpointsMainPathDisjoint}. \end{claimproof}\end{addmargin}

We now have enough to finish the proof of Claim \ref{K'NotJustAnEdgeCorner} and complete Subsection \ref{ShowingK'NotEdgeSubSec}. Applying Subclaim \ref{Lp1NotEqualToSetF0F1} we suppose for the sake of definiteness that $f_0\not\in L(p_1)$. Since $L(p_1)\subseteq L(u_1)$ and $|L(p_1)|=2$, there is a $c\in L(p_1)\cap L(u_1)$ with $L(p_1)=\{c, f_1\}$ and $L(u_1)=\{c, f_0, f_1\}$. Since $L(p_1)\subseteq L(u_1)\subseteq L(w)$, it follows from Subclaim \ref{p1p4EndpointsMainPathDisjoint} that $|L(w)\setminus L(v^*)|\geq 2$, so there is an $s\in L(w)\setminus L(v^*)$ with $s\neq f_0$. Now, there is an $a\in\{c, f_1\}$ with $L(u_1)\setminus\{a, f_0\}\neq\{s\}$. Since $|\textnormal{Col}(\mathcal{B}_{ad_0}\mid p_3)|\geq 2$, there is a $\tau\in\mathcal{B}_{ad_0}$ with $\tau(p_3)\neq s$, and as $\tau(p_1)\neq f_0$, it follows from Subclaim \ref{CrownColOnU1Blocked} that $\tau(p_2)=f_0$. Since $|L_{\tau}(u_1)|\geq 1$, it follows from our choice of $a$ that there is a $b\in L_{\tau}(u_1)$ with $b\neq s$. Thus, $s\in L(w)\setminus\{b, \tau(p_2), \tau(p_3)\}$. Since $s\not\in L(v^*)$, we contradict Subclaim \ref{SecGenObstructionVStarVertK}. This completes the proof of Claim \ref{K'NotJustAnEdgeCorner}. \end{claimproof}

Since $|E(K')|>1$, $p_3$ is incident to a chord of $C$, and thus, since $G$ is a counterexample to Theorem \ref{CornerColoringMainRes}, it follows that, for any $L$-coloring $\phi$ of $\{p_1, p_4\}$, $\phi$ extends to at most one element of $\textnormal{End}(p_3, P, G)$. 

\subsection{Showing that $G$ is outerplanar: part III}\label{GOuterPlanarFinalSubSec}

In this subsection, we first show that $u^{\dagger}=u_m$, (and thus, in particular, $H$ is just a triangle), and then we show that $K'$ is a broken wheel with principal path $T'$, which implies that $V(G)=V(C)$. 

\begin{Claim}\label{PairOf2PathsIntersectOnUmCLCorner} $u^{\dagger}=u_m$, and $H$ is a triangle \end{Claim}

\begin{claimproof} Suppose not. Thus, $u^{\dagger}\in\{u_{m+1}, \ldots, u_t, p_4\}$. By Claim \ref{K'NotJustAnEdgeCorner}, $u^{\dagger}\neq p_4$ and, by Claim \ref{nm+1DistAtL1SubCL}, $u^{\dagger}\neq u_{m+1}$, so $u^{\dagger}=u_n$ for some $n\in\{m+2, \ldots, t\}$.  Applying Claim \ref{UniquewAdjacentAllThreeUmP2P3MainCL}, we let $w$ be the unique vertex of $H\setminus C^H$ adjacent to each of $u_m, p_2, p_3$, and we let $Q$ be the 3-path $u_mp_2p_3u^{\dagger}$. By Claim \ref{Lp1StrictlyLarger1AndKTrianMnCL}, $K$ is a triangle with $|L(p_1)|>1$. In particular, $m=1$.

\vspace*{-8mm}
\begin{addmargin}[2em]{0em}
\begin{subclaim}\label{tauAndROnlyLeaveTwoColorsSubCL} For any $L$-coloring $\tau$ of $V(P)$ which does not extend to an $L$-coloring of $G$, and any $r\in L(u_1)\setminus\{\tau(p_1), \tau(p_2)\}$, we have $|L(w)\setminus\{r, \tau(p_2), \tau(p_3)\}|=2$. \end{subclaim}

\begin{claimproof} Suppose there exist $\tau$ and $r$ violating Subclaim \ref{tauAndROnlyLeaveTwoColorsSubCL}. Thus, $|L(w)\setminus\{r, \tau(p_2), \tau(p_3)\}|\geq 3$. By Claim \ref{UniquewAdjacentAllThreeUmP2P3MainCL}, $w$ is adjacent to all four vertices of $Q$. We now fix an $r'\in\Lambda_{K'}^{T'}(\bullet, \tau(p_3), \tau(p_4))$. Let $\psi$ be the $L$-coloring $(r, \tau(p_2), \tau(p_3), r')$ of $u_mp_2p_3u_n$. Recall that, by Claim \ref{nm+1DistAtL1SubCL}, $n>m+1$, and since $C^H$ is an induced cycle of $G$, $\psi$ is a proper $L$-coloring of its domain. Since $\tau$ does not extend to an $L$-coloring of $G$, $\psi$ does not extend to an $L$-coloring of $H$. Let $J=H\setminus\{p_2, p_3\}$. Since $G$ is short-separation-free and $w$ is adjacent to all four vertices of $Q$, $H\setminus\{p_2, p_3\}$ is bounded by outer cycle $u_mwu_nu_{n-1}\ldots u_{m+1}$, and the outer cycle of $J$ contains the 2-path $u_mwu_n$. Since $\psi$ does not extend to an $L$-coloring of $H$, we have $\Lambda_J^{u_mwu_n}(r, \bullet, r')\cap L_{\psi}(w)=\varnothing$. By 3) of Corollary \ref{CorMainEitherBWheelAtM1ColCor} applied to $J$, $J$ is a broken wheel with principal path $u_mwu_n$, where $J-w$ is a path of odd length. Thus, $H$ is a wheel with central vertex $w$ adjacent to all the vertices of the outer cycle of $H$. 

Since $|E(K')|>1$, it follows from Claim \ref{ForEachdSetSAtLeastTwoLambdaCL} that there is an $L$-coloring $\sigma$ of $V(P)$ which does not extend to an $L$-coloring of $G$, where $|\Lambda_{K'}^{T'}(\bullet, \sigma(p_3), \sigma(p_4))|\geq 2$. Let $b\in\Lambda_K^T(\sigma(p_1), \sigma(p_2), \bullet)$. Since $\sigma$ does not extend to an $L$-coloring of $G$, it follows that, for each $c\in L(w)\setminus\{b, \sigma(p_2), \sigma)(p_3)\}$, we have $|\Lambda_J^{u_mwu_n}(b, c, \bullet)\cap\Lambda_{K'}^{T'}(\bullet, \sigma(p_3), \sigma(p_4))=\varnothing$. Since $|\Lambda_{K'}^{T'}(\bullet, \sigma(p_3), \sigma(p_4))|\geq 2$, there is a lone color $f\in L(u_n)$ sich that $\Lambda_J^{u_mwu_n}(b, c, \bullet)=\{f\}$ for each $c\in L(w)\setminus\{b, \sigma(p_2), \sigma(p_3)\}$. Since $|L(w)\setminus\{b, \sigma(p_2), \sigma(p_3)\}|\geq 2$ and $J-w$ is a path of odd length, we contradict 1a) of Theorem \ref{BWheelMainRevListThm2} applied to $J$.\end{claimproof}\end{addmargin}

\vspace*{-8mm}
\begin{addmargin}[2em]{0em}
\begin{subclaim}\label{BCSetOfPreciselyThreeCol} For any $L$-coloring $\tau$ of $V(P)$ which does not extend to an $L$-coloring of $G$, we have $\tau(p_3)\in\{\tau(p_1)\}\cup (L(w)\setminus L(u_1))$. \end{subclaim}

\begin{claimproof} Let $c=\tau(p_1)$ and suppose toward a contradiction that $\tau(p_3)\not\in\{c\}\cup (L(w)\setminus L(u_1))$. If $\tau(p_3)\not\in L(w)$, then we violate Subclaim \ref{tauAndROnlyLeaveTwoColorsSubCL}, so $\tau(p_3)\in L(w)$. Thus, by assumption, $\tau(p_3)\in L(u_1)\setminus\{c\}$. Since $c\neq\tau(p_3)$ and $\tau(p_2)\neq\tau(p_3)$, and since $K$ is a triangle, it follows that $\tau(p_3)\in L(u_1)\setminus\{c, \tau(p_2)\}$, so we choose $r=\tau(p_3)$ and $|L(w)\setminus\{r, \tau(p_2), \tau(p_3)\}|\geq 3$, contradicting Subclaim \ref{tauAndROnlyLeaveTwoColorsSubCL}. \end{claimproof}\end{addmargin}

\vspace*{-8mm}
\begin{addmargin}[2em]{0em}
\begin{subclaim}\label{ColSubDisjointUnionSet3} For each $c\in L(p_1)$ and $d\in L(p_4)$, we have $\textnormal{Col}(\mathcal{B}_{cd}\mid p_3)=\{c\}\cup (L(w)\setminus L(u_1))$ as a disjoint union, and, in particular, $|\textnormal{Col}(\mathcal{B}_{cd}\mid p_3)|=3$. \end{subclaim}

\begin{claimproof} By Claim \ref{ObvObstructionColorSetS}, since $|E(K')|>1$, it follows that, for each $c\in L(P_1)$ and $d\in L(p_4)$, we have $|\textnormal{Col}(\mathcal{B}_{cd}\mid p_3)|\geq 3$. By 1) of Claim \ref{KBWheelCornerColorClaimM0}, $L(p_1)\subseteq L(u_1)$, and, by Claim \ref{Lp1StrictlyLarger1AndKTrianMnCL}, $L(u_1)\subseteq L(w)$, since $u_1p_3\not\in E(G)$ by assumption. Thus, $|L(w)\setminus L(u_1)|=2$ and $c\in L(w)\setminus L(u_1)$. Thus, $|\{c\}\cup (L(w)\setminus L(u_1)|=3$, and, by Subclaim \ref{BCSetOfPreciselyThreeCol}, $\textnormal{Col}(\mathcal{B}_{cd}\mid p_3)=\{c\}\cup (L(w)\setminus L(u_1))$ as a disjoint union. \end{claimproof}\end{addmargin}

Note that no chord of the outer cycle of $H\cup K'$ is incident to $p_2$. Since $|L(u_1)|=3$, it follows from Theorem \ref{3ChordVersionMainThm1} that there exists a $\pi\in\textnormal{End}(u_1p_2p_3p_4, H\cup K')$. Let $d:=\pi(p_4)$ and $r:=\pi(u_1)$. 

\vspace*{-8mm}
\begin{addmargin}[2em]{0em}
\begin{subclaim}\label{AugElementObstructedP2ColorSubForM} For any $c\in L(p_1)\setminus\{\pi(u_1)\}$ and $\tau\in\mathcal{B}_{cd}$, we have $\tau(p_2)=r$. \end{subclaim}

\begin{claimproof} Suppose there is a $\tau\in\mathcal{B}_{cd}$ with $\tau(p_2)\neq r$. Possibly $\tau(p_3)=r$, but, since $p_3u_1\not\in E(G)$ by assumption, the union $\tau\cup\pi$ is a proper $L$-coloring of its domain. Since $K$ is a triangle and $\pi\in\textnormal{End}(u_1p_2p_3p_4, H\cup K')$, it follows that $\tau\cup\pi$ extends to an $L$-coloring of $G$, which is false, since $\tau\in\mathcal{B}_{cd}$.  \end{claimproof}\end{addmargin}

Recall that, by Claim \ref{Lp1StrictlyLarger1AndKTrianMnCL}, $|L(p_1)|>1$, so let $c$ be a color of $L(p_1)\setminus\{r\}$. Consider the set $\mathcal{B}_{cd}$, recalling that $d=\pi(p_4)$. Since $c\neq r$, it follows from Subclaim \ref{AugElementObstructedP2ColorSubForM} that, for each $\tau\in\mathcal{B}_{cd}$, we have $\tau(p_2)=r$. Let $f$ be a color of $L(u_1)\setminus\{c,r\}$ and let $L'$ be a list-assignment for $V(H\cup K')$, where $L'(p_3)=\textnormal{Col}(\mathcal{B}_{cd}\mid p_3)$ and otherwise $L=L'$. Now, there is an $L'$-coloring $\sigma$ of $\{u_1, p_2, p_4\}$ using $f, r, d$ on the respective vertices $u_1, p_2, p_3$. If $\sigma$ extends to an $L'$-coloring of $H\cup K'$, then the color $c$ is left over for $p_1$ and since each element of $\mathcal{B}_{cd}$ uses $r$ on $p_2$, it then follows that there is a $\tau\in\mathcal{B}_{cd}$ which extends to an $L$-color of $G$, which is false. Thus, $\sigma$ does not extend to an $L'$-coloring of $H\cup K'$. By our assumption, $u_1$ is not adjacent to $p_3$, so $L'_{\sigma}(p_3)=\textnormal{Col}(\mathcal{B}_{cd}\mid p_3)\setminus\{d, r\}$. By Subclaim \ref{ColSubDisjointUnionSet3}, we have $\textnormal{Col}(\mathcal{B}_{cd}\mid p_3)=\{c\}\cup (L(w)\setminus L(u_1))$. In particular, since $r\neq c$ and $r\in L(u_1)$, we have $r\not\in \textnormal{Col}(\mathcal{B}_{cd}\mid p_3)$. By definition, we have $d\not\in\textnormal{Col}(\mathcal{B}_{cd}\mid p_3)$ either, so it follows from Subclaim \ref{ColSubDisjointUnionSet3} that $|L'_{\sigma}(p_3)|\geq 3$. Since every chord of the outer face of $H\cup K'$ is incident to $p_3$, it follows from Lemma \ref{PartialPathColoringExtCL0} that there is a vertex of $H\cup K'$ which does not lie on the outer face of $H\cup K'$ and has three neighbors on $\{u_1, p_2, p_4\}$, so this vertex is precisely $w$ and we contradict Claim \ref{CornerCommonNbrCaseP2P4NotAdjw}. This completes the proof of Claim \ref{PairOf2PathsIntersectOnUmCLCorner}. \end{claimproof}

Since $G$ is short-separation-free, it follows from Claim \ref{PairOf2PathsIntersectOnUmCLCorner} that $H$ is the triangle $p_2p_3u_m$, and $G=(K\cup K')+p_2p_3$. Analogous to Claim \ref{KBWheelCornerColorClaimM0}, we have the following.

\begin{Claim} $K'$ is a broken wheel with principal path $T'$. In particular, $G$ is outerplanar.  \end{Claim}

\begin{claimproof} Since $|L(u_m)|=3$, it follows from Theorem \ref{SumTo4For2PathColorEnds} that $\textnormal{Col}(\textnormal{End}(T, K)\mid u_m)=\textnormal{Col}(\textnormal{End}(T, K')\mid u_m)=L(u_m)$. Since $|L(u_m)|=3$ and $|L(p_1)|+|L(p_4)|=4$, there is an $L$-coloring $\psi$ of $\{p_1, u_m, p_4\}$ where $\psi$ is $(P,G)$-sufficient. Let $a=\psi(p_1)$ and $d=\psi(p_4)$, and let $b=\psi(u_m)$. For each $\sigma\in\mathcal{B}_{ad}$, we have $\sigma(p_2)=b$, or else the color $b$ is left over for $u_m$ and thus $\sigma$ extends to an $L$-coloring of $G$. For each $\sigma\in\mathcal{B}_{ad}$, since $\sigma$ does not extend to $L$-color $G$, we have $\Lambda_K^T(a, b, \bullet)\cap\Lambda_{K'}^{T'}(\bullet, \sigma(p_3), d)=\varnothing$. Now we show that $K'$ is a broken wheel with principal path $T'$. Suppose not. Let $r\in\Lambda_K^T(a, b, \bullet)$. By Claim \ref{ObvObstructionColorSetS}, $|\textnormal{Col}(\mathcal{B}_{ad}\mid p_3)|\geq 3$. Possibly $r=d$, but in any case, by assumption, $K'$ is not a triangle, and since $|\textnormal{Col}(\mathcal{B}_{ad}\mid p_3)|\geq 2$ and $K'$ is not a broken wheel with principal $T'$, it follows from 1) of Theorem \ref{EitherBWheelOrAtMostOneColThm} that $\textnormal{Col}(\mathcal{B}_{ad}\mid p_3)\cap\Lambda_{K'}^{T'}(r, \bullet, d)\neq\varnothing$, so there is a $\sigma\in\mathcal{B}_{ad}$ which extends to an $L$-coloring of $G$, which is false. Thus, $K'$ is a broken wheel with principal path $T'$. By 1) of Claim \ref{KBWheelCornerColorClaimM0}, $K$ is broken wheel with principal path $T$, so $V(G)=V(C)$. \end{claimproof}

\subsection{Completing the proof of Theorem \ref{CornerColoringMainRes}}

\begin{Claim}\label{AtMostOneKK'TriangleCLM} At most one of $K, K'$ is a triangle. \end{Claim}

\begin{claimproof} Suppose toward a contradiction that each of $K$ and $K'$ is a triangle, i.e $m=1=t$ and $|V(G)|=5$. Since $|L(p_1)|+|L(p_4)|=4$, there is an $L$-coloring $\phi$ of $\{p_1, p_4\}$ such that $|L_{\phi}(u_1)|\geq 2$. Let $c=\phi(p_1)$ and $d=\phi(p_4)$. By Claim \ref{ObvObstructionColorSetS}, $|\textnormal{Col}(\mathcal{B}_{cd}\mid p_3)|\geq 3$, so there is an $r\in\textnormal{Col}(\mathcal{B}_{cd}\mid p_3)$ with $L_{\phi}(u_1)\setminus\{r\}|\geq 2$. But then, there is a $\phi'\in\mathcal{B}_{cd}$ with $\phi'(p_3)=r$, so $\phi'$ extends to an $L$-coloring of $G$, which is false. \end{claimproof}

\begin{Claim}\label{p4endpointListInsideUTList} $L(p_4)\subseteq L(u_t)$. In particular, for each $d\in L(p_4)$, we have $L(p_3)\setminus (\{d\}\cup L(u_t))|\geq |L(p_3)\setminus\{d\}|-2$. \end{Claim}

\begin{claimproof} Suppose toward a contradiction that $L(p_4)\not\subseteq L(u_t)$. Since $|L(p_1)|\geq 1$, there is an $L$-coloring $\phi$ of $\{p_1, p_4\}$ with $\phi(p_4)\not\in L(u_t)$. Let $c=\phi(p_1)$ and $d=\phi(p_4)$. By Claim \ref{ObvObstructionColorSetS}, there exist three distinct elements $\phi^0, \phi^1, \phi^2$ of  $\mathcal{B}_{cd}$, each of which uses a different color on $p_3$. Since $K'$ is a broken wheel with principal path $T'$ and $d\not\in L(u_t)$, it follows that, for each $i=0,1,2,$, we have $\Lambda_{K'}^{T'}(\bullet, \phi^i(p_3), d)=L(u_m)\setminus\{\phi^i(p_3)\}$ and thus $\Lambda_K^T(c, \phi^i(p_2), \bullet)=\{\phi^i(p_3)\}$. This is true even if $K'$ is a triangle. For each $0\leq i<j\leq 2$, since $\phi^i(p_3)\neq\phi^j(p_3)$, it follows that $\phi^i(p_2)\neq\phi^j(p_2)$, contradicting 2) of Theorem \ref{BWheelMainRevListThm2}. We conclude that $L(p_4)\subseteq L(u_t)$.  \end{claimproof}

We now have the following result analogous to Claim \ref{IfNotTriThenIntersecBdCL}.

\begin{Claim}\label{Eitherp4Less3OrPathLenTwoCLK'} Either $K'-p_2$ is a path of length at most two or both of the following hold.
\begin{enumerate}[label=\arabic*)]
\itemsep-0.1em
\item $L(p_4)\not\subseteq L(u_{t-1})$; AND
\item $|L(p_4)|<3$.
\end{enumerate}

\end{Claim}

\begin{claimproof} Suppose that $|L(p_4)|=3$ and $K'-p_2$ is a path of length at least three. That is, $t\geq 3$ and $m\leq t-2$. Let $G^*:=G\setminus\{p_4, u_t\}$ and let $P^*:=p_1p_2p_3u_{t-1}$. Note that $G^*$ is bounded by outer cycle $p_1p_2p_3u_{t_1}\ldots u_1$. Let $Y$ be the set of $L$-colorings $\sigma$ of $\{p_1, u_{t-1}\}$ such that $\sigma$ extends to at least $|L_{\sigma}(p_3)|-2$ different elements of $\textnormal{End}_L(p_3, P^*, G^*)$. 

\vspace*{-8mm}
\begin{addmargin}[2em]{0em}
\begin{subclaim}\label{Yp4DisjointSubCL0} $\textnormal{Col}(Y, u_{t-1})\cap L(p_4)=\varnothing$. \end{subclaim}

\begin{claimproof} Suppose toward a contradiction that $\textnormal{Col}(Y, u_{t-1})\cap L(p_4)\neq\varnothing$.  Thus, there is a $d\in L(u_{t-1})\cap L(p_4)$ and an $L$-coloring $\sigma$ of $\{p_1, u_{t-1}\}$ with $\sigma(p_4)=d$, where $\sigma$ extends to $|L_{\sigma}(p_3)|-2$ elements of $\textnormal{End}_L(p_3, P^*, G^*)$. Let $c=\sigma(p_1)$. Since $p_1p_3\not\in E(G)$, we have $L_{\sigma}(p_3)=L(p_3)\setminus\{d\}$. Let $B$ be the set of $s\in L(p_3)$ such that $\sigma$ extends to an element of $\textnormal{End}_{L}(p_3, P^*, G^*)$ using $s$ on $p_3$. Since $\sigma(u_{t-1})\in Y$, we have $|B|\geq |L(p_3)\setminus\{d\}|-2$. Since $d\in L(p_4)$ and $p_3$ is incident to a chord of $C$, it follows from Claim \ref{ObvObstructionColorSetS} that there is a $\tau\in\mathcal{B}_{cd}$ with $\tau(p_3)\in B$. Possibly $\tau(p_2)=d$, but since $K'-p_3$ is a path of length at least three, $p_2u_{t-1}\not\in E(G)$, so the union $\sigma\cup\tau$ is a proper $L$-coloring of its domain which extends to $L$-color $V(G-u_t)$. As two neighbors of $u_t$ are colored with $d$, it follows that $\sigma\cup\tau$ extends to $L$-color $G$, so $\tau$ extends to $L$-color $G$, which is false. \end{claimproof}\end{addmargin}

Since $K'-p_3$ is a path of length at least three by assumption, $p_3$ is incident to a chord of the outer face of $G^*$. As $|L(u_{t-1})|=3$ and $|V(G^*)|<|V(G)$ is a minimal counterexample to Theorem \ref{CornerColoringMainRes}, it follows that $|\textnormal{Col}(Y, p_3)|\geq 4-|L(p_4)|$. Since $|L(u_{t-1})|=3$, it follows from Subclaim \ref{Yp4DisjointSubCL0} that $L(p_4)\not\subseteq L(u_{t-1})$. This proves 1). Now we prove 2). Suppose toward a contradiction that $|L(p_4)|=3$.  As indicated above, we have $|\textnormal{Col}(Y, p_3)|+|L(p_4)|\geq 4$. In particular, since $|L(p_4)|\leq 3$, $Y\neq\varnothing$, so let $\sigma\in Y$. Let $c=\sigma(p_1)$ and $d=\sigma(u_{t-1})$. Since $p_1p_3\not\in E(G)$, we have $L_{\sigma}(p_3)=L(p_3)\setminus\{d\}$. Let $B$ be the set of $s\in L(p_3)$ such that $\sigma$ extends to an element of $\textnormal{End}_{L}(p_3, P^*, G^*)$ using $s$ on $p_3$. Since $\sigma(u_{t-1})\in Y$, we have $|B|\geq |L(p_3)\setminus\{d\}|-2$. By Claim \ref{p4endpointListInsideUTList}, since $|L(p_4)|=3$, we have $L(p_4)=L(u_t)$. Thus, by Subclaim \ref{Yp4DisjointSubCL0}, $\sigma(u_{t-1})\not\in L(u_t)$.

\vspace*{-8mm}
\begin{addmargin}[2em]{0em}
\begin{subclaim}\label{EachColFDisjointB} For each $f\in L(p_4)$, we have $\textnormal{Col}(\mathcal{B}_{cf}\mid p_3)\cap B=\varnothing$. \end{subclaim}

\begin{claimproof} Let $f\in L(p_4)$ and suppose toward a contradiction that $\textnormal{Col}(\mathcal{B}_{cf}\mid p_3)\cap B\neq\varnothing$. Thus, there is a $\tau\in\mathcal{B}_{cf}$ with $\tau(p_3)\in B$. Possibly $\tau(p_2)=d$, but since $K'-p_3$ is a path of length at least three, we have $u_{t-1}\not\in N(p_2)$, so the union $\tau\cup\sigma$ is a proper $L$-coloring of its domain which extends to $L$-color $V(G-u_t)$. But since $\sigma(u_{t-1})\not\in L(u_t)$, it follows that $\tau\cup\sigma$ extends to an $L$-coloring of $G$, so $\tau$ extends to an $L$-coloring of $G$, which is false. \end{claimproof}\end{addmargin}

\vspace*{-8mm}
\begin{addmargin}[2em]{0em}
\begin{subclaim}\label{SubCLColSetDisjointUt} For each $f\in L(p_4)$, we have $\textnormal{Col}(\mathcal{B}_{cf}\mid p_3)\cap L(u_t)=\varnothing$.  \end{subclaim}

\begin{claimproof} Let $f\in L(p_4)$. By Claim \ref{ObvObstructionColorSetS}, $|\textnormal{Col}(\mathcal{B}_{cf}\mid p_3)|\geq 3$. Since $|B|\geq |L(p_3)\setminus\{d\}|-2$, it follows from Subclaim \ref{EachColFDisjointB} that $\textnormal{Col}(\mathcal{B}_{cf}\mid p_3)=L(p_3)\setminus B$. In particular, $\mathcal{Col}(\mathcal{B}_{cf}\mid p_3)$ is a constant set as $f$ runs over $L(p_4)$. For each $f\in L(p_4)$, since $f\not\in\textnormal{Col}(\mathcal{B}_{cf}\mid p_3)$, it follows that $\textnormal{Col}(\mathcal{B}_{cf}\mid p_3)\cap L(p_4)=\varnothing$, and since $L(p_4)=L(u_t)$, we have $\textnormal{Col}(\mathcal{B}_{cf}\mid p_3)\cap L(u_t)=\varnothing$. \end{claimproof}\end{addmargin}

Now choose an arbitrary $f\in L(p_4)$. Since $|\textnormal{Col}(\mathcal{B}_{cf}\mid p_3)|\geq 3$, we let $\tau_0, \tau_1, \tau_2$ be three distinct elements of $\mathcal{B}_{cf}$, each of which uses a different color on $p_3$. By Subclaim \ref{SubCLColSetDisjointUt}, $\tau_i(p_3)\in L(p_3)\setminus L(u_t)$ for each $i=0,1,2$. Thus, for each $i=0,1,2$, we have $\Lambda_{K'}^{T'}(\bullet, \tau_i(p_3), f)=L(u_m)\setminus\{\tau_i(p_3)\}$, and since $\tau_i$ does not extend to $L$-color $G$, we have $\Lambda_K^T(c, \tau_i(p_2), \bullet)=\{\tau_i(p_3)\}$. Since each $\tau_i$ uses a different color on $p_3$, it follows from the above that each $\tau_i$ also uses a different color on $p_2$, and we contradict 2) of Theorem \ref{BWheelMainRevListThm2}. This proves Claim \ref{Eitherp4Less3OrPathLenTwoCLK'}. \end{claimproof}

Applying Claim \ref{Eitherp4Less3OrPathLenTwoCLK'}, we now have the following. 

\begin{Claim}\label{Lp4FewerThanThreeColorsCLMainCase} $|L(p_4)|<3$. \end{Claim}

\begin{claimproof} Suppose not. Thus, $|L(p_4)|=3$ and $L(p_1)=\{c\}$ for some color $c$. Since $|L(u_t)|=3$, it follows from Claim \ref{p4endpointListInsideUTList} that $L(p_4)=L(u_t)$. Since $|L(p_1)|+|L(u_m)|=4$, it follows from Theorem \ref{SumTo4For2PathColorEnds} that there is a $\psi\in\textnormal{End}(T, K)$. 

\vspace*{-8mm}
\begin{addmargin}[2em]{0em}
\begin{subclaim} $K'$ is not a triangle. \end{subclaim}

\begin{claimproof} Suppose that $K'$ is a triangle. As $G$ is a counterexample to Theorem \ref{CornerColoringMainRes}, it follows from Claim \ref{p4endpointListInsideUTList} that, for each $f\in L(p_4)\setminus\{\psi(u_t)\}$, there exists a $\phi^f\in\mathcal{B}_{cf}$ with $\phi^f(p_3)\in L(p_3)\setminus L(u_t)$. Since $K'$ is a triangle and $\phi^f(p_3)\not\in L(u_t)$, it follows that $\Lambda_K^T(c, \phi^f(p_2), \bullet)=\{f\}$. If there exists an $f\in L(p_4)\setminus\{\psi(u_t)\}$ with $\phi^f(p_2)\neq\psi(u_t)$, then, since $\psi\in\textnormal{End}(T, K)$, we have $\psi(u_t)\in\Lambda_K^T(c, \phi^f(p_2), \bullet)$, which is false, as $\psi(u_t)\neq f$. We conclude that $\phi^f(p_2)=\psi(u_t)$. As this holds for each $f\in L(p_4)\setminus\{\psi(u_t)\}$, it follows that $L(p_4)\setminus\{\psi(u_t)\}$ consists of the lone color of $\Lambda_K^T(c, \psi(u_t), \bullet)$, contradicting our assumption that $|L(p_4)|=3$. \end{claimproof}\end{addmargin}

Since $K'$ is not a triangle and $|L(p_4)|=3$, it follows from Claim \ref{Eitherp4Less3OrPathLenTwoCLK'} that $m=t-1$. Let $c^*:=\psi(u_m)$. 

\vspace*{-8mm}
\begin{addmargin}[2em]{0em}
\begin{subclaim} $c^*\in L(p_2)$. \end{subclaim}

\begin{claimproof} Suppose not. Since $|L(p_4)|=3$, there exists an $a\in L(p_4)$ such that $|L(u_t)\setminus\{a, c^*\}|\geq 2$. Since $G$ is a counterexample to Theorem \ref{CornerColoringMainRes} and $p_3$ is incident to a chord of $C$, there exists a pair $\sigma_0, \sigma_1\in\mathcal{B}_{ca}$, where $\sigma_0(p_3)\neq\sigma_1(p_3)$ and, for each $i=0,1$, we have $\sigma_i(p_3)\in L(p_3)\setminus\{a, c^*\}$. For each $i=0,1$, we have $\sigma_i(p_2)\neq c^*$ by assumption, and thus $c^*\in\Lambda_K^T(\sigma_i(p_1), \sigma_i(p_2), \bullet)$, since $\psi\in\textnormal{End}(T, K)$. But since it also holds that neither $\sigma_0, \sigma_1$ uses $c^*$ on $p_3$, it follows from our choice of $a$ that $c^*\in\Lambda_{K'}^{T'}(\bullet, \sigma_i(p_3), \sigma_i(p_4))$ for each $i=0,1$, so each of $\sigma_0, \sigma_1$ extends to an $L$-coloring of $G$, a contradiction. \end{claimproof}\end{addmargin}

Since $c^*\in L(p_2)$, we fix a color $r\in\Lambda_K^T(c, c^*, \bullet)$. 

\vspace*{-8mm}
\begin{addmargin}[2em]{0em}
\begin{subclaim}\label{ColorCstarNotinLprSubCL} $c^*\not\in L(p_4)$. \end{subclaim}

\begin{claimproof} Suppose that $c^*\in L(p_4)$. Since $G$ is a counterexample to Theorem \ref{CornerColoringMainRes}, there exists a pair of distinct $\phi_0, \phi_1\in\mathcal{B}_{cc^*}$, where, for each $i=0,1$, $\phi_i(p_3)\in L(p_3)\setminus\{c, r\}$, and furthermore, $\phi_0(p_3)\neq\phi_1(p_3)$. Consider the following cases.

\textbf{Case 1:} There exists a $j\in\{0,1\}$ such that $\phi_j(p_2)\neq c^*$, 

In this case, we have $c^*\in\Lambda_K^T(c, \phi_j(p_2), \bullet)$ by our choice of $c^*$. Since $K'-p_3$ is a path of length two and $\phi_j(p_4)=c^*$, we have $c^*\in\Lambda_{K'}^{T'}(\bullet, \phi_j(p_3), \phi_j^*(p_4))$ as well, so $\phi_j$ extends to an $L$-coloring of $G$, a contradiction. 

\textbf{Case 2:} For each $i\in\{0,1\}$, $\phi_i(p_2)=c^*$

In this case, since neither $\phi_0$ nor $\phi_1$ extends to $L$-color $G$, we have $r\not\in\Lambda_K(\bullet, \phi_i(p_3), c^*)$ for each $i=0,1$. Since $K'-p_3$ is a path of length two and $r\not\in\{\phi_0(p_3), \phi_1(p_3)\}$, it follows that, for each $i=0,1$, we have $L(u_t)=\{c^*, r, \phi_i(p_3)\}$. Since $\phi_0(p_3)\neq\phi_1(p_3)$, we have a contradiction. \end{claimproof}\end{addmargin}

As $|L(p_4)|=3$, there is an $a\in L(p_4)$ such that $|L(u_t)\setminus\{a, r\}|\geq 2$. Since $G$ is a counterexample to Theorem \ref{CornerColoringMainRes}, there is an $L$-coloring $\phi$ of $V(P)$ which uses $c, a$ on the respective vertices $p_1, p_4$, where $\phi(p_3)\in L(p_3)\setminus\{a, c^*, r\}$ and $\phi$ does not extend to $L$-color $G$. By our choice of $a$, we have $r\in\Lambda_{K'}^{T'}(\bullet, \phi(p_3), a)$, since $\phi(p_3)\neq r$. If $\phi(p_2)=c^*$, then we have $r\in\Lambda_K^T(c, \phi(p_2), \bullet)$ as well, so $\phi$ extend to an $L$-coloring of $G$, a contradiction. Since $\phi(p_2)\neq c^*$, we have $c^*\in\Lambda_K^T(c, \phi(p_2), \bullet)$ by our choice of $c^*$. Since $L(p_4)=L(u_t)$, we have $c^*\not\in L(u_t)$ by Subclaim \ref{ColorCstarNotinLprSubCL}. Since $\phi(p_3)\neq c^*$, we have $c^*\in\Lambda_{K'}^{T'}(\bullet, \phi(p_3), a)$. Thus, $\phi$ extend to an $L$-coloring of $G$, a contradiction.  \end{claimproof}

\begin{Claim}\label{EliminateCasesLp42or3CL} $|L(p_4)|=1$. \end{Claim}

\begin{claimproof} Suppose not. By Claim \ref{Lp4FewerThanThreeColorsCLMainCase}, $|L(p_4)|=2$, so $|L(p_1)|=2$. Let $L(p_1)=\{a_0, a_1\}$ and $L(p_4)=\{b_0, b_1\}$. By 2) of Claim \ref{KBWheelCornerColorClaimM0}, $K$ is a triangle, i.e $m=1$. By Claim \ref{AtMostOneKK'TriangleCLM}, $K'$ is not a triangle. By 1) of Claim \ref{KBWheelCornerColorClaimM0}, we have $\{a_0, a_1\}\subseteq L(u_1)$. Let $f$ be the lone color of $L(u_1)\setminus\{a_0, a_1\}$.

\vspace*{-8mm}
\begin{addmargin}[2em]{0em}
\begin{subclaim} $L(p_4)\subseteq L(u_{t-1})$. \end{subclaim}

\begin{claimproof} Suppose not. Thus, there exists a $j'\in\{0,1\}$ such that $b_{j'}\not\in L(u_{t-1})$. Now, for each $i=0,1$, there is an $L$-coloring $\pi^i$ of $V(P)$ which does not extend to an $L$-coloring of $G$, where $\pi^i$ uses $a_i, b_{1-j'}$ on the respective vertices $p_1, p_4$, and $\pi^i(p_3)\in L(p_3)\setminus\{b_{1-j'}, a_{1-i}, f\}$. Now, since $b_{j'}\in L(u_t)\setminus L(u_{t-1})$ and $K'$ is not a triangle, it follows that, for each $i=0,1$, we have $\{b_{1-j'}, f\}\subseteq\Lambda_{K'}^{T'}(\bullet, \pi^i(p_3), b_{1-j'})$. This is true even if $K'-p_3$ is a path of length two. Thus, for each $i=0,1$, we have $L(u_1)\setminus\{a_i\}\subseteq\Lambda_{K'}^{T'}(\bullet, \pi^i(p_3), b_{1-j'})$, so $\pi^i$ extends to an $L$-coloring of $G$, a contradiction. \end{claimproof}\end{addmargin}

Since $K'$ is not a triangle and $L(p_4)\subseteq L(u_{t-1})$, it follows from Claim \ref{Eitherp4Less3OrPathLenTwoCLK'} that $K'-p_3$ is a path of length precisely two, i.e $t=2$ and $m=1=t-1$. Since $\{a_0, a_1\}$ and $\{b_0, b_1\}$ are both subsets of $L(u_1)$, and $|L(u_1)|=3$, we suppose without loss of generality that $a_0=b_0$. Now consider the following cases.

\textbf{Case 1:} $b_1\neq a_1$

Since $G$ is a counterexample to Theorem \ref{CornerColoringMainRes}, there is an $L$-coloring $\pi$ of $V(P)$ which does not extend to an $L$-coloring of $G$, where $\pi$ uses $a_0, b_1$ on the respective vertices $p_1, p_4$ and $\pi(p_3)\in L(p_3)\setminus\{a_1, b_1, f\}$. But since $a_0$ is left over for $u_2$, we have $\{\pi(p_2), \pi(p_3)\}=L(u_1)\setminus\{a_0\}=\{a_1, f\}$, which is false, since $\pi(p_3)\not\in\{a_1, f\}$. 

\textbf{Case 2} $b_1=a_1$

In this case, we have $L(p_1)=L(p_4)$. Since $G$ is a counterexample to Theorem \ref{CornerColoringMainRes}, there is an $L$-coloring $\pi$ of $V(P)$ which does not extend to an $L$-coloring of $G$, where $\pi$ uses $a_0, a_1$ on the respective vertices $p_1, p_4$ and $\pi(p_3)\in L(p_3)\setminus L(u_1)$. This is permissible since $a_1\in L(u_1)$. As above, the color $a_0$ is left over for $u_2$, so $\{\pi(p_2), \pi(p_3)\}=\{a_1, f\}$, which is false, since $\pi(p_3)\not\in\{a_1, f\}$. This completes the proof of Claim \ref{EliminateCasesLp42or3CL}. \end{claimproof}

Applying Claim \ref{EliminateCasesLp42or3CL}, we have $|L(p_4)|=1$ and $|L(p_3)|=3$. Let $L(p_4)=\{d\}$ for some color $d$, and let $L(p_1)=\{a_0, a_1, a_2\}$. By Claim \ref{p4endpointListInsideUTList}, we have $d\in L(u_t)$. By Claim \ref{KBWheelCornerColorClaimM0}, we have $L(p_1)=L(u_1)$, and $K$ is a triangle, i.e $m=1$. Thus, by Claim \ref{AtMostOneKK'TriangleCLM}, $K'$ is not a triangle. 

\begin{Claim}\label{NotUniformColorCL13Case} There exists a $u\in V(K'\setminus T')$ such that $L(u)\neq L(u_1)$. \end{Claim}

\begin{claimproof} Suppose not. Since $d\in L(u_t)$, we have $d\in L(u_1)$, and since $L(u_1)=L(p_1)$, the vertices of the path $p_1u_1\ldots u_t$ all have the same 3-list. In particular, $|L(p_3)\setminus (\{d\}\cup L(u_1))|\geq 2$, and since $G$ is a counterexample to Theorem \ref{CornerColoringMainRes}, there is an $L$-coloring $\phi$ of $V(P)$ which does not extend to an $L$-coloring of $G$, where $\phi(p_3)\not\in L(u_1)$. Since $K'$ is not a triangle and $\phi(p_3)$ lies in the list of no vertex of the path $u_1\ldots u_t$, it follows that $\Lambda_{K'}^{T'}(\bullet, \phi(p_3), d)=L(u_1)$. Thus, $\Lambda_K^T(\phi(p_1), \phi(p_2), \bullet)\cap\Lambda_{K'}^{T'}(\bullet, \phi(p_3), d)\neq\varnothing$, so $\phi$ extends to an $L$-coloring of $G$, a contradiction. \end{claimproof}

\begin{Claim}\label{LClaimCaseL(u1)L(u2)} $L(u_1)=L(u_2)$. \end{Claim}

\begin{claimproof} Suppose toward a contradiction that $L(u_1)\neq L(u_2)$. Since $L(u_1)=L(p_1)$, we suppose without loss of generality that $a_2\not\in L(u_2)$. Since $G$ is a counterexample to Theorem \ref{CornerColoringMainRes} and $d$ is the lone color of $L(p_4)$, it follows that, for each $i=0,1$, there is an $L$-coloring $\pi^i$ of $V(P)$ which does not extend to an $L$-coloring of $G$, where $\pi^i(p_1)=a_i$ and $\pi^i(p_3)\in L(p_3)\setminus\{a_{1-i}, a_2, d\}$. For each $i=0,1$, since $\pi^i(p_3)\in L(p_3)\setminus\{a_{1-i}, a_2, d\}$, we have $L_{\pi^i}(u_1)=L(u_1)\setminus\{a_i, \pi^i(p_2)\}$, so $L_{\pi^i}(u_1)\neq\varnothing$. Since $\pi^i$ does not extend to an $L$-coloring of $G$, we have $a_2\not\in L_{\pi^i}(u_1)$, so we have $\pi^i(p_2)=a_2$ and $L_{\pi^i}(u_1)=\{a_{1-i}\}$. That is, for each $i=0,1$, we have $\Lambda_K^T(\pi^i(p_1), \pi^i(p_2), \bullet)=\{a_{1-i}\}$. For each $i=0,1$, since $\pi^i$ does not extend to an $L$-coloring of $G$, we have $a_{1-i}\not\in\Lambda_{K'}^{T'}(\bullet, \pi^i(p_3), d)$, and since $\pi^i(p_3)\neq a_{1-i}$, we have $a_{1-i}\in L(u_2)$ and $\pi^i(p_3)\in L(u_2)\cap\ldots\cap L(u_t)$. Let $S:=\{\pi^0(p_3), \pi^1(p_3)\}$.

\vspace*{-8mm}
\begin{addmargin}[2em]{0em}
\begin{subclaim} $S=\{a_0, a_1\}$. In particular, for each $i=0,1$, we have $\pi^i(p_3)=a_i$. \end{subclaim}

\begin{claimproof} Since $K'$ is not a triangle, we let $H$ be the broken wheel $K'-u_1$ with principal path $u_2p_3p_4$. For each $i=0,1$, since $\Lambda_K^T(a_i, \pi^i(p_2), \bullet)=\{a_{1-i}\}$ and $\pi^i(p_3)\neq a_{1-i}$, we have $\Lambda_H(\bullet, \pi^i(p_3), d)=\{a_{1-i}\}$, or else $\pi^i$ extends to an $L$-coloring of $G$. Thus, it follows from 1b) of Theorem \ref{BWheelMainRevListThm2} applied to $H$ that $\pi^i(p_3)=a_i$ for each $i=0,1$, so $S=\{a_0, a_1\}$.\end{claimproof}\end{addmargin}

Since $G$ is a counterexample to Theorem \ref{CornerColoringMainRes}, there is an $L$-coloring $\tau$ of $V(P)$ which does not extend to an $L$-coloring of $G$, where $\tau(p_1)=a_2$ and $\tau(p_3)\in L(p_3)\setminus (S\cup\{d\})$. Since $S=\{a_0, a_1\}$ and $\tau(p_1)=a_2$, there exists an $a\in\{a_0, a_1\}\cap L_{\tau}(u_1)$. Recall that $S\subseteq L(u_2)\cap\ldots\cap L(u_t)$. But now, since $\tau(p_3)\not\in S$ and $\tau(p_4)\not\in S$, we simply extend $\tau$ to an $L$-coloring of $G$ by 2-coloring the path $u_1\ldots u_t$ with $S$, starting by using $a$ on $u_1$, contradicting our assumption that $\tau$ does not extend to an $L$-coloring of $G$. \end{claimproof}

Applying Claim \ref{NotUniformColorCL13Case}, we let $\ell\in\{2, \ldots, t\}$ be the minimal index such that $L(u_{\ell})\neq L(u_1)$. Since $L(u_1)=L(p_1)$, we suppose without loss of generality that $a_2\not\in L(u_{\ell})$. By Claim \ref{LClaimCaseL(u1)L(u2)}, we have $\ell>2$. Since $G$ is a counterexample to Theorem \ref{CornerColoringMainRes}, there is an $L$-coloring $\tau$ of $V(P)$ which does not extend to an $L$-coloring of $G$, where $\tau(p_1)=a_2$ and $\tau(p_3)\in L(p_3)\setminus\{d, a_0, a_1\}$. Consider the following cases.

\textbf{Case 1:} $\tau(p_3)=a_2$

In this case, we have $d\neq a_2$ and $L_{\tau}(u_1)\cap\{a_0, a_1\}\neq\varnothing$, so let $i\in\{0,1\}$ with $a_i\in L_{\tau}(u_1)$. Since $K'$ is not a triangle and $a_2\not\in L(u_{\ell})$, the $L$-coloring $(a_i, a_2, d)$ of $u_1p_3p_4$ extends to an $L$-coloring of $K'$, so $\tau$ extends to an $L$-coloring of $G$, a contradiction.

\textbf{Case 2:} $\tau(p_3)\neq a_2$

In this case, since $L(p_1)=L(u_1)$, we have $\tau(p_3)\not\in L(u_1)$. By Claim \ref{LClaimCaseL(u1)L(u2)}, we also have $\tau(p_3)\not\in L(u_2)$, so we get $\Lambda_{K'}^{T'}(\bullet, \tau(p_3), \tau(p_4))=L(u_1)$. Thus, we have $\Lambda_K(a_2, \tau(p_2), \bullet)\cap\Lambda_{K'}^{T'}(\bullet, \tau(p_3), d)\neq\varnothing$, so $\tau$ extends to an $L$-coloring of $G$, a contradiction. This completes the proof of Theorem \ref{CornerColoringMainRes}. \end{proof}

\section{4-Paths on Induced Cycles}\label{BoxLemmaFor4PathSecPropIntermediate}

The second of the two results which makes up Paper II is a somewhat specialized and surprisingly technical lemma, which is Lemma \ref{EndLinked4PathBoxLemmaState}  (indeed, in the proof of Theorem \ref{MainHolepunchPaperResulThm} in Paper III, we use  Lemma \ref{EndLinked4PathBoxLemmaState} precisely once, in order to rule out one configuration). To state Lemma \ref{EndLinked4PathBoxLemmaState}, we first introduce the notation below. 

\begin{defn}\label{PartialLColOmega} \emph{Let $G$ be a planar graph with outer cycle $C$ and list-assignment $L$. Let $P=p_0q_0wq_1p_1$ be a 4-path on $C$. For any partial $L$-coloring $\psi$ of $G-w$, we let $\mathcal{C}_G^P(\psi)$ be the set of $c\in L_{\psi}(w)$ such that $\psi$ extends to an $L$-coloring of $G$ using $c$ on $w$. For any endpoint $p_i$ of $P$, we let $\Omega_G(P, [p_i])$ be the set of partial $L$-colorings $\phi$ of $C\setminus\mathring{P}$ such that 1) and 2) below are satisfied}
\begin{enumerate}[label=\emph{\arabic*)}] 
\itemsep-0.1em
\item\emph{$\{p_0, p_1\}\subseteq\textnormal{dom}(\phi)\subseteq N(q_0)\cup N(q_1)$ and $N(q_i)\cap\textnormal{dom}(\phi)\subseteq\{p_0, p_1\}$}; AND
\item\emph{For any two (not necessarily distinct) extensions of $\phi$ to $L$-colorings $\psi$ and $\psi'$ of $\textnormal{dom}(\phi)\cup\{q_0, q_1\}$, at least one of the following holds.}
\begin{enumerate}[label=\emph{\alph*)}] 
\itemsep-0.1em
\item\emph{$\mathcal{C}_G^P(\psi)\cap\mathcal{C}_G^P(\psi')\neq\varnothing$}; OR
\item\emph{At least one of $\mathcal{C}_G^P(\psi), \mathcal{C}_G^P(\psi')$ has size at least two}; OR
\item \emph{At least one of $\mathcal{C}_G^P(\psi), \mathcal{C}_G^P(\psi')$ is empty.}
\end{enumerate} 
\end{enumerate}

\emph{For each endpoint $p_i$ of $P$, we also define $\Omega_G^{\geq 2}(P, [p_i])$ to be the set of $\phi\in\Omega_G(P, [p_i])$ such that, for any extension of $\phi$ to an $L$-coloring $\psi$ of $\textnormal{dom}(\phi)\cup\{q_0, q_1\}$, $|\mathcal{C}(\psi)|\geq 2$.} \end{defn}

Note that, in the setting above, if $C$ is an induced cycle, then $\Omega_G(P, [p_0])=\Omega_G(P, [p_0])$, and each element of this common set is an $L$-coloring of $\{p_0, p_1\}$. Likewise for $\Omega_G^{\geq 2}(P, [p_0])$ and $\Omega_G^{\geq 2}(P, [p_1])$ In the case where $C$ is induced, we drop the second coordinate from the sets defined above and just write $\Omega_G(P)$ and $\Omega_G^{\geq 2}(P)$ respectively. The goal of Section \ref{Main4PathLemmaSec} is to prove that, under certain conditions, the set of colorings defined above in Definition \ref{PartialLColOmega} is nonempty. To do this, we first prove an intermediate result (Proposition \ref{InterMedUnobstrucGInduced}) which we need for the proof of Lemma \ref{EndLinked4PathBoxLemmaState}. To state Proposition \ref{InterMedUnobstrucGInduced}, we first introduce some additional terminology. 

\begin{defn}\label{DefnPPUnObstruc} \emph{Let $G$ be a graph and let $P=p_0q_0wq_1p_1$ be a 4-path in $G$.}
\begin{enumerate}[label=\emph{\arabic*)}]
\itemsep-0.1em
\item\emph{For each $i\in\{0,1\}$, we say that $G$ is $(P, p_i)$-\emph{unobstructed} if, for any $z\in V(G\setminus C)$ which is adjacent to all three of $p_i, q_i, w$, the vertices $z, q_{1-i}$ have no common neighbor in $G$ except for $w$. Otherwise, we say that $G$ is $(P, p_i)$-obstructed.}
\item\emph{Let $L$ be a list-assignment for $V(G)$. Given an $i\in\{0,1\}$ and a partial $L$-coloring $\sigma$ of $V(P)$, we say that $\sigma$ is $(L, P, p_i)$-\emph{blocked} if}
\begin{enumerate}[label=\emph{\roman*)}]
\itemsep-0.1em
\item\emph{there is a $z\in V(G\setminus C)$ adjacent to all three of $p_i, q_i, w$ and  $|L_{\sigma}(z)|=2$}; AND
\item\emph{$V(C)\neq V(P)$ and, letting $x$ be the unique neighbor of $p_i$ on the path $C\setminus\mathring{P}$,  $\{\sigma(q_i), \sigma(w)\}=L(z)\setminus L(x)$.}
\end{enumerate}
\end{enumerate}
\end{defn}

The remainder of Section \ref{BoxLemmaFor4PathSecPropIntermediate} consists of the proof of Proposition \ref{InterMedUnobstrucGInduced}. 

\begin{prop}\label{InterMedUnobstrucGInduced} Let $(G, C, P, L)$ be a rainbow, where $P=p_0q_0wq_1p_1$ is a 4-path and $G$ is a short-separation-free graph in which every face, except for $C$, is bounded by a triangle. Suppose that $|L(w)|\geq 5$ and that $C$ is an induced cycle, and suppose further that $N(p_0)\cap N(p_1)\subseteq V(C)$ and $N(q_0)\cap N(q_1)\subseteq V(C)$. Let $I=\{i\in\{0,1\}: G\ \textnormal{is}\ (P, p_i)\textnormal{-unobstructed}\}$. Then all three of the following hold.
\begin{enumerate}
\itemsep-0.1em
\item [\mylabel{}{P1)}]  For any $i\in I$ and $L$-coloring $\sigma$ of $V(P-p_{1-i})$ and any $T\subseteq L(p_{1-i})\setminus\{\sigma(q_{1-i})\}$ with $|T|=2$, either $\sigma$ is $(L, P, p_i)$-blocked or $\sigma$ extends to an $L$-coloring $\sigma'$ of $G$ with $\sigma'(q_i)\in T$; AND
\item [\mylabel{}{P2)}]  Let $\sigma, \sigma'$ be a pair of distinct $L$-colorings of $V(P)$ which restrict to the same $L$-coloring of $V(P-w)$, where neither $\sigma$ nor $\sigma'$ extends to an $L$-coloring of $G$. Then $|I|\neq 1$, and furthermore, if $(G, C, P, L)$ is also end-linked, then at least one of the following holds.
\begin{enumerate}[label=\roman*)]
\item $I=\varnothing$; OR
\item $\Omega_G^{\geq 2}(P)\neq\varnothing$; OR
\item For each $\sigma^*\in\{\sigma, \sigma'\}$, there is a $p\in\{p_0, p_1\}$ such that $\sigma^*$ is $(L, P, p)$-blocked.
\end{enumerate}
\item [\mylabel{}{P3)}] If $I=\{0,1\}$ and $\psi, \psi'$ is a pair of $L$-colorings of $V(P-w)$ which differ on precisely one vertex of $\{q_0, q_1\}$ and are otherwise equal, then either $\mathcal{C}_G^P(\psi)\cap\mathcal{C}_G^P(\psi')\neq\varnothing$, or there is a $\psi^*\in\{\psi, \psi'\}$ with $|\mathcal{C}_G^P(\psi^*)|\geq 2$. 
\end{enumerate}
 \end{prop}

\begin{proof} We break the proof of into five subsections. Before proving P1)-P3), we gather a few facts.

\makeatletter
\renewcommand{\thetheorem}{\thesubsection.\arabic{theorem}}
\@addtoreset{theorem}{subsection}
\makeatother

\subsection{Preliminary facts and definitions}

We may suppose, by removing some colors from some lists if necessary, that $|L(v)|=3$ for each $v\in V(C\setminus P)$ and $|L(v)|=5$ for each $v\in V(G\setminus C)$. Because of condition 2c) of Definition \ref{PartialLColOmega}, we have to be careful when doing this if we are proving a statement about $\Omega_G(P, [p_i])$ for some endpoint $p_i$ of $P$, but if we are only considering the subset $\Omega_G^{\geq 2}(P, [p_i])$, as in the statement of Proposition \ref{InterMedUnobstrucGInduced}, then this is permissible. To avoid clutter, for any partial $L$-coloring $\psi$ of $G-w$, we write $\mathcal{C}(\psi)$ in place of $\mathcal{C}_P^G(\psi)$. We now introduce the following notation, which is illustrated in Figure \ref{GZ0AndGZ1Figure}. By assumption, no vertex of $G\setminus C$ is adjacent to both of $q_0, q_1$, and, for each $i\in\{0,1\}$, there is at most one vertex of $G\setminus C$ adjacent to all three of $p_i, q_i, w$, so everything in Definition \ref{DefGZForEAchSide} is well-defined.

\begin{defn}\label{DefGZForEAchSide} \emph{For each $i\in\{0,1\}$ and any $z\in V(G\setminus C)$ which is adjacent to each of $p_i, q_i, w$, if such a $z$ exists, we define the following.}
\begin{enumerate}[label=\emph{\arabic*)}] 
\itemsep-0.1em
\item\emph{We let $p^{z}$ be the unique vertex of $N(z)\cap V(C\setminus\mathring{P})$ which is farthest from $p_i$ on the path $C\setminus\mathring{P}$.}
\item\emph{We let $G_{z}$ be the subgraph of $G$ bounded by outer face $(p_i\ldots p^{z})z$ and $H_{z}$ be the subgraph of $G$ bounded by outer cycle $(p^{z}(C\setminus\mathring{P})u_{1-i}^{\star})q_{1-i}wz_i$. Possibly $p^{z}=p_i$ and $G_{z_i}$ is an edge.}
\item\emph{For each $a\in L(p_i)$, we define a set $S^i_a\subseteq L(z)\setminus\{a\}$, where $S^i_a=L(z)\setminus\{a\}$ if $G_{z}$ is an edge, and otherwise $S^i_a$ is the set of $b\in L(z)\setminus\{a\}$ such that the $L$-coloring $(b, a)$ of $z_ip_i$ extends to at least two different $L$-colorings of $G_{z_i}$ which use different colors on $p^{z_i}$.}
\end{enumerate}
 \end{defn}

\begin{Claim}\label{SetSetAzAtLeastTwoCL} For each $i\in\{0,1\}$ and $a\in L(p_i)$, and any $z\in V(G\setminus C)$ which is adjacent to each of $p_i, q_i, w$, $|S_a^i|\geq 2$. \end{Claim}

\begin{claimproof} This is immediate if $G_z$ is an edge, and, if not, then, it just follows from 2ii) of Corollary \ref{CorMainEitherBWheelAtM1ColCor}. \end{claimproof}

\begin{Claim}\label{LPPkBlockedIF} Let $k\in\{0,1\}$, let $\tau$ be a partial $L$-coloring of $V(P)$, and let $c=\tau(p_k)$. If there exists a $z\in V(G\setminus C)$ with $N(z)\cap\textnormal{dom}(\tau)=\{p_k, q_k, w\}$, where $L_{\tau}(w)|=2$ and $\{\tau(q_k), \tau(w)\}=S_c^k$, then $\tau$ is $(L, P, p_k)$-blocked.  \end{Claim}

\begin{claimproof} Say for the sake of definiteness that $k=0$. By assumption, no neighbor of $p_0, q_0, w$ in $G\setminus C$ has no other neighbors in $P$. Let $x$ be the unique neighbor of $p_0$ on the path $C\setminus\mathring{P}$. Since $|S_c^0|<4$, $G_z$ is not an edge and, by 2i) of Corollary \ref{CorMainEitherBWheelAtM1ColCor}, $G_z$ is a broken wheel with principal path $p_0zp^z$, and, in particular, $|L(x)|=3$ and $c\in L(x)$. Now, we have $L(z)\setminus L(x)\subseteq S_c^0$, and since $|S_c^0|=2$ it follows that $\{\tau(q_k), \tau(w)\}$ is precisely equal to $L(z)\setminus L(x)$, so $\tau$ is $(L, P, p_0)$-blocked. \end{claimproof}

\begin{center}\begin{tikzpicture}

\node[shape=circle,draw=black] (p0) at (-4, 0) {$p_0$};
\node[shape=circle,draw=white] (mid) at (-3, 0) {$\ldots$};
\node[shape=circle,draw=black] (pz) at (-1.5, 0) {$p^{z_0}$};
\node[shape=circle,draw=white] (mid+) at (0, 0) {$\ldots$};
\node[shape=circle,draw=black] (pz1) at (1.5, 0) {$p^{z_1}$};
\node[shape=circle,draw=white] (un+) at (3, 0) {$\ldots$};
\node[shape=circle,draw=black] (p1) at (4, 0) {$p_1$};
\node[shape=circle,draw=black] (q0) at (-3,2) {$q_0$};
\node[shape=circle,draw=black] (q1) at (3,2) {$q_1$};
\node[shape=circle,draw=black] (w) at (0,4) {$w$};
\node[shape=circle,draw=black] (z0) at (-1.9, 1.9) {$z_0$};
\node[shape=circle,draw=white] (Gz) at (-2.1, 0.8) {$G_{z_0}$};
\node[shape=circle,draw=black] (z1) at (1.9, 1.9) {$z_1$};
\node[shape=circle,draw=white] (Gz1) at (2.1, 0.8) {$G_{z_1}$};
\node[shape=circle,draw=white] (H) at (0, 0.8) {$H_{z_0}\cap H_{z_1}$};
 \draw[-] (p1) to (un+);
 \draw[-] (p0) to (q0) to (w) to (q1);
\draw[-] (q1) to (p1);
\draw[-]  (p0) to (mid) to (pz) to (mid+) to (pz1) to (un+);
\draw[-] (w) to (z0) to (p0);
\draw[-] (w) to (z1) to (p1);
\draw[-] (q0) to (z0) to (pz);
\draw[-] (q1) to (z1) to (pz1);
\end{tikzpicture}\captionof{figure}{}\label{GZ0AndGZ1Figure}\end{center}

\begin{Claim}\label{IndexIExtendColorFromS}  let $\tau$ be an $L$-coloring of $V(P-w)$, let $c_0=\tau(p_0)$ and $c_1=\tau(p_1)$. Suppose that, for each $i\in\{0,1\}$, there is a $z_i\in V(G\setminus C)$ adjacent to all of $p_i, q_i, w$. For each $s_0\in S_{c_0}^0\setminus\{\tau(q_0)\}$ and $s_1\in S_{c_1}^1\setminus\{\tau(q_1)\}$ and $s'\in L_{\tau}(w)\setminus\{s_0, s_1\}$, at least one of the following holds.
\begin{enumerate}[label=\alph*)] 
\itemsep-0.1em
\item $\tau$ extends to an $L$-coloring of $G$ using $s_0, s', s_1$ on the respective vertices $z_0, w, z_1$; OR
\item There is a vertex $y\in V(G\setminus C)$ adjacent to all three of $z_0, w, z_1$, where $|L(y)\setminus\{s_0, s', s_1\}|=2$, and, in particular, $\tau$ extends to an $L$-coloring of $G-w$ using $s_0, s_1$ on the respective vertices $z_0, z_1$. 
\end{enumerate}
\end{Claim}

\begin{claimproof} Let $\tau'$ be an extension of $\tau$ to an $L$-coloring of $V(P)\cup\{z_0, z_1\}$ using $s_0, s', s_1$ on the respective vertices $z_0, w, z'$. Suppose that $\tau'$ does not extend to an $L$-coloring of $G$. Let $H=H_{z_0}\cap H_{z_1}$ and $Q=p^{z_0}z_0wz_1p^{z_1}$. Let $L^*$ be a list-assignment for $H$, where $L^*=L$ outside of $V(Q)$, the vertices $z_0, w, z_1$ are $L^*$-precolored with the respective colors $s_0, s', s_1$, and, for each $i\in\{0,1\}$ if $G_{z_i}$ is an edge, then $L^*(z_i)=\{c_i\}$. If $G_{z_i}$ is not an edge and $i=0$, we set $L^*(z_0)=\{s_0\}\cup\Lambda_{G_{z_0}}(c_0, s_0, \bullet)$, and, if $G_{z_i}$ is not an edge and $i=1$, we set $L^*(z_1)=\{s_1\}\cup\Lambda_{G_{z_1}}(\bullet, s_1, c_1)$. In particular, if $G_{z_i}$ is not an edge, then $|L^*(z_i)|\geq 3$. Let $\sigma$ be the unique $L^*$-coloring of the subpath of $Q$ consisting of the vertices of $Q$ with $L^*$-lists of size one. Since $\tau'$ does not extend to an $L$-coloring of $G$, $H$ is not $L^*$-colorable, and, in particular, since the outer cycle of $H$ is induced, it follows from Lemma \ref{PartialPathColoringExtCL0} that there is a vertex $y$ of $H\setminus C^H$ with $L^*_{\sigma}(y)|<3$, and at least three neighbors in $\textnormal{dom}(\sigma)$. For each $i\in\{0,1\}$, if $G_{z_i}$ is an edge, then $y\not\in N(w)\cap N(p^{z_i})$, since $G$ is $K_{2,3}$-free. Furthermore, since $I\neq\varnothing$, $y$ is adjacent to at most one of $p_0, p_1$. In particular, since $|N(w)\cap\textnormal{dom}(\sigma)|\geq 3$, we have $w\in N(y)$, and since $G$ has no induced 4-cycles and there are no chords of the outer cycle of $H$, we have $y\in N(z_0)\cap N(z_1)$ as well. It follows that, for each $i\in\{0,1\}$, if $G_{z_i}$ is an edge, then $y\not\in N(p^{z_i})$, so $N(y)\cap\textnormal{dom}(\sigma)=\{z_0, w, z_1\}$ and $|L(y)\setminus\{s_0, s', s_1\}|=2$. Uncoloring $w$, the argument above also shows that $\tau$ extends to an $L$-coloring of $G-w$ using $s_0, s_1$ on the respective vertices $z_0, z_1$, since $y$ is the unique vertex of $H\setminus C^H$ with more than two neighbors in $\textnormal{dom}(\sigma)$. \end{claimproof}

\begin{Claim}\label{TauObstructedByZiOnAtLeastOne} Let $k\in I$ and let $\tau$ be an $L$-coloring of $V(P)$. Then at least one of the following holds.
\begin{enumerate}[label=\alph*)] 
\itemsep-0.1em
\item $\tau$ extends to an $L$-coloring of $G$; OR
\item $\tau$ is $(L, P, p_k)$-blocked; OR
\item There is a $z\in V(G\setminus C)$ with $N(z)\cap V(P)=\{p_{1-k}, q_{1-k}, w\}$.
\end{enumerate}
\end{Claim}

\begin{claimproof} For each $i\in\{0,1\}$, let $c_i=\tau(p_i)$. Suppose for the sake of definiteness that $k=0$, i.e $0\in I$. Suppose that $\tau$ does not extend to an $L$-coloring of $G$, but $\tau$ is also not $(L, P, p_0)$-blocked. It suffices to show that there is a vertex of $G\setminus C$ adjacent to both of $p_1, w$. Any such vertex, if it exists, is also adjacent to $q_1$ (as $G$ has no induced 4-cycles and $C$ has no chords) and has no other neighbors in $P$, since $I\neq\varnothing$. Suppose toward a contradiction that $p_1, w$ have no common neighbor in $G\setminus C$. Since $C$ is an induced cycle and $\tau$ does not extend to an $L$-coloring of $G$, it follows from Lemma \ref{PartialPathColoringExtCL0} that there is a $y\in V(G\setminus C)$ with $|L_{\tau}(y)|<3$ and since $I\neq\varnothing$, this $y$ is unique and $G[N(y)\cap V(P)]=p_0q_0w$. In particular, $|L_{\tau}(y)|=2$. By Claim \ref{SetSetAzAtLeastTwoCL}, $|S^0_{c_0}|\geq 2$, and, since $\tau$ is not $(L, P, p_0)$-blocked, it follows from Claim \ref{LPPkBlockedIF} that there is an $s\in S^0_{c_0}\setminus\{\tau(q_0), \tau(w)\}$. Consider the following cases.

\textbf{Case 1:} $G_{z_0}$ is an edge

In this case, the outer cycle of $G-q_0$ is induced. We let $Q$ be the 4-path $p_0z_0wz_1p_1$ on the outer cycle of $G-q_0$ and let $\tau'$ be the $L$-coloring $(c_0, s, \tau(w), \tau(q_1), c_1)$ of $p_0z_0wz_1p_1$. By assumption, no vertex of $G\setminus C$ is adjacent to both of $p_1, w$, and since $0\in I$ and $G$ has no copies of $K_{2,3}$, it follows that no vertex of $G-q_0$ outside of $C^{G-q_0}$ has more than two neighbors in $Q$. But then, by Lemma \ref{PartialPathColoringExtCL0}, $\tau'$ extends to an $L$-coloring of $G-q_0$, so $\tau$ extends to an $L$-coloring of $G$, which is false.

\textbf{Case 2:} $G_{z_0}$ is not an edge.

In this case, we let $L'$ be a list-assignment for $V(H_{z_0})$, where $L'(p^{z_0})=\{s\}\cup\Lambda_{G_{z_0}}(c_0, s, \bullet)$ and the vertices $z_0, w, q_1, p_1$ are precolored with the respecrtive colors $s, \tau(w), \tau(q_1), \tau(p_1)$, and otherwise $L'=L$. Note that $|L'(p^{z_0})|\geq 3$. Since $s'\not\in\mathcal{C}(\tau)$, $H_{z_0}$ is not $L'$-colorable, and since $0\in I$, it follows from Lemma \ref{PartialPathColoringExtCL0} applied to $H_{z_0}$ that there is a vertex of $G\setminus C$ adjacent to all of $w, q_1, p_1$, contradicting our assumption. \end{claimproof}

\subsection{Proof of P1)}

We now prove P1) of Proposition \ref{InterMedUnobstrucGInduced}. Let $i\in I$, say $i=0$ without loss of generality. Suppose that $\sigma$ does not extend to an $L$-coloring of $G$ using a color of $T$ on $p_1$. We  just need to show that $\sigma$ is $(L, P, p_0)$-blocked. Let $L'$ be a list-assignment for $G$ where $L'(p_1)=\{\sigma(q_1)\}\cup T$ and otherwise $L'=L$. Since $G$ is not $L'$-colorable and $L'(p_1)|=3$, and since $C$ is an induced cycle, it follows from Lemma \ref{PartialPathColoringExtCL0} that there is a $z_0\in V(G\setminus C)$ with $|L_{\sigma}(z_0)|<3$. Note that $G[N(z_0)\cap V(P-p_1)]$ is a subpath of $P-p_1$. Since $I\neq\varnothing$, we have $N(z_0)\cap V(P-p_1)=\{p_0, q_0, w\}$. Let $c_0=\sigma(p_0)$ and suppose toward a contradiction that $\sigma$ is not $(L, P, P_0)$-blocked. Let $c=\sigma(p_0)$. By Claim \ref{LPPkBlockedIF}, since $|S_c^0|\geq 2$, there is an $s\in S_c^0\setminus\{\sigma(q_0), \sigma(w)\}$. If $G_{z_0}$ is an edge, then, since $0\in I$, it follows from a second application of Lemma \ref{PartialPathColoringExtCL0} that the $L'$-coloring $(c, s, \sigma(w), \sigma(q_1))$ of $p_0z_0wp_1$ extends to an $L'$-coloring of $G-q_0$, so $G$ is $L'$-colorable, which is false. Thus, $G_{z_0}$ is not an edge. Let $L''$ be a list-assignment for $V(H_{z_0})$ where the vertices $z_0, w, q_1$ are precolored with the respective colors $s, \sigma, \sigma,$ and furthermore, $L''(p^{z_0})=\{s\}\cup\Lambda_{G_{z_0}}(c_0, s, \bullet)$ and otherwise $L''=L'$. Since $|L''(p^{z_0})|\geq 3$ and $0\in I$, another application of Lemma \ref{PartialPathColoringExtCL0} shows that $H_{z_0}$ is $L''$-colorable, so $G$ is $L'$-colorable, contradicting our assumption. This proves P1).

\subsection{Proof of P2): Part I}\label{SubSecFirstPartIneq1}

We now prove P2). Let $\sigma, \sigma'$ be as in the statement of P2) and let $\tau$ be the common restriction of $\sigma, \sigma'$ to $V(P-w)$. Let $\sigma(w)$ and $r'=\sigma'(w)$, and let $A=\{\sigma(w), \sigma'(w)\}$. By assumption, $|A|=2$ and $A\subseteq L_{\tau}(w)\setminus\mathcal{C}(\tau)$. Let $c_0=\tau(p_0)$ and $c_1=\tau(p_1)$. For each $i\in\{0,1\}$, let $x_i$ be the unique neighbor of $p_i$ on the path $C\setminus\mathring{P}$. Possibly $C$ is a 5-cycle and $x_i=p_{1-i}$ for each $i=0,1$. We first show that $|I|\neq 1$. The proof that $|I|\neq 1$ makes up all of Subsection \ref{SubSecFirstPartIneq1}.

\begin{Claim}\label{IClaimNotExactlyOneCL} $|I|\neq 1$. \end{Claim}

\begin{claimproof} Suppose toward a contradiction that $|I|=1$, say $I=\{0\}$ for the sake of definiteness. As $1\not\in I$, there is a vertex $z\in V(G\setminus C)$ adjacent to all of $w, q_1, p_1$, and $z, q_0$ have a common neighbor $y\in V(G)$ with $y\neq w$. Since $C$ is an induced cycle, $y\in V(G\setminus C)$, and since $N(q_0)\cap N(q_1)\subseteq V(C)$, it follows from our triangulation conditions that $N(w)=\{q_0, y, z, q_1\}$. As neither $\sigma$ nor $\sigma'$ extends to an $L$-coloring of $G-w$, we immediately have the following.

\vspace*{-8mm}
\begin{addmargin}[2em]{0em}
\begin{subclaim}\label{AnyExtTauUsesAOnEdgeSUBCL} Any extension of $\tau$ for an $L$-coloring of $V(G-w)$ uses the colors of $A$ on the endpoints of $yz$. \end{subclaim}\end{addmargin}

We now have the following.

\vspace*{-8mm}
\begin{addmargin}[2em]{0em}
\begin{subclaim} $p_0\not\in N(y)$ and $d(p_0, p^z)\geq 2$. Furthermore, $y\in N(p^z)$. \end{subclaim}

\begin{claimproof} Since $0\in I$, we have $y\not\in N(p_0)$. Now suppose that $d(p_0, p^z)<2$. By assumption, $p^z\neq p_1$, since $z\not\in N(p_0)\cap N(p_1)$.  Thus, $p_0p^z\in E(C)$, as $C$ is an induced cycle. If $|V(C)=5$, then, since $y\neq z$, it follows from Theorem \ref{BohmePaper5CycleCorList} that each of $\sigma, \sigma'$ extends to an $L$-coloring of $G$, which is false. Thus, $|V(C)|>5$, and $G_z$ is not an edge. Since $|L_{\tau}(z)|\geq 3$, it follows from 2) of Theorem \ref{BWheelMainRevListThm2} that there is an $s\in L_{\tau}(z)$ with $\Lambda_{G_z}(\bullet, s, c_1)\neq\{c_0\}$, so $\tau$ extends to an $L$-coloring $\tau^*$ of $V(P-w)\cup V(G_z)$. At least one color of $A$ is distinct from $\tau^*(z)$, and since neither $\sigma$ nor $\sigma'$ extends to $L$-color $G$, it follows from Theorem \ref{BohmePaper5CycleCorList} applied to $H_z$ that $y\in N(p_0)$, which is false. Thus, $d(p^0, p^z)\geq 2$. 

Now suppose toward a contradiction that $y\not\in N(p^z)$. Since $L_{\tau}(z)|\geq 3$, we choose an $s\in L_{\tau}(z)\setminus A$. Since $d(p_0, p^z)\geq 2$, $\tau$ extends to an $L$-coloring  $\pi$ of $V(P-w)\cup V(G_z)$ using $s$ on $z$. Note that $H_z-w$ is bounded by outer cycle $q_0yzp^z(C\setminus\mathring{P})p_0$. Since $|A\setminus\{\pi(z)\}|=2$, it follows that $\pi$ does not extend to $L$-color $H_z-w$. Since $y$ is adjacent to neither $p_0$ nor $p^z$, $|L_{\pi}(y)|\geq 3$. Applying Lemma \ref{PartialPathColoringExtCL0} to $H_z-w$ and the 4-path $p_0q_0yzp^z$ on the outer cycle of $H_z-w$, it follows that there is a vertex $y'$ which does not lie on the outer face of $H_z-w$ and is adjacent to at least three vertices of $\{p_0, q_0, z, p^z\}$. Since $G$ is $K_{2,3}$-free, this vertex is adjacent to both of $p_0, p^z$ and precisely one of $q_0, z$. Thus, $G$ contains a 5-cycle $F$ which contains $y$ in its interior, where $F$ is either $wzp^zy'q_0$ or $wzy'p_0q_0$. 

Since $|L_{\tau}(w)|\geq 3$, and since $C$ is an induced cycle with $N(p_0)\cap N(p_1)\subseteq V(C)$ and $N(q_0)\cap N(q_1)\subseteq V(C)$, it follows from Lemma \ref{PartialPathColoringExtCL0} applied to $G$ that $\tau$ extends to an $L$-coloring of $G$. Thus, $\tau$ extends to an $L$-coloring $\tau^+$ of $V(\textnormal{Ext}(F))\setminus\{w\}$. Since $y$ is not yet colored, $L_{\tau^+}(w)\setminus A\neq\varnothing$, so there is an $L$-coloring of $\textnormal{Ext}(F)$ which does not extend to $L$-color the interior of $F$. But then, again by Theorem \ref{BohmePaper5CycleCorList}, $y$ is adjacent to all five vertices of $F$, so $y$ is adjacent to either $p_0$ or $p^z$. Since $y\not\in N(p_0)$, we have $y\in N(p_z)$, contradicting our assumption. \end{claimproof}\end{addmargin}

 Since $G$ is $K_{2,3}$-free and $y\in N(p^z)$, $G_z$ is not an edge. We now consider the graph $H':=H_{z_1}\setminus\{w, z_1\}$. This has outer cycle $p_0(C\setminus\mathring{P})p^{z_1}yq_0$, and every chord of the outer cycle of $H'$ is incident to $y$. This is illustrated in Figure \ref{GZANDY'WheelISizeOne}. 

\vspace*{-8mm}
\begin{addmargin}[2em]{0em}
\begin{subclaim}\label{ColH'+yzpz} For each $s\in S_{c_1}^1\setminus\{\tau(q_1)\}$, $\tau$ extends to an $L$-coloring of $G-w$ using $s$ on $z$. \end{subclaim}

\begin{claimproof} Firstly, $H'+yzp^z$ is bounded by outer cycle $(C^{H'}-yp^z)+yzp^z$. Let $L^*$ be a list-assignment for $H'+yzp^z$ where $L(p^z)=\{s\}\cup\Lambda_{G_z}(\bullet, s, c_1)$ and otherwise $L^*=L$. Let $\psi$ be the $L^*$-coloring of $p_0, q_0, z$ using $\tau(p_0), \tau(q_0), s$ on the respective vertices $p_0, q_0, z$. Note that every chord of the outer cycle of $H'+yzp^z$ is incident to $y$, and since $yp_0\not\in E(G)$, we have $|L^*_{\psi}(y)|\geq 3$. Since $|L^*(p^z)|\geq 3$, it follows from Lemma \ref{PartialPathColoringExtCL0} that $\psi$ extends to an $L^*$-coloring of $H'+yzp^z$, and, in particular, this $L^*$-coloring uses a color of $\Lambda_{G_z}(\bullet, s, c_0)$ on $p^z$, so $\tau$ extends to an $L$-coloring of $G-w$ using $s$ on $z$. \end{claimproof}\end{addmargin}

\begin{center}\begin{tikzpicture}

\node[shape=circle,draw=black] [label={[xshift=0.0cm, yshift=-1.3cm]\textcolor{red}{$\{c_0\}$}}] (p0) at (-4, 0) {$p_0$};
\node[shape=circle,draw=white] (mid) at (-2, 0) {$\ldots$};
\node[shape=circle,draw=black] (pz) at (0, 0) {$p^z$};
\node[shape=circle,draw=white] (un+) at (2, 0) {$\ldots$};
\node[shape=circle,draw=black] [label={[xshift=0.0cm, yshift=-1.3cm]\textcolor{red}{$\{c_1\}$}}] (p1) at (4, 0) {$p_1$};
\node[shape=circle,draw=black] [label={[xshift=-1.1cm, yshift=-0.6cm]\textcolor{red}{$\{\tau(q_0)\}$}}] (q0) at (-3,2) {$q_0$};
\node[shape=circle,draw=black] [label={[xshift=1.1cm, yshift=-0.6cm]\textcolor{red}{$\{\tau(q_1)\}$}}] (q1) at (3,2) {$q_1$};
\node[shape=circle,draw=black] (w) at (0,4) {$w$};
\node[shape=circle,draw=black] (z) at (1.9, 1.9) {$z$};
\node[shape=circle,draw=white] (Gz) at (2.1, 0.8) {$G_{z}$};
\node[shape=circle,draw=white] (H') at (-1.8, 0.9) {$H'$};
\node[shape=circle,draw=black] (y) at (0, 2.2) {$y$};
 \draw[-] (p1) to (un+);
 \draw[-] (p0) to (q0) to (w) to (q1);
\draw[-] (q1) to (p1);
\draw[-]  (p0) to (mid) to (pz) to (un+);
\draw[-]  (y) to (pz);
\draw[-] (w) to (z) to (p1);
\draw[-] (q1) to (z) to (pz);
\draw[-] (q0) to (y) to (w);
\draw[-] (y) to (z);
\draw[-] (y) to (pz);
\end{tikzpicture}\captionof{figure}{}\label{GZANDY'WheelISizeOne}\end{center}

\vspace*{-8mm}
\begin{addmargin}[2em]{0em}
\begin{subclaim}\label{TwoColorWorksContainedSC1} $S_{c_1}^1\setminus\{\tau(q_1)\}\subseteq A$. In particular, $G_z$ is a broken wheel with principal path $p^zzp_1$. \end{subclaim}

\begin{claimproof} If $S_{c_1}^1\setminus\{\tau(q_1)\}\not\subseteq A$, then, by Subclaim \ref{ColH'+yzpz}, $\tau$ extends to an $L$-coloring of $G-w$ in which $z$ is colored with a color not in $A$, contradicting Subclaim \ref{AnyExtTauUsesAOnEdgeSUBCL}. Thus, $S_{c_1}^1\setminus\{\tau(q_1)\}\subseteq A$. Choosing an $s\in L(z)\setminus (A\cup\{c_1, \tau(q_1)\})$, we have $|\Lambda_{G_z}(\bullet, s, c_1)|=1$, so, by 2) i) of Corollary \ref{CorMainEitherBWheelAtM1ColCor}, $G_z$ is a broken wheel with principal path $p^zzp_1$.\end{claimproof}\end{addmargin}

We now have the following key subclaim. 

\vspace*{-8mm}
\begin{addmargin}[2em]{0em}
\begin{subclaim}\label{BoxInColoringOFH'} There is an $L$-coloring of $\{p_0, q_0, p^z\}$ which does not extend to an $L$-coloring of of $H'$. \end{subclaim}

\begin{claimproof} Suppose toward a contradiction that Subclaim \ref{BoxInColoringOFH'} does not hold. Since $d(p_0, p^z)\geq 2$, it follows that, for each $d\in L(p^z)$, there is an $L$-coloring $\pi_d$ of $V(H')$ using $\tau(p_0), \tau(q_0), d$ on the respective vertices $p_0, q_0, p^z$. We now have the following intermediate facts.

\begin{enumerate}[label=\Alph*)]
\itemsep-0.1em
\item For each $d\in L(p^z)$, $\Lambda_{G_z}(d, \bullet, c_1)$ is a subset of either $\{\tau(q_1)\}\cup A$ or $\{\tau(q_1), \pi_d(y)\}$ 
\item $G_z$ is not a triangle
\item $|L_{\tau}(z)|\setminus A|=1$. In particular, $A\subseteq L_{\tau}(z)$.
\end{enumerate}

We first prove A). Let $d\in L(p^z)$. and suppose that $\Lambda_{G_z}(d, \bullet, c_1)\not\subseteq\{\tau(q_1)\}\cup A$. Thus, there is an $s\in\Lambda_{G_z}(d, \bullet, c_1)$ with $s\not\in\{\tau(q_1)\}\cup A$, so $s=\pi_d(y)$, or else we contradict Subclaim \ref{AnyExtTauUsesAOnEdgeSUBCL}. Since $\pi_d(y)\not\in A$, we have $\Lambda_{G_z}(d, \bullet, c_1)\subseteq\{s, \tau(q_1)\}$, or else $\pi_d\cup\tau$ extends to an $L$-coloring of $G-w$ using $s$ on $y$, which, again, contradicts Subclaim \ref{AnyExtTauUsesAOnEdgeSUBCL}. This proves A).

Suppose that $G_z$ is a triangle. Since $|L_{\tau}(z)|\geq 3$, there is an $s\in L_{\tau}(z)\setminus A$, and since $|L_{\tau}(p^z)|\geq 2$, there is an $d\in L_{\tau}(p^z)$ with $d\neq s$. Since $G_z$ is a triangle, we have $s\in\Lambda_{G_z}(d, \bullet, c_1)$, so, by A), $\pi_d(y)=s$. Furthermore, $\tau\cup\pi_d$ leaves a color for $z$, and since $G_z$ is a triangle, it follows that $\tau\cup\pi_d$ extends to an $L$-coloring of $G-w$ using $s$ on $y$, contradicting Subclaim \ref{AnyExtTauUsesAOnEdgeSUBCL}. This proves B).

Suppose C) does not hold and let $s, s'$ be distinct colors of $L_{\tau}(z)\setminus A$. By Subclaim \ref{TwoColorWorksContainedSC1}, $|\Lambda_{G_z}(\bullet, s, c_1)|=|\Lambda_{G_z}(\bullet, s', c_1)|=1$. Since $G_z$ is not a triangle, $L(x_1)=\{c_1, s, s'\}$. If $\Lambda_{G_z}(\bullet, s, c_1)=\Lambda_{G_z}(\bullet, s', c_1)=\{d\}$ for some color $d$, then, since $\pi_d(y)$ is distinct from at least one of $s, s'$, it follows that $\tau$ extends to an $L$-coloring of $G-w$ in which $z$ is colored with one of $s, s'$, contradicting Subclaim \ref{AnyExtTauUsesAOnEdgeSUBCL}. Thus, by 1b) of Theorem \ref{BWheelMainRevListThm2}, $\Lambda_{G_z}(\bullet, s, c_1)=\{s'\}$ and $\Lambda_{G_z}(\bullet, s', c_1)=\{s\}$. In particular, $s, s'\in L(p^z0)$, and, by A), $\tau_s(y)=s'$ and $\tau_{s'}(y)=s$. Now we choose an $r\in L_{\tau}(w)\setminus\{s, s'\}$. Since $L(x_1)=\{c_1, s, s'\}$, we have $r\not\in L(x_1)$, so $\tau\cup\pi_s$ extends to an $L$-coloring of $yz$ using $s', r$ on the respective vertices $y, z$, contradicting Subclaim \ref{AnyExtTauUsesAOnEdgeSUBCL}. This proves C). Applying C), we fix $s$ as the lone color of $L_{\tau}(z)\setminus A$. Since $|\Lambda_{G_z}(\bullet, s, c_1)|=1$, we have $s\in L(p^z)$ and $s, c_1\in L(x_1)$. Let $A^*=\{a\in A: \Lambda_{G_z}(\bullet, a, c_1)=L(p^z)\setminus\{a\}\}$. Since $s, c_1\in L(x_1)$, we have $A\setminus L(x_1)\neq\varnothing$, and since $G_z$ is not a triangle, there is an $a\in A$ with $\Lambda_{G_z}(\bullet, a, c_1)=L(p^z)\setminus\{a\}$. In particular, $A^*\neq\varnothing$. We now have one more intermediate fact. 

\begin{enumerate}
\item  [\mylabel{}{\textnormal{D)}}] $|A^*|=1$, and, in particular, $\Lambda_{G_z}(\bullet, s, c_1)=A^*$
\end{enumerate}

Let $s'$ be the lone color of $\Lambda_{G_z}(\bullet, s, c_1)$. Suppose that D) does not hold. Thus, there is an $a\in A^*$ with $a\neq s'$. By A), $\pi_{s'}(y)=s$, and since $s'\neq a$ and $a\in A^*$, we have $s'\in\Lambda_{G_z}(\bullet, a, c_1)$, so it follows that $\tau\cup\pi_{s'}$ extends to an $L$-coloring of $G-w$ using $s, a$ on the respective vertices $y, z$, contradicting Subclaim \ref{AnyExtTauUsesAOnEdgeSUBCL}. This proves D). We now let $a=\sigma(w)$ and $a'=\sigma'(w)$. We now suppose for the sake of definiteness that $a\in A\setminus L(x_1)$, so $A^*=\{a\}$, and $a'\not\in A^*$. Thus, $L(x_1)=\{c_1, s, a'\}$. By D), $a\in L(p^z)$ and, by A), $\pi_a(y)=s$. Since $\pi_a(y)=s$, we have $a\not\in\Lambda_{G_z}(\bullet, a', c_1)$, or else $\pi_a$ extends to an $L$-coloring of $G-w$ using $s, a'$ on the respective vertices $y,z$, contradicting Subclaim \ref{AnyExtTauUsesAOnEdgeSUBCL}. We need one more intermediate fact to finish the proof of Subclaim \ref{BoxInColoringOFH'}.

\begin{enumerate}
\item  [\mylabel{}{\textnormal{E)}}] $L(p^z)=A\cup\{s\}$
\end{enumerate}

Suppose not. Since $a, s\in L(p^z)$, there is a $b\in L(p^z)\setminus\{a, s\}$ with $b\neq a'$. Let $x_1'$ be the unique vertex of $G_z-z$ adjacent to $p^z$. Since $a\not\in\Lambda_{G_z}(\bullet, a', c_1)$ and $|\Lambda_{G_z}(\bullet, s, c_1)|=1$, we have $L(x_1')=\{a, a', s\}$, so $b\not\in L(x_1')$. Since $L_{\tau}(z)=A\cup\{s\}$, it follows that $\Lambda(b, \bullet, c_1)=A\cup\{s\}$. Since at least one color of $A$ is distinct from $\pi_b(y)$, it follows that one of $\sigma, \sigma'$ extends to an $L$-coloring of $G$, which is false. This proves E). In particular, we have $a'\in L(p^z)$. Since $\Lambda_{G_z}(\bullet, s, c_1)=\{a\}$ and $a\not\in\Lambda_{G_z}(\bullet, a', c_1)$, we conclude that $(a', s, c_1)$ and $(a, a', c_1)$ are two $L$-colorings of $p^zzp_1$ which do not extend to $L$-color $G$. That is, $\Lambda_{G_z}(\bullet, s, c_1)=\{a\}$ and $\Lambda_{G_z}(\bullet, a', c_1)=\{s\}$. As $a, a', s$ are all distinct, this contradicts 1b) of Theorem \ref{BWheelMainRevListThm2}. This completes the proof of Subclaim \ref{BoxInColoringOFH'}. \end{claimproof}\end{addmargin}

Applying Subclaim \ref{BoxInColoringOFH'}, we have the following. 

\vspace*{-8mm}
\begin{addmargin}[2em]{0em}
\begin{subclaim}\label{y'AdjFourVerticesH'} There is a $y'\in V(G\setminus C)$ adjacent to all four of $p_0, q_0, y, p^{z_1}$. \end{subclaim}

\begin{claimproof} It suffices to show that there is a $y'\in V(G\setminus C)$ adjacent to $p_0, q_0, p^{z}$. If that holds, then, applying the fact that $G$ has no induced 4-cycles, $y'\in N(y)$ as well, again, as $G$ is $K_{2,3}$-free. By Subclaim \ref{BoxInColoringOFH'}, there is an $L$-coloring $\psi$ of $\{p_0, q_0, p^z\}$ which does not extend to an $L$-coloring of $H'$. Since $yp_0\not\in E(G)$, we have $|L_{\psi}(y)|\geq 3$, and since every chord of the outer cycle of $H'$ is incident to $y$, it follows  from Subclaim \ref{BoxInColoringOFH'} applied to $H'$ that there is a $y'\in V(G\setminus C)$ adjacent to $p_0, q_0, p^{z}$ \end{claimproof}\end{addmargin}

Let $y'$ be as in Subclaim \ref{y'AdjFourVerticesH'}. Let $G'=H'-y$ be the subgraph of $G$ bounded by outer cycle $p_0(C\setminus\mathring{P})p^{z_1}y'$.  In particular $N(y)=\{w, q_0, y', p^{z_1}, z_1\}$, and $y$ is the central vertex of a wheel. This is illustrated in Figure \ref{GZANDY'WheelISizeExtraVertY'}. 

\begin{center}\begin{tikzpicture}
\node[shape=circle,draw=black] [label={[xshift=0.0cm, yshift=-1.3cm]\textcolor{red}{$\{c_0\}$}}] (p0) at (-4, 0) {$p_0$};
\node[shape=circle,draw=white] (mid) at (-2, 0) {$\ldots$};
\node[shape=circle,draw=black] (pz) at (0, 0) {$p^z$};
\node[shape=circle,draw=white] (un+) at (2, 0) {$\ldots$};
\node[shape=circle,draw=black] [label={[xshift=0.0cm, yshift=-1.3cm]\textcolor{red}{$\{c_1\}$}}] (p1) at (4, 0) {$p_1$};
\node[shape=circle,draw=black] [label={[xshift=-1.1cm, yshift=-0.6cm]\textcolor{red}{$\{\tau(q_0)\}$}}] (q0) at (-3,2) {$q_0$};
\node[shape=circle,draw=black] [label={[xshift=1.1cm, yshift=-0.6cm]\textcolor{red}{$\{\tau(q_1)\}$}}] (q1) at (3,2) {$q_1$};
\node[shape=circle,draw=black] (w) at (0,4) {$w$};
\node[shape=circle,draw=black] (z) at (1.9, 1.9) {$z$};
\node[shape=circle,draw=white] (Gz) at (2.1, 0.8) {$G_{z}$};
\node[shape=circle,draw=white] (G') at (-1.8, 0.6) {$G'$};
\node[shape=circle,draw=black] (y) at (0, 2.2) {$y$};
\node[shape=circle,draw=black] (y') at (-1.8, 1.3) {\tiny $y'$};
 \draw[-] (p1) to (un+);
 \draw[-] (p0) to (q0) to (w) to (q1);
\draw[-] (q1) to (p1);
\draw[-]  (p0) to (mid) to (pz) to (un+);
\draw[-]  (y) to (pz);
\draw[-] (w) to (z) to (p1);
\draw[-] (q1) to (z) to (pz);
\draw[-] (q0) to (y) to (w);
\draw[-] (y) to (z);
\draw[-] (y) to (pz);
\draw[-] (q0) to (y') to (pz);
\draw[-] (p0) to (y') to (y);
\end{tikzpicture}\captionof{figure}{}\label{GZANDY'WheelISizeExtraVertY'}\end{center}

Applying Theorem \ref{thomassen5ChooseThm}, and the fact that $L_{\tau}(z)\setminus A\neq\varnothing$, we now fix an $L$-coloring $\pi$ of $G_z$ with $\pi(p_1)=\tau(p_1)=c_1$ and $\pi(z)\in L_{\tau}(z)\setminus A$. Possibly $\pi(p^z)=c_0$, but, as $G$ is short-separation-free, $G'$ is not a triangle, so $\tau\cup\pi$ is a proper $L$-coloring of its domain. As $\pi(z)\not\in A$, $\tau\cup\pi$ does not extend to $L$-color $G-w$, so we immediately get the following. 

\vspace*{-8mm}
\begin{addmargin}[2em]{0em}
\begin{subclaim}\label{Ltauy'NoIntersectionLambdaG'} $L_{\tau}(y')\cap\Lambda_{G'}(c_0, \bullet, \pi(p^z))=\varnothing$. \end{subclaim}
\end{addmargin}

Since $|S_{c_1}^1|\geq 2$, there is an $a\in L_{\tau}(z)$ with $\Lambda_{G_z}(\bullet, a, c_1)\neq\{\pi(p^z)\}$. In particular, $a\neq s$, and we fix a second $L$-coloring $\psi$ of $G_z$, where $\psi=\tau=c_1$, $\psi(z)=a$, and $\psi(p^z)\neq\pi(p^z)$. 

\vspace*{-8mm}
\begin{addmargin}[2em]{0em}
\begin{subclaim}\label{LambdaG'ThreeCOLLeftSubCL} $\Lambda_{G'}(c_0, \bullet, \psi(p^z))=L(y')\setminus\{c_0, \psi(p^z)\}$. \end{subclaim}

\begin{claimproof} Suppose not. Thus, there is an $L$-coloring of $p_0y'p^z$ which uses $c_0, \psi(p^z)$ on the respective vertices $p_0, p^z$ but does not extend to an $L$-coloring of $G'$. We now let $T=L_{\tau}(y')\setminus\{pi(p^z)\}$. Note that $|T|\geq 2$, and, by Subclaim \ref{Ltauy'NoIntersectionLambdaG'}, $T\cap\Lambda_{G'}(c_0, \bullet, \pi(p^z))=\varnothing$. Since $\pi(p^z)\neq\psi(p^z)$, it follows from 3) ii) of Corollary \ref{CorMainEitherBWheelAtM1ColCor} that $c_0=\psi(p^z)$ and that $c_0\in T$, which is false, since $c_0\not\in L_{\tau}(y')$. \end{claimproof}\end{addmargin}

By Subclaim \ref{LambdaG'ThreeCOLLeftSubCL}, $|\Lambda_{G'}(c_0, \bullet, \psi(p^z))\setminus\{\tau(q_0)\}|\geq 2$. Since at least one color of $A$ is distinct from $\psi(z)$, at least one of $\sigma, \sigma'$ extends to an $L$-coloring of $G$, contradicting our assumption. This proves Claim \ref{IClaimNotExactlyOneCL}. \end{claimproof}

\subsection{Proof of P2): Part II}

We now suppose that the rainbow $(G, C, P, L)$ is also end-linked and prove the rest of P2). Suppose that P2) does not hold, i.e none of P2) i)-iii) hold. In particular, since P2) i) does not hold and $|I|\neq 1$ by Claim \ref{IClaimNotExactlyOneCL}, we have $I=\{0,1\}$. Since P2) iii) does not hold, there is an element of $\{\sigma, \sigma'\}$, say $\sigma$, which is neither $(L, P, p_0)$-blocked or $(L, P, p_1)$-blocked. Applying Claim \ref{TauObstructedByZiOnAtLeastOne} to $\sigma$, since $I=\{0,1\}$, it follows that, for each $i\in\{0,1\}$, there is a $z_i\in V(G\setminus C)$ with $G[N(z_i)\cap V(P)]=p_iq_iw$. 

\begin{Claim}\label{ForEachAvoidSCiCol} For each $i\in\{0,1\}$, $S_{c_i}^i\not\subseteq\{\sigma(w), \tau(q_i)\}$ \end{Claim}

\begin{claimproof} Suppose that, for some $i\in\{0,1\}$, $S_{c_i}^i\subseteq\{\sigma(w), \tau(q_i)\}$. Since $|S_{c_i}^i|\geq 2$, it follows that $S_{c_i}^i=\{r, \tau(q_i)\}$, and furthermore, $G_{z_i}$ is not an edge, and $|L(z_i)\setminus (\{c_i\}\cup S_{c_i}^i)|\geq 2$. It follows from 1) of Theorem \ref{EitherBWheelOrAtMostOneColThm} that $G_{z_i}$ is a broken wheel with principal path $p_iz_ip^{z_i}$. Since $I\neq\varnothing$, $N(p_0)\cap N(p_1)\subseteq V(C)$, so each of $x_i$ is an internal vertex of the path $C\setminus\mathring{P}$. In particular, $L(z_i)\setminus (L(x_i)\setminus\{c_i\})\subseteq \{c_i\}\cup S_{c_i}^i$, so we have $|L_{\sigma}(z_i)|=2$ and $L(z_i)\setminus L(x_i)=\{r, \tau(q_i)\}$. But then, $\sigma$ is $(L, P, p_i)$-blocked, contradicting our assumption. \end{claimproof}

Iit follows from Claim \ref{ForEachAvoidSCiCol} that, for each $i\in\{0,1\}$, there exists an $s_i\in S_{c_i}^i\setminus\{r, \tau(q_i)\}$. Since $r\not\in\mathcal{C}(\tau)$, it follows from Claim \ref{IndexIExtendColorFromS} that there is $y\in V(G\setminus C)$ adjacent to all three of $z_0, w, z_1$. In particular, $N(w)=\{q_0, z_0, w, z_1, q_1\}$.

\begin{Claim}\label{OneAtLeastThreeCL} For some $i\in\{0,1\}$, $S_{c_i}^i\subseteq A\cup\tau(q_i)\}$. \end{Claim}

\begin{claimproof} Suppose not. Thus, by Claim \ref{ForEachAvoidSCiCol}, there exist $s_0\in S_{c_0}^0\setminus\{\tau(q_0)\}$ and $s_1\in S_{c_1}^1\setminus\{\tau(q_1)\}$, where $\{s_0, s_1\}\cap A=\varnothing$. By b) of Claim \ref{IndexIExtendColorFromS}, $\tau$ extends to an $L$-coloring of $G-w$ using $s_0, s_1$ on the respective vertices $z_0, z_1$. Possibly $y$ is using a color of $A$, but one color of $A$ is left over for $w$, which is false since $A\cap\mathcal{C}(\tau)=\varnothing$. \end{claimproof}

Now we finish the proof of P2) by constructing an element of $\Omega_G^{\geq 2}(P, [p_0])\cap\Omega_G^{\geq 2}(P, [p_1])$, contradicting our assumption that P2) ii) does not hold. 

\begin{Claim}\label{NeitherGz0NorGz1EdgeCL} Neither $G_{z_0}$ nor $G_{z_1}$ is an edge \end{Claim}

\begin{claimproof} By Claim \ref{OneAtLeastThreeCL}, there is an $i\in\{0,1\}$ such that $|S_{c_i}^i|<4$, so at most one of $G_{z_0}, G_{z_1}$ is an edge. Suppose the claim does not hold, and suppose without loss of generality that $G_{z_0}$ is not an edge, but $G_{z_1}$ is an edge. Let $H=H_{z_0}\setminus\{w, q_1\}$. Since $G_{z_1}$ is an edge, $H$ is bounded by outer cycle $p_{z_0}(C\setminus\mathring{P}p_1)z_1yz_0$. Note that every chord of the outer face of $H$ is incident to $y$.

\vspace*{-8mm}
\begin{addmargin}[2em]{0em}
\begin{subclaim} There is a chord of $H$ incident to $y$. \end{subclaim}

\begin{claimproof} Suppose not. Thus, the outer cycle of $H$ is induced. Let $Q$ be the 3-path $z_0yz_1p_1$ on $C^H$. Applying Claim \ref{ForEachAvoidSCiCol}, we choose an $s\in S_{c_0}^0\setminus\{r, \tau(q_0)\}$. Let $L^*$ be a list-assignment for $H$, where $z_0, p_1$ are precolored with the respective vertices $s, c_1$ and $L^*(p^{z_0})=\{s\}\cup\Lambda_{G_{z_0}}(c_0, s, \bullet)$ and and otherwise $L^*=L$. Note that $|L^*(p^{z_0})|\geq 3$. Since $|L^*(z_1)\setminus\{\sigma(w), \tau, c_1\}|\geq 2$, there is an $L^*$-coloring $\pi$ of $V(Q)$, where $\pi(y)\neq\sigma(w)$ and $\pi(z_1)\not\in\{\tau(q_1), \sigma(w), c_1\}$. Since $\sigma$ does not extend to an $L$-coloring of $G$, $\pi$ does not extend to an $L^*$-coloring of $H$, so it follows from Lemma \ref{PartialPathColoringExtCL0} that there is a $y'\in V(H\setminus C^H)$ with $|L^*_{\pi}(y')|<2$, so $y'$ at least three neighbors on $Q$. Since $G$ has no copies of $K_{2,3}$, we have $N(y)\cap V(Q)=\{y, z_1, p_1\}$, and $G$ contains a wheel in which $z_1$ is adjacent to every vertex of the 5-cycle $wq_1p_1y'y$. In particular, $N(z_1)=\{w, q_1, p_1, y', y\}$. Now we just uncolor $y$. Since $|L^*(y)\setminus\{\sigma(w), s\}|\geq 2$, there is an $s^*\in L^*(y)\setminus\{\sigma(w), s\}$ with $|L^*(y)\setminus\{\sigma(w), s^*, c_1, \tau(q_1)\}|\geq 2$, and the argument above also shows that there is an $L^*$-coloring $\psi$ of $H-y$ with $\psi(y)=s^*$. Since $y$ has only five neighbors, there is a color other than $\sigma(w)$ left over for $y$, so $\sigma\cup\psi$ extends to an $L$-coloring of $G$, which is false, since $\sigma$ does not extend to an $L$-coloring of $G$.\end{claimproof}\end{addmargin}

Let $v_y$ be the unique vertex of $N(y)\cap V(C\setminus\mathring{P})$ which is closest to $p_1$ on the path $C\setminus\mathring{P}$. As $G$ is $K_{2,3}$-free, $v_y\neq p_1$. Possibly $v_y=p^{z_0}$. Let $J$ be the subgraph of $G$ bounded by outer cycle $p_0(C\setminus\mathring{P})v_yyz_0$ and let $R$ be the 3-path $p_0z_0yv_y$ on the outer cycle of $J$. Let $J'$ be the subgraph of $G$ bounded by outer cycle $v_y(C\setminus\mathring{P})p_1z_1y$ and let $R'$ be the 3-path $v_yyz_1p_1$ on the outer cycle of $J'$. This is illustrated in Figure \ref{GZ0AndGZ1AndJNoChords}, where the paths $R$ and $R'$ are indicated in bold. As $G_{z_1}$ is an edge, the outer cycle of $J'$ is induced. Since $1\leq |L(p_1)|\leq 3$, it follows from 2) of Theorem \ref{ThmFirstLink3PathForUseInHolepunch} that there is a family of $|L(p_1)|+(|L(v_y)|-3)$ different elements of $\textnormal{End}(R', J')$, each using a different color on $v_y$.

\begin{center}\begin{tikzpicture}

\node[shape=circle,draw=black] (p0) at (-4.5, 0) {$p_0$};
\node[shape=circle,draw=white] (mid) at (-3.5, 0) {$\ldots$};
\node[shape=circle,draw=black] (pz) at (-2, 0) {$p^{z_0}$};
\node[shape=circle,draw=white] (vy+) at (1, 0) {$\ldots$};
\node[shape=circle,draw=black] (vy) at (0, 0) {$v_y$};
\node[shape=circle,draw=white] (vy-) at (-1, 0) {$\ldots$};
\node[shape=circle,draw=black] (p1) at (2.5, 0) {$p_1$};
\node[shape=circle,draw=black] (q0) at (-3.5,2) {$q_0$};
\node[shape=circle,draw=black] (q1) at (1.5,2) {$q_1$};
\node[shape=circle,draw=black] (w) at (0,4) {$w$};
\node[shape=circle,draw=black] (z0) at (-2.4, 1.9) {$z_0$};
\node[shape=circle,draw=white] (Gz) at (-2.6, 0.8) {$G_{z_0}$};
\node[shape=circle,draw=white] (J') at (0.7, 0.8) {$J'$};
\node[shape=circle,draw=black] (z1) at (0.4, 1.9) {$z_1$};
\node[shape=circle,draw=black] (y) at (-1, 1.8) {$y$};
\draw[-] (p0) to (q0) to (w) to (q1);
\draw[-] (q1) to (p1);
\draw[-]  (p0) to (mid) to (pz) to (vy-) to (vy) to (vy+) to (p1);
\draw[-] (w) to (z0);
\draw[-, line width=1.8pt] (p0) to (z0) to (y);
\draw[-] (w) to (z1);
\draw[-, line width=1.8pt] (z1) to (p1);
\draw[-] (q0) to (z0) to (pz);
\draw[-] (q1) to (z1);
\draw[-] (y) to (w);
\draw[-, line width=1.8pt] (vy) to (y) to (z1);

\end{tikzpicture}\captionof{figure}{}\label{GZ0AndGZ1AndJNoChords}\end{center}

\vspace*{-8mm}
\begin{addmargin}[2em]{0em}
\begin{subclaim} $v_y\neq p^{z_0}$ \end{subclaim}

\begin{claimproof} Suppose that $v_y=p^{z_0}$. Thus, $J-y=G_{z_0}$. By Corollary \ref{GlueAugFromKHCor}, there is a $\psi\in\textnormal{End}(p_0z_0p^{z_0}, G_{z_0})$ and a $\psi'\in\textnormal{End}(R', J')$ with $\psi(p^{z_0})=\psi'(p^{z_0})$. Let $\phi$ be the restriction of $\psi\cup\psi'$ to $\{p_0, p_1\}$. It follows from our choice of $\psi$ and $\psi'$ that any extension of $\phi$ to an $L$-coloring of $V(P)$ extends to $L$-color all of $G$, so, for any extension of $\phi$ to an $L$-coloring $\phi^*$ of $V(P-w)$, we have $|\mathcal{C}(\phi^*)|\geq 3$ and thus $\phi\in\Omega_G^{\geq 2}(P, [p_0])\cap\Omega_G^{\geq 2}(P, [p_1])$, contradicting our assumption that P2 ii) does not hold. \end{claimproof}\end{addmargin}

Let $X$ be the set of $L$-colorings $\psi$ of $\{p_0, v_y\}$ such that $\psi$ extends to two different elements of $\textnormal{End}(z_0, R, J)$. By 1) of Theorem \ref{CornerColoringMainRes}, since $|L(v_y)|\geq 3$ and $1\leq |L(p_0)|\leq 3$, there is a family of $|L(p_0)|$ elements of $X$ which all use different colors on $v_y$. Thus, there is a $\psi\in X$ and a $\psi'\in\textnormal{End}(R', J')$ with $\psi(v_y)_=\psi'(v_y)$, so $\psi\cup\psi'$ is a proper $L$-coloring of its domain $\{p_0, v_y, p_1\}$, and there exists a set $X^{\textnormal{col}}\subseteq L(z_0)\setminus\{\psi(p_0)\}$, where $|X^{\textnormal{col}}|=2$ and any extension of $\psi$ to an $L$-coloring of $V(R)$ using a color of $X^{\textnormal{col}}$ on $p^{z_0}$ extends to $L$-color all of $J$. This immediately implies the following.

\vspace*{-8mm}
\begin{addmargin}[2em]{0em}
\begin{subclaim}\label{IfPsiUnionPsi'AndXColForVyCase} Any extension of $\psi\cup\psi'$ to an $L$-coloring of $V(P)\cup\{z_0, y, z_1, v_y\}$ which uses a color of $X^{\textnormal{col}}$ on $z_0$ extends to $L$-color all of $G$. \end{subclaim}\end{addmargin}

We now let $\phi$ be the restriction of $\psi\cup\psi'$ to $\{p_0, p_1\}$. By assumption, P2) ii) does not hold, so $\phi\not\in\Omega_G^{\geq 2}(P)$.  Thus, $\phi$ extends to an $L$-coloring $\phi^*$ of $V(P-w)$ with $|\mathcal{C}(\phi^*)|<2$, so $|L_{\phi^*}(w)\setminus\mathcal{C}(\phi^*)|\geq 2$. Since $|X^{\textnormal{col}}\setminus\{\phi^*(q_0)\}|\geq 1$ and $\psi\cup\psi'$ leaves $p^{z_0}$ uncolored, it follows from our choice of $\psi$ and $\psi'$ that there is at most one color of $L_{\phi^*}(w)$ which does not lie in $\mathcal{C}(\phi^*)$, and this color, if it exists, lies in $L(w)\cap (X^{\textnormal{col}}\setminus\{\phi^*(q_0)\})$. Thus, $|\mathcal{C}(\phi^*)|\geq 2$, a contradiction. This completes the proof of Claim \ref{NeitherGz0NorGz1EdgeCL}. \end{claimproof}

Applying Claim \ref{NeitherGz0NorGz1EdgeCL}, we have the following.

\begin{Claim} $H_{z_0}\cap H_{z_1}$ is a wheel, where $y$ is adjacent to all the vertices of the outer cycle of $H_{z_0}\cap H_{z_1}$. \end{Claim}

\begin{claimproof} Applying Claim \ref{ForEachAvoidSCiCol}, for each $i\in\{0,1\}$, we let $s_i\in S_{c_i}^i\setminus\{\sigma(w), \tau(q_i)\}$. Let $L^*$ be a list-assignment for $H_{z_0}\cap H_{z_1}$, where $z_0, w, z_1$ are precolored with the respective colors $s_0, \sigma(w), s_1$ and $L^*(p^{z_i})=\{s_i\}\cup\Lambda_{G_{z_i}}(\tau(p_i), s_i, \bullet)$ for each $i\in\{0,1\}$, and otherwise $L^*=L$. Note that each of $p^{z_0}$ and $p^{z_1}$ has an $L^*$-list of size at least three, as neither $G_{z_0}$ nor $G_{z_1}$ is an edge. Let $G^*=(H_{z_0}\cap H_{z_1})-w$. Since $\sigma$ does not extend to an $L$-coloring of $G$, it follows that, for each $r\in L^*(y)\setminus\{s_0, \sigma(w), s_1\}$, the $L^*$-coloring $(s_0, r, s_1)$ of $z_0yz_1$ does not extend to $L^*$-color $G^*$. Since $|L^*(y)\setminus\{s_0, \sigma(w), s_1\}|\geq 2$ and every chord of the outer cycle of $G^*$ is incident to $y$, it follows from 1) of Theorem \ref{EitherBWheelOrAtMostOneColThm} that $G^*$ is a broken wheel with principal path $z_0yz_1$, and the claim follows.\end{claimproof}

We now proceed in a way similar to the proof of Claim \ref{NeitherGz0NorGz1EdgeCL}. We let $J$ be the subgraph bounded by outer cycle $p_0(C\setminus\mathring{P})p^{z_1}yz_1$ and we let $R$ be the 3-path $p_0z_0yp^{z_1}$ on the outer cycle of $J$. Let $X$ be the set of $L$-colorings $\psi$ of $\{p_0, p^{z_1}\}$ such that $\psi$ extends to two different elements of $\textnormal{End}(z_0, R, J)$. By 1) of Theorem \ref{CornerColoringMainRes}, since $|L(p^{z_1})|\geq 3$ and $1\leq |L(p_0)|\leq 3$, there is a family of $|L(p_0)|$ elements of $X$ which all use different colors on $p^{z_1}$. By Corollary \ref{GlueAugFromKHCor}, there is a $\psi\in X$ and a $\psi'\in\textnormal{End}(p^{z_1}z_1p_1, G_{z_1})$ with $\psi(p^{z_1})=\psi'(p^{z_1})$. That is, $\psi\cup\psi'$ is a proper $L$-coloring of its domain $\{p_0, p^{z_1}, p_1\}$, and there exists a set $X^{\textnormal{col}}\subseteq L(z_0)\setminus\{\psi(p_0)\}$, where $|X^{\textnormal{col}}|=2$ and any extension of $\psi$ to an $L$-coloring of $V(R)$ using a color of $X^{\textnormal{col}}$ on $p^{z_0}$ extends to $L$-color all of $J$, so we immediately have the following.

\begin{Claim}\label{IfPsiUnionPsi'AndXColThenExt} Any extension of $\psi\cup\psi'$ to an $L$-coloring of $V(P)\cup\{z_0, y, z_1, p^{z_1}\}$ which uses a color of $X^{\textnormal{col}}$ on $z_0$ extends to $L$-color all of $G$. \end{Claim}

As in the proof of Claim \ref{NeitherGz0NorGz1EdgeCL}, since $p^{z_0}$ is uncolored, it follows that the restriction of $\psi\cup\psi'$ to $\{p_0, p_1\}$ is an element of $\Omega_G^{\geq 2}(P)$, contradicting our assumption that P2 ii) does not hold. This proves P2). 

\subsection{Proof of P3)}

Let $\psi, \psi'$ be as in the statement of P3), and suppose without loss of generality that $|\{\psi(q_0), \psi'(q_0)\}|=2$. Let $B=\{\psi(q_0), \psi'(q_0)\}$ and let $\phi$ be the common restriction of $\psi, \psi'$ to $p_0, p_1, q_1$. Suppose toward a contradiction that P3) does not hold. Thus, $I=\{0,1\}$, but $\mathcal{C}(\psi)$ and $\mathcal{C}(\psi')$ are disjoint sets of size at most one. 

\begin{Claim}\label{PhiExtToLColGMinLeftSide} Each of $\psi$ and $\psi'$ extends to an $L$-coloring of $G$. That is, $\mathcal{C}(\psi)|=|\mathcal{C}(\psi')|=1$. Furthermore, $|L(w)\setminus (B\cup\{\phi(q_1)\})|=2$. \end{Claim}

\begin{claimproof} Since $C$ is induced, $|L_{\psi}(w)|\geq 3$ and $|L_{\psi'}(w)|\geq 3$. By assumption, $N(p_0)\cap N(p_1)\subseteq V(C)$ and $N(q_0)\cap N(q_1)\subseteq V(C)$, so it follows from Lemma \ref{PartialPathColoringExtCL0} that each of $\psi$ and $\psi'$ extends to an $L$-coloring of $G$. Thus, $|\mathcal{C}(\psi)|=|\mathcal{C}(\psi')|=1$. Furthermore, we have $|L(w)\setminus (B\cup\{\phi(q_1)\})|=2$, or else it also follows from Lemma \ref{PartialPathColoringExtCL0} that either $|\mathcal{C}(\psi)|\geq 2$ or $|\mathcal{C}(\psi')|\geq 2$, which is false. \end{claimproof}

\begin{Claim}\label{VertZ0Z1CoveringBothSidesCL} There exist $z_0, z_1\in V(G)$ such that $N(z_i)\cap V(P)=\{p_i, q_i, w\}$ for each $i=0,1$, and $z_0z_1\not\in E(G)$. \end{Claim}

\begin{claimproof} We choose an arbitrary element of $\{\psi, \psi'\}$, say $\psi$. By assumption, $\psi$ extends to two distinct $L$-colorings $\psi^0, \psi^1$ of $V(P)$, where neither $\psi^0$ nor $\psi^1$ extends to an $L$-coloring of $G$. Since $\psi^0$ and $\psi^1$ restrict to the same $L$-coloring of $V(P-w)$, it follows that, for each endpoint $p_i$ of $P$, at most one of $\{\psi^0, \psi^1\}$ is $(L, P, p_i)$-blocked. Since $I=\{0,1\}$ by assumption, it follows from Claim \ref{TauObstructedByZiOnAtLeastOne} that there are $z_0, z_1$ satisfying Claim \ref{VertZ0Z1CoveringBothSidesCL}. If $z_0z_1\in E(G)$, then $I=\varnothing$, contradicting our assumption, so $z_0z_1\not\in E(G)$. \end{claimproof}

\begin{Claim}\label{AnyLCol4CornerToGMinQ0} Any $L$-coloring $\sigma$ of $\{p_0, z_0, q_1, p_1\}$ extends to $L$-color $G-q_0$. \end{Claim}

\begin{claimproof} To prove Claim \ref{AnyLCol4CornerToGMinQ0}, we first show the following intermediate result. 

\vspace*{-8mm}
\begin{addmargin}[2em]{0em}
\begin{subclaim}\label{IntermFactsForGMQ0CL} $d(p^{z_0}, p_1)>1$ and furthermore, there is no vertex $y$ of $H^{z_0}$ which lies outside the outer cycle of $H^{z_0}$ and is adjacent to all three of $z_0, p^{z_0}, p_1$. \end{subclaim}

\begin{claimproof}  Suppose that $d(p^{z_0}, p_1)\leq 1$. Since $C$ is an induced cycle and $N(p_0)\cap N(p_1)\subseteq V(C)$, $p^{z_0}p_1$ is an edge of $C$, and the outer cycle of $H_{z_0}$ has length five. It follows from Claim \ref{PhiExtToLColGMinLeftSide} that $\phi$ extends to an $L$-coloring $\phi^+$ of $V(G_{z_0})\cup\{p_1, q_1\}$. At least one color of $B$ is distinct from $\phi^+(z_0)$, say $\psi(q_0)$ without loss of generality. As $|L_{\phi^+}(w)|\geq 3$ and $\mathcal{C}(\psi)|\leq 1$, there is an $L$-coloring of the outer cycle of $H_{z_0}$ which does not extend to $L$-color $H_{z_0}$. Since the outer cycle of $H_{z_0}$ has length five, it follows from Theorem \ref{BohmePaper5CycleCorList} that $z_1$ is adjacent to all five vertices of this cycle, contradicting Claim \ref{VertZ0Z1CoveringBothSidesCL}. Thus, $d(p^{z_0}, p_1)>1$. Now suppose toward a contradiction that there is a vertex $y$ of $H^{z_0}$ which lies outside the outer cycle of $H^{z_0}$ and is adjacent to all three of $z_0, p^{z_0}, p_1$. Let $F$ be the 5-cycle $wz_0yp_1z_1$. As $\phi$ extends to $L$-color $G-q_0$, it follows that $\phi$ extends to an $L$-coloring $\phi^*$ of $V(\textnormal{Ext}(F))\setminus\{q_0, w\}$. Since $\phi^*$ colors neither $z_0$ nor $z_1$, we have $|L_{\phi^*}(w)|\geq 3$. At least one color of $B$ is distinct from $\phi^*(z_0)$, say $\psi(q_0)$ without loss of generality. Since $|L_{\phi^*}(w)|\geq 3$ and $\mathcal{C}(\psi)|\leq 1$, there is an $L$-coloring of $V(F)$ which does not extend to $L$-color $\textnormal{Int}(F)$. Since $F$ has length five and $z_1\in V(\textnormal{Int}(F)\setminus F)$, it follows from Theorem \ref{BohmePaper5CycleCorList} that $z_1$ is adjacent to all five vertices of this cycle,  which again contradicts Claim \ref{VertZ0Z1CoveringBothSidesCL}. \end{claimproof}\end{addmargin}

Now suppose toward a contradiction that Claim \ref{AnyLCol4CornerToGMinQ0} does not hold. Let $\sigma$ be an $L$-coloring of $\{p_0, z_0, q_1, p_1\}$ which does not extend to $L$-color $G-q_0$.  Possibly $p^{z_0}\neq p_0$ and $\Lambda_{G_{z_0}}(\sigma(p_0), \sigma(z_0), \bullet)=\{\sigma(p_1)\}$, but since $d(p^{z_0}, p_1)>1$ by Subclaim \ref{IntermFactsForGMQ0CL}, $\sigma$ extends to an $L$-coloring $\sigma'$ of $V(G_{z_0})\cup\{p_1, q_1\}$. Note that $L_{\sigma'}(w)=L_{\sigma}(w)$. It follows from Lemma \ref{PartialPathColoringExtCL0} that there is a $y\in V(H_{z_0})$ which lies outside the outer cycle of $H_{z_0}$ and has at least three neighbors in $\{p^{z_0}, z_0, q_1, p_1\}$. If $y=z_1$, then, since $z_0z_1\not\in E(G)$, we have $z_1p^{z_0}\in E(G)$. If that holds, then, since $G$ has no induced 4-cycles and $C$ is an induced cycle, we get $z_0z_1\in E(G)$, a contradiction. Thus, $y\neq z_1$, and since $N(q_1)=\{w, z_1, p_1\}$, it follows that $y$ is adjacent to all three of $z_0, p^{z_0}, p_1$, contradicting Subclaim \ref{IntermFactsForGMQ0CL}. \end{claimproof}

\begin{Claim}\label{CPsiToPsi'OrFlipped} Either $\mathcal{C}(\psi)=\{\psi'(q_0)\}$ or $\mathcal{C}(\psi')=\{\psi(q_0)\}$ (possibly both). In particular, $L_{\phi}(w)\setminus (B\cup\mathcal{C}(\psi)\cup\mathcal{C}(\psi'))$ is nonempty. \end{Claim}

\begin{claimproof} We choose a color $c\in L(z_0)\setminus (B\cup\{\phi(p_0)\})$. Since $q_1, p_1\not\in N(z_0)$, it follows from Claim \ref{AnyLCol4CornerToGMinQ0} that $\phi$ extends to an $L$-coloring $\phi^*$ of $G-q_0$ with $\phi^*(z_0)=c$. If $\phi^*(w)\not\in B$, then $\phi^*(w)\in\mathcal{C}(\psi)\cap\mathcal{C}(\psi')$, contradicting our assumption that these sets are disjoint, so $\phi^*(w)\in B$. Suppose for the sake of definiteness that $\phi^*(w)=\psi(q_0)$. Thus, $\psi'$ extends to an $L$-coloring of $G$ using $\psi(q_0)$ on $w$. Since $|\mathcal{C}(\psi')|\leq 1$, we have $\mathcal{C}(\psi')=\{\psi(q_0)\}$. Thus, $|B\cup\mathcal{C}(\psi)\cup\mathcal{C}(\psi')|\leq 3$, so $L_{\phi}(w)\setminus (B\cup\mathcal{C}(\psi)\cup\mathcal{C}(\psi'))\neq\varnothing$. \end{claimproof}

\begin{Claim} There is a $y\in V(G\setminus C)$ adjacent to all three of $z_0, w, z_1$. In particular, $N(w)=\{q_0, z_0, y, z_1, q_1\}$. \end{Claim}

\begin{claimproof} Suppose not. By Claim \ref{SetSetAzAtLeastTwoCL}, $|S_{\phi(p_1)}^1\setminus\{\phi(q_1)\}|\geq 1$, so we fix an $s_1\in S_{\phi(p_1)}^1$ with $s_1\neq\phi(q_1)$.

\vspace*{-8mm}
\begin{addmargin}[2em]{0em}
\begin{subclaim}\label{UnionCPsiPsi'BAtLeast3} $L(w)\setminus\{\phi(q_1), s_1\}=\mathcal{C}(\psi)\cup\mathcal{C}(\psi')\cup B$ \end{subclaim}

\begin{claimproof} Suppose not. By Claim \ref{CPsiToPsi'OrFlipped}, $|\mathcal{C}(\psi)\cup\mathcal{C}(\psi')\cup B|\leq 3$, so there is an $r\in L(w)\setminus\{\phi(q_1), s_1\}$ with $r\not\in\mathcal{C}(\psi)\cup\mathcal{C}(\psi')\cup B$. Since $|S_{\phi(p_0)}^0\setminus\{r\}|\geq 1$, $\phi$ extends to an $L$-coloring $\pi$ of $\{p_0, p_1, q_1\}\cup\{z_0, z_1, w\}$ which uses the colors $r, s_1$ on the respective vertices $w, z_1$, where $\pi(z_0)\in S_{\phi(p_0)}^0$. By Claim \ref{IndexIExtendColorFromS}, $\pi$ extends to an $L$-coloring of $G-q_0$. At least one color of $B$ is distinct from $\pi(z_0)$, so $r\in\mathcal{C}(\psi)\cup\mathcal{C}(\psi')$, which is false. \end{claimproof}\end{addmargin}

Applying Claim \ref{CPsiToPsi'OrFlipped}, we suppose without loss of generality that $\mathcal{C}(\psi')=\{\psi(q_0)\}$. By Subclaim \ref{UnionCPsiPsi'BAtLeast3}, $|\mathcal{C}(\psi)\cup\mathcal{C}(\psi')\cup B|\geq 3$, and, in particular, $\mathcal{C}(\psi)=r$ for some $r\in L(w)\setminus (\{\phi(q_1), s_1\}\cup B)$. 

\vspace*{-8mm}
\begin{addmargin}[2em]{0em}
\begin{subclaim} $S_{\phi(p_0)}^0=\{r, \psi'(q_0)\}$. \end{subclaim}

\begin{claimproof} Suppose not. Since $|S_{\phi(p_0)}^0|\geq 2$, there is an $s_0\in S_{\phi(p_0)}^0$ with $s_0\not\in\{r, \psi'(q_0)\}$. By Claim \ref{IndexIExtendColorFromS}, since $r\neq s_1$, $\phi$ extends to an $L$-coloring of $G-q_0$ using the colors $s_0, r, s_1$ on the respective vertices $z_0, w, z_1$. Since $s_0\neq\psi'(q_0)$, it follows that $r\in\mathcal{C}(\psi')$, which is false. \end{claimproof}\end{addmargin}

Since $\psi'(q_0)\in S_{\phi(p_0)}^0$ and $|L_{\phi}(w)\setminus (B\cup\{s_1\})|\geq 1$, it follows from Claim \ref{IndexIExtendColorFromS} that $\phi$ extends to an $L$-coloring $\pi^*$ of $G-q_0$, where $\pi^*(z_0)=r$ and $\pi^*(z_1)=s_1$, and $\pi^*(w)\not\in B$. Since $\pi^*(w)$ and $\psi'(q_0)$ are both distinct from $\psi(q_0)$, we have $\pi^*(w)\in\mathcal{C}(\psi)$, which is false, as $\pi^*(w)\not\in B$. \end{claimproof}

\begin{Claim}\label{P0PZ1Not5CycleCL} $d(p_0, p^{z_1})>1$. \end{Claim}

\begin{claimproof} Suppose toward a contradiction that  $d(p_0, p^{z_1})\leq 1$. Note that $p_0\neq p^{z_1}$, and, since $C$ is an induced cycle, $p_0p^{z_1}$ is an edge of $C$, and $H_{z_1}-q_0$ has an outer cycle of length five. 

\vspace*{-8mm}
\begin{addmargin}[2em]{0em}
\begin{subclaim} $\phi$ extends to an $L$-coloring of $V(G_{z_1})\cup\{p_0, z_1, q_1\}$.  \end{subclaim}

\begin{claimproof} This is immediate if $G_{z_0}$ is an edge, so suppose that $G_{z_0}$ is not an edge. Since $|S_{\phi(p_1)}^1|\geq 2$, we fix an $s\in S_{\phi(p_1)}^1\setminus\{\phi(q_1)\}$. Since $|\Lambda_{G_{z_0}}(\bullet, s, \phi(p_0))|\geq 2$, there is a color of $\Lambda_{G_{z_1}}(\bullet, s, \phi(p_0))$ distinct from $\phi(p_0)$, and the subclaim follows. \end{claimproof}\end{addmargin}

Let $\pi$ be an extension of $\phi$ to an $L$-coloring of $V(G_{z_1})\cup\{p_0, z_1, q_1\}$. Since $G$ is $K_{2,3}$-free, $z_0\not\in N(p^{z_1})$, so $|L_{\pi}(z_0)|\geq 4$. Since $|L_{\pi}(w)|\geq 3$, there is a $b\in L_{\pi}(z_0)$ with $L_{\pi}(w)\setminus\{b\}|\geq 3$. At least one color of $B$ is distinct from $b$, so let $\psi^*\in\{\psi, \psi'\}$, where $\psi^*(q_0)\neq b$, and let $\pi^*$ be an extension of $\pi$ to $\textnormal{dom}(\pi)\cup\{z_0\}$ using $b$ on $z_0$. Note that $L_{\psi^*\cup\pi^*}(w)|\geq 2$. Since $y\not\in N(p_0)$, it follows from Theorem \ref{BohmePaper5CycleCorList} that $\mathcal{C}(\psi^*\cup\pi^*)=L_{\psi^*\cup\pi^*}(w)$, so $|\mathcal{C}(\psi^*)|\geq |\mathcal{C}(\psi^*\cup\pi^*)|\geq 2$, which is false. \end{claimproof}

We now let $\overline{B}=L_{\phi}(w)\setminus (B\cup\mathcal{C}(\psi)\cup\mathcal{C}(\psi'))$. By Claim \ref{CPsiToPsi'OrFlipped}, $|\overline{B}|\geq 1$. We now have the following key claim. 

\begin{Claim}\label{GZ1NotEdgeOverlineBSingle} $G_{z_1}$ is not an edge, and $p^{z_1}\in N(y)$. Furthermore, $|\overline{B}|=1$ and $\overline{B}\subseteq L_{\phi}(z_0)\cap (S_{\phi(p_1)}^1\setminus\{\phi(q_1)\})$.  \end{Claim}

\begin{claimproof} We fix an $r\in\overline{B}$ and let $\phi^+$ be an extension of $\phi$ to an $L$-coloring of $\{p_0, p_1, q_1, w\}$, where $\phi^+(w)=r$. If $\phi^+$ extends to an $L$-coloring of $G-q_0$, then, since $r\not\in B$ and at least one color of $B$ is left for $q_0$, $r\in\mathcal{C}(\psi)\cup\mathcal{C}(\psi')$, which is false. Thus, $\phi^+$ does not extend to $L$-color $G-q_0$.

\vspace*{-8mm}
\begin{addmargin}[2em]{0em}
\begin{subclaim}\label{SubCLGZ1Not Edge4} $G_{z_1}$ is not an edge.  \end{subclaim}

\begin{claimproof} Suppose that $G_{z_1}$ is an edge. Let $T=\{p_0, p_1, z\}$ and let $\pi$ be an extension of $\phi^+$ to an $L$-coloring of $\textnormal{dom}(\phi^+)\cup\{z\}$. Recalling the notation of Definition \ref{GeneralListDefn1}, we consider the list-assignment $L_{\pi}^T$ for $G-\{q_0, w, q_1\}$. Since $G_{z_1}$ is an an edge, every chord of the outer cycle of $G-\{q_0, w, q_1\}$ is incident to one of $z_0, y$. Note that, since we have not colored $q_0$ and $y\not\in N(p_0)$, each of $z_0, y$ has an $L_{\pi}^T$-list of size at least four. Applying Lemma \ref{PartialPathColoringExtCL0} to the 4-path $p_0z_0yz_1p_1$ on the outer cycle of $G-\{q_0, w, q_1\}$, it follows that $G-\{q_0, w, q_1\}$ is $L^T_{\pi}$-colorable, so $\pi$ exends to an $L$-coloring of $G-q_0$, which is false. \end{claimproof}\end{addmargin}

\vspace*{-8mm}
\begin{addmargin}[2em]{0em}
\begin{subclaim}\label{OverlineBSizeOneR} $\overline{B}=\{r\}=S_{\phi(p_1)}^1\setminus\{\phi(q_1)\}$ \end{subclaim}

\begin{claimproof} Suppose not. Since $|S_{\phi(p_1)}^1|\geq 2$, $\phi^+$ extends to an $L$-coloring $\sigma$ of, where $\sigma(z_1)\in S_{\phi(p_1)}^1$. Let $L^+$ be a list-assignment for, where $L^+(p^{z_1})=\{s\}\cup\Lambda_{G_{z_1}}(\bullet, s, \phi(p_1))$, and the vertices $p_0, w, z_1$ are precolored with the respective colors $\phi(p_0), r,s$, and otherwise $L^+=L$. Note that $|L^+(p^{z_1})|\geq 3$. Since $y\not\in N(p_0)$, each of $z_0, y$ has two $L^+$-precolored neighbors on the outer cycle of $H_{z_1}-q_0$. It follows from Lemma \ref{PartialPathColoringExtCL0} applied to the 3-path $p_0z_0wz_1$ on the outer cycle of $H_{z_1}-q_0$ that $H_{z_1}-q_0$ is $L^+$-colorable, and, in particular, this $L^+$-coloring uses a color of $\Lambda_{G_{z_1}}(\bullet, s, \phi(p_1))$ on $p^{z_1}$, so $\phi^+$ extends to an $L$-coloring of $G-q_0$, which is false.  \end{claimproof}\end{addmargin}

\vspace*{-8mm}
\begin{addmargin}[2em]{0em}
\begin{subclaim}  $p^{z_1}y\in E(G)$. \end{subclaim}

\begin{claimproof} Suppose not. We now choose an arbitrary $s\in L_{\phi^+}(z)$. Possibly $\Lambda_{G_{z_1}}(\bullet, s, \phi(p_1))=\{\phi(p_0)\}$, but since $d(p_0, p^{z_1})>1$ by Claim \ref{P0PZ1Not5CycleCL}, $\phi^+$ extends to an $L$-coloring $\pi$ of $\textnormal{dom}(\phi^+)\cup V(G_{z_1})$ using $s$ on $z$. By assumption, $p^{z_1}y\not\in E(G)$. Note that $z_0\not\in N(p^{z_1})$, and since we have not colored $q_0$, each of $z_0, y$ has an $L_{\pi}$-list of size at least three, so, as in Subclaim \ref{SubCLGZ1Not Edge4}, it follows from Lemma \ref{PartialPathColoringExtCL0}  applied to the 4-cycle $p_0z_0yz_1p_1$ on the outer face of $H_{z_1}-\{q_0, w\}$ that $\pi$ extends to an $L$-coloring of $G-q_0$, which is false. \end{claimproof}\end{addmargin}

To finish the proof of Claim \ref{GZ1NotEdgeOverlineBSingle}, we just need to check that $r\in L_{\phi}(z_0)$. Let $G^*=H_{z_1}-\{w, z_1, q_0\}$ and let $R$ be the 3-path $p_0z_0yp^{z_1}$ on the outer cycle of $G^*$. This is illustrated in Figure \ref{Gz0Z1RColorInZ1Case}, where $R$ is indicated in bold. Suppose toward a contradiction that $r\not\in L_{\phi}(z_0)$. Since $|S_{\phi(p_1)}^1|=2$, we have $S_{\phi(p_1)}^1=\{r, \phi(q_1)\}$ by Subclaim \ref{OverlineBSizeOneR}, so there exists a pair of $L$-colorings $\sigma, \sigma^*$ of the edge $z_1p_1$, where $\sigma(z_1)\neq\sigma^*(z_1)$ and $\{\sigma(z_1), \sigma^*(z_1)\}\subseteq L_{\phi}(z_1)\setminus S_{\phi(p_1)}^1$. Let $T=\{\sigma(z_1), \sigma^*(z_1)\}$. Note that $|\Lambda_{G_{z_1}}(\bullet, \sigma(z_1), \phi(p_1))|=|\Lambda_{G_{z_1}}(\bullet, \sigma^*(z_1), \phi(p_1))|=1$ by definition, and, in particular, $T\subseteq L(p^{z_1})$. By 2 i) of Corollary \ref{CorMainEitherBWheelAtM1ColCor}, $G_{z_1}$ is a broken wheel with principal path $p^{z_1}z_1p_1$. 

\begin{center}\begin{tikzpicture}

\node[shape=circle,draw=black] (p0) at (-4, 0) {$p_0$};
\node[shape=circle,draw=white] (mid) at (-1, 0) {$\ldots$};
\node[shape=circle,draw=black] (pz1) at (1.5, 0) {$p^{z_1}$};
\node[shape=circle,draw=white] (un+) at (3, 0) {$\ldots$};
\node[shape=circle,draw=black] (p1) at (4, 0) {$p_1$};
\node[shape=circle,draw=black] (q0) at (-3,2) {$q_0$};
\node[shape=circle,draw=black] (q1) at (3,2) {$q_1$};
\node[shape=circle,draw=black] (w) at (0,4) {$w$};
\node[shape=circle,draw=black] (z0) at (-1.9, 1.9) {$z_0$};
\node[shape=circle,draw=white] (G*) at (-1, 0.8) {$G^*$};
\node[shape=circle,draw=black] (z1) at (1.9, 1.9) {$z_1$};
\node[shape=circle,draw=white] (Gz1) at (2.1, 0.8) {$G_{z_1}$};
\node[shape=circle,draw=black] (y) at (0, 1.8) {$y$};
 \draw[-] (p1) to (un+);
 \draw[-] (p0) to (q0) to (w) to (q1);
\draw[-] (q1) to (p1);
\draw[-]  (p0) to (mid) to (pz1) to (un+);
\draw[-] (w) to (z0);
\draw[-, line width=1.8pt] (p0) to (z0) to (y) to (pz1);
\draw[-] (w) to (z1) to (p1);
\draw[-] (q0) to (z0);
\draw[-] (q1) to (z1) to (pz1);
\draw[-] (y) to (w);
\draw[-] (y) to (z1);
\draw[-] (y) to (pz1);
\end{tikzpicture}\captionof{figure}{}\label{Gz0Z1RColorInZ1Case}\end{center}

\vspace*{-8mm}
\begin{addmargin}[2em]{0em}
\begin{subclaim}\label{TwoLambdaSetsSameSubMain} $\Lambda_{G_{z_1}}(\bullet, \sigma(z_1), \phi(p_1))=\Lambda_{G_{z_1}}(\bullet, \sigma^*(z_1), \phi(p_1))$. \end{subclaim}

\begin{claimproof} Suppose not. Thus, by 1b) of Theorem \ref{BWheelMainRevListThm2}, $G_{z_1}-z_1$ is a path of odd length, and $\Lambda_{G_{z_1}}(\bullet, \sigma(z_1), \phi(p_1))=\{\sigma^*(z_1)\}$ and $\Lambda_{G_{z_1}}(\bullet, \sigma^*(z_1), \phi(p_1))=\{\sigma(z_1)\}$. Now, there exists a color $s\in L(y)\setminus (\{r\}\cup T)$. Let $L^*$ be a list-assignment for $V(G^*)$, where $L^*(p_0)=\{\phi(p_0)\}$ and $L^*(y)=\{s\}$, and $L^*(p^{z_1})=\{s\}\cup T$, and otherwise $L^*=L$. Note that $|L^*(p^{z_1})|=3$ and $L^*(z_0)|\geq 5$, and since only two vertices of $R-p^{z_1}$ are precolored, it follows from Lemma \ref{PartialPathColoringExtCL0} applied to the 2-path $R-p^{z_1}$ that $G^*$ admits an $L^*$-coloring $\tau$. Since $\tau(y)=s$, we have $\tau(p^{z_1})\in T$, so $\tau$ is an $L$-coloring of $G^*$. By assumption, $r\not\in L_{\phi}(z_0)$, so $\tau\cup\phi^+$ is a proper $L$-coloring of its domain. Note that this is true even if $G_{z_1}$ is a triangle. In particular, since $\tau(p^{z_1})\in T$, we suppose without loss of generality that $\tau(p^{z_1})=\sigma(z_1)$. As $\tau(y)=s$, the color $\sigma^*(z_1)$ is left over for $z_1$, and since $\Lambda_{G_{z_1}}(\bullet, \sigma^*(z_1), \phi(p_1))=\{\sigma(z_1)\}$, it follows that $\phi^+$ extends to an $L$-coloring of $G-q_0$, which is false. \end{claimproof}\end{addmargin}

By Subclaim \ref{TwoLambdaSetsSameSubMain}, that there is an r $f\in L(p^{z_1})\setminus T$, where $\Lambda_{G_{z_1}}(\bullet, \sigma(z_1), \phi(p_1))=\Lambda(\bullet, \sigma^*(z_1), \phi(p_1))=\{f\}$. Now we let $\tilde{L}$ be another list-assignment for $G^*$, where $\tilde{L}(p_0)=\{\phi(p_0)\}$ and $\tilde{L}(p^{z_1})=\{f\}$, and $\tilde{L}(y)=L(y)\setminus\{r\}$ and otherwise $\tilde{L}=L$. Possibly $f=\phi(p_0)$, but, by Claim \ref{P0PZ1Not5CycleCL}, $d(p_0, p^{z_1})>1$. Since $p_0y\not\in E(G)$, and since $|\tilde{L}(y)|\geq 4$ and $|\tilde{L}(z_0)|\geq 5$, it follows from Lemma \ref{PartialPathColoringExtCL0} applied to $G^*$ and $R$ that $G^*$ admits a $\tilde{L}$-coloring $\tau$. By assumption, $\tau(z_0)\neq r$. By definition of $\tilde{L}$, $\tau(y)\neq r$ as well, so $\tau\cup\phi^+$ is a proper $L$-coloring of its domain. Since $\tau(y)$ is distinct from at least one color of $T$, it follows that $\phi^+$ extends to an $L$-coloring of $G-q_0$, which is false. \end{claimproof}

Applying Claim \ref{GZ1NotEdgeOverlineBSingle}, we let $r$ be the lone color of $\overline{B}$, and we let $\sigma$ be an extension of $\phi$ to an $L$-coloring of $\{p_0, p_1, q_1\}\cup\{z_0, z_1\}$ which colors both of $z_0, z_1$ with $r$.

\begin{Claim}\label{SigExtGNotQ0WCLFin} $\sigma$ extends to an $L$-coloring of $G\setminus\{q_0, w\}$ \end{Claim}

\begin{claimproof} Firstly, $\sigma$ extends to an $L$-coloring $\sigma'$ of $\textnormal{dom}(\sigma)\cup V(G_{z_0})$. Let $H=H_{z_0}\cap H_{z_1})-w$. That is, $H$ is bounded by outer cycle $p^{z_0}(C\setminus\mathring{P})p^{z_1}z_1yz_0$, and the outer cycle of $H$ contains the 3-path $Q=p^{z_0}z_0yz_1$. Let $L^+$ be a list-assignment for $V(H)$ in which $L^+(p^{z_1})=\{\sigma(z_1)\}\cup\Lambda_{G_{z_1}}(\bullet, r, \phi(p_1))$ and otherwise $L^+=L$. By Claim \ref{GZ1NotEdgeOverlineBSingle}, $r\in S_{\phi(p_1)}^1$, so $|L^+(p^{z_1})|\geq 3$. We just need to check that the restriction of $\sigma'$ to $\{p^{z_0}, z_0, z_1\}$ extends to $L^+$-color $H$. Possibly $p^{z_0}y$ is an edge, but since $\sigma$ colors $z_0, z_1$ with the same color, $|L^+_{\sigma'}(y)|\geq 3$. Suppose that $\sigma'$ does not extend to an $L^+$-coloring of $H$. Note that every chord of the outer cycle of $H$ is incident to $y$. Thus, by Lemma \ref{PartialPathColoringExtCL0} applied to $H$ and $Q$, there is a vertex of $H$ which does not lie on the outer cycle of $H$ and is adjacent to the three vertices of $V(Q)\cap\textnormal{dom}(\sigma')$, which is false, since $G$ contains no copies of $K_{2,3}$. \end{claimproof}

We now have enough to complete the proof of Proposition \ref{InterMedUnobstrucGInduced}. Applying Claim \ref{SigExtGNotQ0WCLFin}, we let $\sigma^*$ be an extension of $\sigma$ to an $L$-coloring of $G\setminus\{q_0, w\}$. Since $|\overline{B}|=1$, we suppose without loss of generality that $\mathcal{C}(\psi')=\{\psi(q_0)\}$ and $\mathcal{C}(\psi)=\{s\}$ for some $s\in L_{\phi}(w)\setminus (B\cup\overline{B})$. Since $r\not\in B$ and $s\not\in\mathcal{C}(\psi')$, we have $\sigma^*(y)=s$. But then, since $B\subseteq L_{\phi}(y)$ by Claim \ref{PhiExtToLColGMinLeftSide}, we have $\psi'(q_0)\in L_{\sigma^*}(w)$, so $\psi'(q_0)\in\mathcal{C}(\psi)$, which is false. \end{proof}

\section{The Main 4-Path Lemma}\label{Main4PathLemmaSec}

In Section \ref{Main4PathLemmaSec}, which is also the remainder of this paper, we prove Lemma \ref{EndLinked4PathBoxLemmaState}. 

\begin{lemma}\label{EndLinked4PathBoxLemmaState} Let $(G, C, P, L)$ be an end-linked rainbow, where $P$ is a path of length four and $G$ is a short-separation-free graph in which every face of $G$, except for $C$, is bounded by a triangle. Suppose that $|L(w)|\geq 5$ and no chord of $C$ is incident to $w$. Then, for each endpoint $p$ of $P$, $\Omega_G(P, [p])\neq\varnothing$.
\end{lemma}

\begin{proof} Suppose not, and let $G$ be a vertex-minimal counterexample to Lemma \ref{EndLinked4PathBoxLemmaState}. To avoid clutter, for any partial $L$-coloring $\phi$ of $V(G-w)$, we let $\mathcal{C}(\phi)=\mathcal{C}_G^P(\phi)$. We let $P=p_0q_0wq_1p_1$. We may suppose, by removing colors from the lists of $p_0, p_1$ if necesssary, that $|L(p_0)|+|L(p_1)|=4$, and thus $1\leq |L(p_i)|\leq 3$ for each $i\in\{0,1\}$. Because of Condition 2c) of Definition \ref{PartialLColOmega}, we have to be careful when we remove colors from the lists of some vertices of $G$, but it is permissible to remove colors from the lists of the endpoints of $P$ in the way described above since each element of $\Omega_G(P, [p_0])\cup\Omega_G(P, [p_0])$ has both endpoints of $P$ in its domain. We break the proof of Lemma \ref{EndLinked4PathBoxLemmaState} into six subsections, which are organized as follows.

\begin{enumerate}[label=\arabic*)] 
\itemsep-0.1em
\item In Subsection \ref{BoxLemmaBasicFactsSubSec}, we gather some basic facts about $G$, and, in particular, we show that at most one of $q_0, q_1$ is incident to a chord of $C$. 
\item  In Subsection \ref{PrelimFor6CycleSub}, we show that, for each $i\in\{0,1\}$ and any $x\in V(C\setminus\mathring{P})\cap N(q_{1-i})$, no vertex of $G\setminus C$ is adjacent to both of $p_i$ and $x$. In Subsection \ref{SubsecChordsIncToQi}, we gather a few facts about chords of $C$. 
\item In Subsections \ref{VertexZAllofPiQiW}-\ref{2ChordIntermSecondPartSec}, we show that, if there is an $i\in\{0,1\}$ such that $p_i, q_i, w$ have a common neighbor $z$, then $N(z)\cap N(q_{1-i})=\{w\}$. In Subsection \ref{FinalSubSecLemmaBox}, we complete the proof of Lemma \ref{EndLinked4PathBoxLemmaState}.
\end{enumerate}

\makeatletter
\renewcommand{\thetheorem}{\thesubsection.\arabic{theorem}}
\@addtoreset{theorem}{subsection}
\makeatother

\subsection{Preliminary restrictions}\label{BoxLemmaBasicFactsSubSec}

\begin{Claim}\label{MidPathChordOnlyQ0Q1AdjLast} Every chord of $C$ is incident to at one of $q_0, q_1$. \end{Claim}

\begin{claimproof} Suppose not. Since no chord of $C$ is incident to $G$, there is a chord $uu'$ of $C$ with $u, u'\not\in V(\mathring{P})$. Let $G=G_0\cup G_1$, where $P\subseteq G_0$ and $G_0\cap G_1=P$. Note that no chord of the outer face of $G_0$ is incident to $z$. By the minimality of $G$, there is an endpoint $p$ of $P$ such that $\Omega_{G}(P, [p])=\varnothing$ and $\Omega_{G_1}(P, [p])\neq\varnothing$, so let $\phi\in\Omega_{G_1}(P, [p])$. Since $\phi$ satisfies 1) of Theorem \ref{thomassen5ChooseThm}, but $\phi\not\in\Omega_{G}(P, [p])$, there is a pair $\{\phi_0, \phi_1\}$ of extensions of $\phi$ to $L$-colorings of $\textnormal{dom}(\phi)\cup\{q_0, q_1\}$, where $\mathcal{C}(\phi_0)$ and $\mathcal{C}(\phi_1)$ are disjoint singletons, yet it follows from Definition \ref{PartialLColOmega} that, for each $k\in\{0,1\}$, $\mathcal{C}_{G}^P(\phi_k)$ and $\mathcal{C}_{G_1}^P(\phi_k)$ are precisely equal, contradicting the fact that $\phi\in\Omega_{G_1}(P, [p])$. \end{claimproof}

\begin{Claim}\label{EachColorPhiP-wCLeq1} For any $L$-coloring $\phi$ of $\{p_0, p_1\}$, $\phi$ extends to two distinct $L$-colorings $\psi$ and $\psi'$ of $V(P-w)$ such that $\mathcal{C}(\psi)$ and $\mathcal{C}(\psi')$ are disjoint singletons.\end{Claim}

\begin{claimproof} Since $\textnormal{dom}(\phi)\cap N(q_i)\subseteq\{p_0, p_1\}$ for each $i=0,1$, and since $\Omega_G(P, [p_k])=\varnothing$ for some $k\in\{0,1\}$, this is immediate from the fact that $G$ is a counterexample. \end{claimproof}

\begin{Claim}\label{Q0Q1NoNbrExceptwCL} $\textnormal{deg}(w)\geq 4$. In particular, $q_0q_1\not\in E(G)$ and $q_0, q_1$ have no common neighbor in $G$, except for $w$.  \end{Claim}

\begin{claimproof}  Suppose toward a contradiction that $\textnormal{deg}(w)\leq 3$. Since $G$ is a counterexample, there is an $L$-coloring $\phi$ of $V(P-w)$ with $|\mathcal{C}_P^G(\phi')|=1$. Since $\mathcal{C}_P^G(\phi')$ is nonempty, $\phi'$ extends to an $L$-coloring of $G$, so $\phi$ extends to an $L$-coloring $\psi$ of $G-w$, and since $|\textnormal{deg}(w)|\leq 3$, $\psi$ extends to at least two different $L$-colorings of $G$ which use different colors on $w$, so $|\mathcal{C}(\phi')|\geq \mathcal{C}(\psi)|\geq 2$. which is false. Thus, $\textnormal{deg}(w)\geq 4$. Since $G$ is short-separation-free, it immediately follows that $q_0, q_1$ have no common neighbor other than $w$ and that $q_0q_1\not\in E(G)$. \end{claimproof}

\begin{Claim}\label{HiSatConCl} For each $i\in\{0,1\}$, if $K_{1-i}$ is not an edge and $q_ix$ is a chord of $C$ incident to $q_i$, then, letting $H$ be the subgraph of $G$ bounded by outer cycle $(p_i(C\setminus\mathring{P})x)q_{1-i}wq_i$, there exists a family of $|L(p_{i})|+(|L(x)|-3)$ different elements of $\Omega_H(p_iq_iwq_{1-i}x, [p_i])$, each of which uses a different color on $x$. \end{Claim}

\begin{claimproof} Note that no chord of the outer face of $H$ is incident to $w$, and, by Claim \ref{Q0Q1NoNbrExceptwCL}, $x\not\in N(q_i)$. In particular $x\neq p_i$, so $x$ is an internal vertex of the path $C\setminus\mathring{P}$, and $|L(x)|\geq 3$, and since $q_ix$ is a chord of $C$ we have $|V(H)|<|V(G)|$. By Claim \ref{Q0Q1NoNbrExceptwCL}, $p_iq_{1-i}\not\in E(G)$, so $u_{1-i}^{\star}\in V(C\setminus P)$ and $|L(x)|\geq 3$. Since $1\leq |L(p_i)|\leq 3$, it follows from the minimality of $G$ that there exists a family of $|L(p_{i})|+(|L(x)|-3)$ different elements of $\Omega_{H_{1-i}}(p_iq_iwq_{1-i}x, [p_i])$, each of which uses a different color on $u_{1-i}^{\star}$. \end{claimproof}

We now introduce some additional notation. 

\begin{defn} \emph{We let $C\setminus\mathring{P}=p_0u_1\ldots u_tp_1$ for some $t\geq 0$. For each $i\in\{0,1\}$, we define the following.}
\begin{enumerate}[label=\emph{\arabic*)}] 
\itemsep-0.1em
\item\emph{We let $u^{\star}_i$ be the unique vertex of $N(q_i)\cap V(C\setminus\mathring{P})$ which is farthest from $p_i$ on the path $C\setminus\mathring{P}$.}
\item\emph{We let $K_i$ be the subgraph of $G$ with outer face $(p_i(C\setminus\mathring{P})u_i^{\star})q_i$. That is, if $u_i^{\star}=p_i$, then $K_i$ is an edge, and otherwise the outer face of $K_i$ is a cycle.}
\item\emph{We let $H_i$ be the subgraph of $G$ bounded by outer cycle $(u_i^{\star}(C\setminus\mathring{P})p_{1-i}q_{1-i}wq_i$. That is, $K_i\cup H_i=G$ and $K_i\cap H_i=q_iu_i^{\star}$.}
\item\emph{We let $P_i$ be the 4-path $p_{1-i}q_{1-i}wq_iu_i^{\star}$ on the outer cycle of $H_i$}
\end{enumerate}
\end{defn}

\begin{Claim}\label{AtMostOneChordToC-P} 
Both of the following hold.
\begin{enumerate}[label=\alph*)] 
\itemsep-0.1em
\item For each $i\in\{0,1\}$, either $\Omega_G(P, [p_i])\neq\varnothing$ or $K_{1-i}$ is an edge; AND
\item In particular, there is at least one $i\in\{0,1\}$ such that $K_{1-i}$ is an edge. 
\end{enumerate}
 \end{Claim}

\begin{claimproof} Suppose that a) does not hold, and suppose without loss of generality that $K_1$ is not an edge, but $\Omega_{G}(P, [p_0])=\varnothing$. By Claim \ref{HiSatConCl}, there exist $|L(p_0)|$ different elements of $\Omega_{H_1}(P_1, [p_0])$, each of which uses a different color on $u_1^{\star}$. By Corollary \ref{GlueAugFromKHCor}, there exists $\phi\in\Omega_{H_1}(P_1, [p_0])$ and a $\psi\in\textnormal{End}(u_1^{\star}q_1p_1, K_1)$ with $\phi(u_1^{\star})=\psi(u_1^{\star})$. Now, $\phi\cup\psi$ is a proper $L$-coloring of its domain, which is $\textnormal{dom}(\phi)\cup\{p_1\}$. In particular, this is a subset of $C\setminus\mathring{P}$. Since $\phi\in\Omega_{H_1}(P_1, [p_0])$, we have $\{p_0, u_1^{\star}\}\subseteq\textnormal{dom}(\phi)\subseteq N(q_0)\cup N(q_1)$. We also have $N(q_0)\cap\textnormal{dom}(\phi\cup\psi)\subseteq\{p_0, u_1^{\star}\}$. By Claim \ref{Q0Q1NoNbrExceptwCL}, $q_0$ is not adjacent to $u_1^{\star}$, so $N(q_0)\cap\textnormal{dom}(\phi\cup\psi)\subseteq\{p_0\}$. Since $\Omega_G(P, [p_0])=\varnothing$, we have $\phi\cup\psi\not\in\Omega_G(P, [p_0])$, so it follows that $\phi\cup\psi$ extends to a pair of $L$-colorings $\sigma_0, \sigma_1$ of $\textnormal{dom}(\phi\cup\psi)\cup\{q_0, q_1\}$, where $\mathcal{C}(\sigma_0)$ and $\mathcal{C}(\sigma_1)$ are disjoint sets of size precisely one. For each $k\in\{0,1\}$, let $\sigma_k'$ be the restriction of $\sigma_k$ to $\textnormal{dom}(\sigma_k)\setminus\{p_1\}$. For each $k=0,1$, since $\mathcal{C}(\sigma_k)$ is nonempty, $\mathcal{C}^{P_1}_{H_1}$ is also nonempty. Since $\phi\in\Omega_{H_1}(P_1, [p_0])$, we have either $\mathcal{C}^{P_1}_{H_1}(\sigma'_0)\cap\mathcal{C}^{P_1}_{H_1}(\sigma_1')\neq\varnothing$ or at least one of $\mathcal{C}^{P_1}_{H_1}(\sigma'_0), \mathcal{C}^{P_1}_{H_1}(\sigma_1')$ has size at least two. Since $\psi\in\textnormal{End}(u_1^{\star}q_1p_1, K_1)$, it follows that either $\mathcal{C}(\sigma_0)\cap\mathcal{C}^P_G(\sigma_1)\neq\varnothing$ and or at least one of $\mathcal{C}(\sigma_0), \mathcal{C}(\sigma_1)$ has size at least two, which is false. This proves a). Since $G$ is a counterexample, b) immediately follows from a). \end{claimproof}

\begin{Claim}\label{u1-istarNotAdjpi} For each $i\in\{0,1\}$, $p_iu^{\star}_{1-i}$ is not an edge of $G$. In particular, $|V(C)|>5$ and $p_0p_1\not\in E(G)$. \end{Claim}

\begin{claimproof} Suppose not, and suppose without loss of generality that $p_0u^{\star}_1\in E(G)$.  By Claim \ref{MidPathChordOnlyQ0Q1AdjLast}, $p_0u^{\star}_1\in E(C)$. By Claim \ref{EachColorPhiP-wCLeq1}, there is an $L$-coloring $\phi'$ of $V(P-w)$ with $|\mathcal{C}(\phi')|=1$. Since $\mathcal{C}(\phi')$ is nonempty, $\phi'$ extends to an $L$-coloring $\psi$ of $V(P-w)\cup V(K_1)$, and $\mathcal{C}(\psi)\subseteq\mathcal{C}(\phi')$. Since $|L_{\psi}(w)|\geq 3$, but $|\mathcal{C}(\psi)|\leq 1$, there is an $L$-coloring of the set of vertices $\{p_0, u_1^{\star}, q_1, w, p_1\}$ which does not extend to $L$-color the interior of the 5-cycle $p_0u_1^{\star}q_1wp_1$. By Theorem \ref{BohmePaper5CycleCorList}, there is a vertex of $G\setminus C$ adjacent to every vertex of the cycle $p_0u_1^{\star}q_1wp_1$, contradicting Claim \ref{Q0Q1NoNbrExceptwCL} \end{claimproof}

By Claim \ref{u1-istarNotAdjpi}, $p_0p_1\not\in E(G)$, so, for any $(a,b)\in L(p_0)\times L(p_1)$, there is an $L$-coloring of $\{p_0, p_1\}$ using $a,b$ on the respective vertices $p_0, p_1$. Applying Claim \ref{EachColorPhiP-wCLeq1}, we introduce the following notation.

\begin{defn}\label{SetFForeachAb01} \emph{For any $a\in L(p_0)$ and $b\in L(p_1)$, we let $\phi_{ab}^0, \phi_{ab}^1$ be $L$-colorings of $V(P-w)$ using $a,b$ on the respective vertices $p_0, p_1$, where $\mathcal{C}(\phi_{ab}^0)$ and $\mathcal{C}(\phi_{ab}^1)$ are disjoint singletons. Note that, since $\mathcal{C}(\phi_{ab}^0)$ and $\mathcal{C}(\phi_{ab}^1)$ are disjoint and nonempty, $\phi_{ab}^0$ and $\phi_{ab}^1$ are distinct. We also let $\mathcal{F}=\{\phi_{ab}^k: (a,b)\in L(p_0)\times L(p_1)\ \textnormal{and}\ k\in\{0,1\}\}$.}
\end{defn}

\subsection{Dealing with 6-cycles}\label{PrelimFor6CycleSub}

\begin{Claim}\label{NoVertexVReachesAcrossFromPiToU1-iStar} For each $i\in\{0,1\}$, $p_i, u_{1-i}^{\star}$ have no common neighbor in $G\setminus C$. \end{Claim}

\begin{claimproof} Suppose for the sake of definiteness that $i=0$ and that $p_0, u_1^{\star}$ have a common neighbor $v\in V(G\setminus C)$.  $G$ contains the 6-cycle $D:=p_0vu_1^{\star}q_1wq_0$. Let $H^{\downarrow}$ be the subgraph of $G$ bounded by outer cycle $(p_0(C\setminus\mathring{P})u_1^{\star})v$ and $H_{\uparrow}=\textnormal{Int}_G(D)$. This is illustrated in Figure \ref{V0V1CommonNbrHUpFigure}, where $D$ is in bold (possibly $u_1^{\star}=p_1$ and $K_1$ is an edge). 

\vspace*{-8mm}
\begin{addmargin}[2em]{0em}
\begin{subclaim}\label{NoChordDSubCLForVstar} There is no chord of $D$. \end{subclaim}

\begin{claimproof} By Claim \ref{u1-istarNotAdjpi}, any chord of $D$, if it exists, lies in $H^{\uparrow}$. By Claim \ref{Q0Q1NoNbrExceptwCL}, $v$ is adjacent to at most one of $q_0, q_1$, and, by assumption, $w$ is adjacent to neither $p_0$ nor $u_1^{\star}$. Thus, if $wv\in E(G)$, then each of the two 4-cycles $wvp_0q_0$ and $wvu_1^{\star}q_1$ is induced. Since $G$ is short-separation-free, this contradicts our triangulation conditions, so $wv\not\in E(G)$. We conclude that any chord of $D$, if it exists, lies in $\{q_0v, q_1v\}$, and there is at most one such chord, so suppose  there is a $k\in\{0,1\}$ such that $q_kv$ is the lone chord of $D$ and let $D'$ be the unique induced 5-cycle whose vertices lie in $V(D)$. Since $G$ is short-separation-free, we have $V(\textnormal{Int}_G(D))=V(\textnormal{Int}_G(D'))\cup V(D\setminus D')$. Now we choose an arbitrary $\phi\in\mathcal{F}$. Since $\mathcal{C}(\phi)$ is nonempty, $\phi$ extends to an $L$-coloring $\psi$ of $V(\textnormal{Ext}_D(G))\setminus\{w\}$, and $|\mathcal{C}(\psi)|\leq 1$. Since no chord of $C$ is incident to $w$, we have $|L_{\psi}(w)|\geq 3$, so there is an $L$-coloring of $V(D')$ which does not extend to $L$-color the interior of $D'$. By Theorem \ref{BohmePaper5CycleCorList}, there is a vertex of $G\setminus C$ adjacent to all five vertices of $D'$, and, in particular, this vertex is adjacent to all three of $q_0, w, q_1$, contradicting Claim \ref{Q0Q1NoNbrExceptwCL}. \end{claimproof}\end{addmargin}

We now show that there is no vertex of $H^{\uparrow}\setminus D$ adjacent to all of $p_0, v, u_1^{\star}$. 

\vspace*{-8mm}
\begin{addmargin}[2em]{0em}
\begin{subclaim}\label{IFzAdjAllThreeExt} If there is a $z\in V(H^{\uparrow})\setminus V(D)$ adjacent to all three of $p_0, v, u_1^{\star}$, then, letting $H'$ be the subgraph of $G$ bounded by outer cycle $(p_0(C\setminus\mathring{P})u_1^{\star}z$, there is an $L$-coloring $\psi$ of $V(K_1)\cup V(P-w)$ such that $\mathcal{C}(\psi)|\leq 1$, where $\Lambda_{H'}^{p_0zu_1^{\star}}(\psi(p_0), \bullet, \psi(u_1^{\star}))=L_{\psi}(z)$. \end{subclaim}

\begin{claimproof} Note that $H'$ is not a broken wheel with principal path $p_0zu_1^{\star}$, and the outer cycle of $H'$ is induced. By 1) of Theorem \ref{EitherBWheelOrAtMostOneColThm}, there is at most one $L$-coloring of $p_0zu_1^{\star}$ which does not extend to an $L$-coloring of $H'$. Consider the following cases. 

\textbf{Case 1:} $|L(p_0)|\geq 2$

In this case, there is a $c\in L(p_0)$ such that any $L$-coloring of $p_0zu_1^{\star}$ using $c$ on $p_0$ extends to $L$-color $H'$. We now choose a $\phi\in\mathcal{F}$ with $\phi(p_0)=c$. Since $\mathcal{C}(\phi)$ is nonempty, $\phi$ extends to an $L$-coloring $\psi$ of $V(P-w)\cup V(K_1)$, and $|\mathcal{C}(\psi)|\leq 1$, as $\mathcal{C}(\psi)\subseteq\mathcal{C}(\phi)$, so we are done in this case.

\textbf{Case 2:} $|L(p_0)|=1$

In this case, we have $|L(u_1^{\star})|=3$, even if $u_1^{\star}=p_1$, so there is an $S\subseteq L(u_1^{\star})$ with $|S|=2$, where, for any $s\in S$, any $L$-coloring of $p_0vu_1^{\star}$ using $s$ on $u_1^{\star}$ extends to $L$-color $H'$. If $u_1^{\star}=p_1$, then we choose a $\psi\in\mathcal{F}$ with $\psi(p_1)\in S$, and since $V(K_1)\subseteq V(P-w)$, we are done in that case, so suppose that $u_1^{\star}\neq p_1$. Since $|L(p_1)|=3$, there exist three distinct elements $\phi_0, \phi_1, \phi_2$ of $\mathcal{F}$, each of which uses a different color on $p_1$, where $|\mathcal{C}(\phi_k)|=1$ for each $k=0,1,2$. Since $|L(u_1^{\star})\setminus S|=1$ and each $\phi_k$ uses a different color on $p_1$, it follows that there exists a $k\in\{0,1,2\}$ with $S\cap\Lambda_{K_1}(\bullet, \phi_k(q_1), \phi_k(p_1))\neq\varnothing$, or else we contradict 1c) of Theorem \ref{BWheelMainRevListThm2}. Thus, $\phi_k$ extends to an $L$-coloring $\psi$ of $V(P-w)\cup V(K_1)$ with $\psi(u_1^{\star})\in S$ and $|\mathcal{C}(\psi)|\leq 1$, so we are done. \end{claimproof}\end{addmargin}

\begin{center}\begin{tikzpicture}

\node[shape=circle,draw=black] (p0) at (-4, 0) {$p_0$};
\node[shape=circle,draw=white] (mid) at (-3, 0) {$\ldots$};
\node[shape=circle,draw=white] (mid+) at (0, 0) {$\ldots$};
\node[shape=circle,draw=black] (v1) at (1, 0) {$u_1^{\star}$};
\node[shape=circle,draw=white] (un+) at (2, 0) {$\ldots$};
\node[shape=circle,draw=black] (ut) at (3, 0) {$u_t$};
\node[shape=circle,draw=black] (p1) at (4, 0) {$p_1$};
\node[shape=circle,draw=black] (q0) at (-3,2) {$q_0$};
\node[shape=circle,draw=black] (q1) at (3,2) {$q_1$};
\node[shape=circle,draw=black] (w) at (0,4) {$w$};
\node[shape=circle,draw=black] (v) at (-1.5, 1.5) {$v$};
\node[shape=circle,draw=white] (K1) at (2.8, 1) {$K_1$};
\node[shape=circle,draw=white] (Hdown) at (-1.5, 0.6) {$H_{\downarrow}$};
\node[shape=circle,draw=white] (Hdown) at (0, 2.5) {$H^{\uparrow}$};
 \draw[-] (p1) to (ut);
 \draw[-, line width=1.8pt] (p0) to (q0) to (w) to (q1) to (v1);
\draw[-] (q1) to (p1);
\draw[-]  (p0) to (mid) to (mid+) to (v1) to (un+) to (ut);
\draw[-, line width=1.8pt] (p0) to (v) to (v1);

\end{tikzpicture}\captionof{figure}{}\label{V0V1CommonNbrHUpFigure}\end{center}

 Applying Subclaim \ref{IFzAdjAllThreeExt}, we have the following. 

\vspace*{-8mm}
\begin{addmargin}[2em]{0em}
\begin{subclaim}\label{notAdjAllThreep0vu1star} There is no vertex of $H^{\uparrow}\setminus D$ adjacent to all three of $p_0, v, u_1^{\star}$ \end{subclaim}

\begin{claimproof} Suppose there is a $z\in V(H^{\uparrow}\setminus D)$ adjacent to all three of $p_0, v, u_1^{\star}$ and let $\psi, H'$ be as in the statement of Subclaim \ref{IFzAdjAllThreeExt}. Possibly $z\in N(w)$, but, in any case, $|L_{\psi}(w)|\geq 2$. Since $\psi$ extends to an $L$-coloring of $\textnormal{Ext}_G(D)$ and $|\mathcal{C}(\psi)|=1$, it follows from Theorem \ref{BohmePaper5CycleCorList} that $1\leq |V(H^{\uparrow}\setminus D)|\leq 3$. By Claim \ref{Q0Q1NoNbrExceptwCL},  no vertex of $H^{\uparrow}\setminus D$ is adjacent to all three of $q_0, w, q_1$. If $z\in N(w)$, then, since $w$ is adjacent to neither $p_0$ nor $u_1^{\star}$ and $G$ contains no induced 4-cycles, $z$ is adjcent to all six vertices of $D$, a contradiction. Thus, $z\not\in N(w)$. Since every face of $G$ other than $C$ is bounded by a triangle, we have $V(H^{\uparrow}\setminus D)\neq\{z\}$, so $2\leq |V(H^{\uparrow}\setminus D)|\leq 3$. It also follows from Theorem \ref{BohmePaper5CycleCorList} that each vertex of $H^{\uparrow}\setminus D$ has degree precisely five. In particular, since no vertex of $V(H^{\uparrow}\setminus D)$ is adjacent to all three of $q_0, w, q_1$, there exists an edge $z_0z_1$, where $H^{\uparrow}\setminus D$ is the triangle $zz_0z_1$, and $N(z_0)=\{z, z_1\}\cup\{p_0, q_0, w\}$ and $N(z_1)=\{z, z_0\}\cup\{u_1^{\star}, q_1, w\}$. For any $c\in L_{\psi}(w)$, there is an $L_{\psi}$-coloring of $\{w, z, z_0, z_1\}$ using $c$ on $w$, since we have not yet colored $v$. Thus, it follows from our choice of $\psi$ that $\psi$ extends to an $L$-coloring of $G$ using $c$ on $w$, so $\mathcal{C}_P^G(\psi)=L_{\psi}(w)$ and $|\mathcal{C}_P^G(\psi)|\geq 3$, which is false. \end{claimproof}\end{addmargin}

\vspace*{-8mm}
\begin{addmargin}[2em]{0em}
\begin{subclaim}\label{ChordofDIncV} $H^{\uparrow}\setminus D$ is an edge, and furthermore, one endpoint of this edge is adjacent to every vertex of the 3-path $wq_0p_0v$, and the other endpoint of this edge is adjacent to every vertex of the 3-path $wq_1u_1^{\star}v$. \end{subclaim}

\begin{claimproof} By Subclaim \ref{NoChordDSubCLForVstar}, there are no chords of $D$. We now choose an arbitrary $\phi\in\mathcal{F}$. Since $\mathcal{C}(\phi)$ is nonempty, $\phi$ extends to an $L$-coloring $\phi^*$ of $V(\textnormal{Ext}_G(D))$. Since $w$ is incident to no chords of $C$, $|L_{\phi^*}(w)|\geq 3$. If there is a lone vertex $z$ of $H^{\uparrow}\setminus D$ adjacent to at least five vertices of $D$, then, since $D$ is induced and $G$ has no induced 4-cycles, it follows that $z$ is adjacent to all six vertices of $D$, contradicting Subclaim \ref{notAdjAllThreep0vu1star}. Thus, no such vertex exists. Since $|\mathcal{C}(\phi)|\leq 1$ and $|L_{\phi^*}(w)|\geq 3$, it follows from Theorem \ref{BohmePaper5CycleCorList} that $2\leq |V(H^{\uparrow}\setminus D)|\leq 3$, and $H^{\uparrow}\setminus D$ is either an edge or a triangle. If $H^{\uparrow}\setminus D$ is an edge, but does not satisfy the adjacency specified in the statement of Subclaim \ref{ChordofDIncV}, then one endpoint of this edge is adjacent to all three of $q_0, w, q_1$, contradicting Claim \ref{Q0Q1NoNbrExceptwCL}. Likewise, if $H^{\uparrow}\setminus D$ is a triangle, then it follows from Claim \ref{Q0Q1NoNbrExceptwCL} that one vertex of this triangle is adjacent to all three of $p_0, v, u_1^{\star}$, contradicting Subclaim \ref{notAdjAllThreep0vu1star}. \end{claimproof}\end{addmargin}

Applying Subclaim \ref{ChordofDIncV}, there are vertices $y, y^*$ such that $H^{\uparrow}\setminus D=yy^*$, where $y$ is adjacent to the 3-path $vp_0q_0w$ and $y^*$ is adjacent to the 3-path $wq_1u_1^{\star}v$.

\vspace*{-8mm}
\begin{addmargin}[2em]{0em}
\begin{subclaim} $K_1$ is not an edge. \end{subclaim}

\begin{claimproof} Suppose that $K_1$ is an edge, i.e $u_1^{\star}=p_1$. By Theorem \ref{SumTo4For2PathColorEnds}, there is a $\phi\in\textnormal{End}(p_0vp_1, H_{\downarrow})$. Since $\phi$ is an $L$-coloring of $\{p_0, p_1\}$, there is a $\phi'\in\mathcal{F}$ which is an extension of $\phi$. Let $c$ be an arbitrary color of $L_{\phi'}(w)$. Note that $\phi'$ extends to an $L$-coloring $\psi$ of $V(P)\cup\{y, y^*\}$ using $c$ on $w$, since we have not yet colored $v$. Furthermore, $v$ is adjacent to precisely for vertices of $\textnormal{dom}(\psi)$, so there is a color left over for $v$. It follows from our choice of $\phi$ that $\psi$ extends to an $L$-coloring of $G$, so we get $\mathcal{C}(\phi')=L_{\phi'}(w)$, which is false, as $|L_{\phi'}(w)|\geq 3$. \end{claimproof}\end{addmargin}

Since $K_1$ is not an edge, $u_1^{\star}q_1p_1$ is a 2-path.

\vspace*{-8mm}
\begin{addmargin}[2em]{0em}
\begin{subclaim}\label{AtMostOneColorDownFromK1Q1toU1Star} For any $L$-coloring $\psi$ of $V(C-w)$, if $|\mathcal{C}(\psi)|\leq 1$, then $|\Lambda_{K_1}(\bullet, \psi(q_1), \psi(p_1))|=1$. \end{subclaim}

\begin{claimproof} Suppose not. By Theorem \ref{thomassen5ChooseThm}, $\Lambda_{K_1}(\bullet, \psi(q_1), \psi(p_1))$ is nonempty, so there is an $S\subseteq \Lambda_{K_1}(\bullet, \psi(q_1), \psi(p_1))$ with $|S|=2$. Let $T_v=L_{\psi}(v)\setminus S$ and $T_{y^*}=L_{\psi}(y')\setminus S$, and let $T_w=L_{\psi}(w)\setminus\mathcal{C}_G^P(\psi)$. Since $D$ has no chords and we have not yet colored $u_1^{\star}$, we have $|L_{\psi}(v)|\geq 4$ and $|L_{\psi}(y^*)|\geq 2$, so each of $T_v$ and $T_{y^*}$ has size at least two. Likewise, since $|L_{\psi}(w)|\geq 3$, we have $|T_w|\geq 2$. Since $vy^*w$ is an induced path and $|L_{\psi}(y)|\geq 3$, there is an $L_{\psi}$-coloring $\sigma$ of $vy^*w$, where $\sigma(x)\in T_x$ for each $x\in\{v, y^*, w\}$, and $\sigma$ uses at most two colors of $L_{\psi}(y)$, i.e $|L_{\psi\cup\sigma}(y)|\geq 1$. Thus, $\psi\cup\sigma$ extends to an $L$-coloring $\tau$ of $V(P)\cup\{y, y^*\}$. Since $\tau(v)\not\in S$, we have $\Lambda_{H_{\downarrow}}^{p_0vu_1^{\star}}(\tau(p_0), \tau(v), \bullet)\cap S\neq\varnothing$ by Corollary \ref{2ListsNextToPrecEdgeCor}. Since $\tau(y^*)\not\in S$, it follows that $\tau$, and thus $\psi$, extends to $L$-color $G$, which is false, since $\tau(w)\not\in\mathcal{C}_G^P(\psi)$.  \end{claimproof} \end{addmargin}

Now let $\phi$ be an arbitrary $L$-coloring of $\{p_0, p_1\}$ and let $\phi^0, \phi^1\in\mathcal{F}$, where $\phi^0$ and $\phi^1$ are distinct extensions of $\phi$. For each $k\in\{0,1\}$, since $\mathcal{C}(\phi^k)$ is nonempty, $\phi^k$ extends to $L$-color $G$. We have $|L(y)|=|L(y^*)|=5$, or else, if one of $y, y^*$ has a list of size greater than five, then, for each $k\in\{0,1\}$, $\mathcal{C}(\phi^k)=L_{\phi}(w)$ and $|\mathcal{C}(\phi^k)|\geq 3$, which is false. 

\vspace*{-8mm}
\begin{addmargin}[2em]{0em}
\begin{subclaim}\label{TwoFactsPhiKQ0Q1} For each $k\in\{0,1\}$, both of the following hold.
\begin{enumerate}[label=\arabic*)]
\itemsep-0.1em
\item $\phi^k(q_0)\in L(y)$ and $\phi^k(q_1)\in L(y^*)$; AND
\item  $\phi^k(q_0)\neq\phi^k(q_1)$
\end{enumerate} \end{subclaim}

\begin{claimproof} Let $k\in\{0,1\}$. Suppose that $k$ does not satisfy 1). Note that $\phi^k$ extends to an $L$-coloring of $G$ and thus extends to an $L$-coloring $\psi$ of $G\setminus\{y, w, y^*\}$. Since either $\phi^k(q_0)\in L(y)$ or $\phi^k(q_1)\in L(y^*)$, we have $|\mathcal{C}(\phi^k)|\geq |\mathcal{C}(\psi)|\geq |L_{\psi}(w)|\geq 3$, which is false. This proves 1). Now we prove 2). Suppose that $\phi^k(q_0)=\phi^k(q_1)$. Now, $\phi^k$ extends to an $L$-coloring of $G$ and thus extends to an $L$-coloring $\psi$ of $G-w$. Since $N(w)=\{q_0, y, y^*, q_1\}$ and $\psi$ uses the same color on $q_0, q_1$, it follows that $\psi$ extends to two different $L$-colorings of $G$ using different colors on $w$, so $|\mathcal{C}(\phi^k)|\geq 2$, which is false. \end{claimproof}\end{addmargin}

We now let $A_0=\{\phi^0(q_0), \phi^1(q_0)\}$ and $A_1=\{\phi^0(q_1), \phi^1(q_1)\}$. 

\vspace*{-8mm}
\begin{addmargin}[2em]{0em}
\begin{subclaim}\label{phi0andphi1UseDiffColorsOnQ1} $|A_1|=2$ \end{subclaim}

\begin{claimproof} Suppose not. Thus, $\phi^0(q_1)=\phi^1(q_1)=c$ for some color $c$, and let $c'\in\Lambda_{K_1}(\bullet, c, \phi(p_1))$. By Subclaim \ref{AtMostOneColorDownFromK1Q1toU1Star}, we have $\Lambda_{K_1}(\bullet, c, \phi(p_1))=\{c'\}$. Now, since $|L(v)\setminus\{\phi(p_0), c'\}|\geq 3$, it follows from 3) of Corollary \ref{CorMainEitherBWheelAtM1ColCor} that there is a $d\in\Lambda_{H_{\downarrow}}^{p_0vu_1^{\star}}(\phi(p_0), \bullet, c')$. Note that this is permissible even if $\phi(p_0)=c'$, since, by Claim \ref{u1-istarNotAdjpi}, $H_{\downarrow}$ is not a triangle. Thus, for each $k=0,1$, $\phi^k$ extends to an $L$-coloring $\psi^k$ of $V(\textnormal{Ext}_G(D))\setminus\{w\}$ using $c' d$ on the respective vertices $u_1^{\star}, v$.  Now, since $\phi^0(q_1)=\phi^1(q_1)=c$, we have $\phi^0(q_0)\neq\phi^1(q_0)$, or else $\phi^0=\phi^1$, which is false, since $\mathcal{C}(\phi^k)$ and $\mathcal{C}(\phi^k)$ are disjoint and nonempty. Thus, $|A_0|=2$. We have $|L(y)\setminus\{\phi(p_0), d\}|\geq 3$ and $L(y^*)\setminus\{c, c', d\}|\geq 2$. As shown above, $|A_0|=2$, so there is a $k\in\{0,1\}$ and an $f\in L(y^*)\setminus\{c, c', d\}$ with $L(y)\setminus\{\phi(p_0), d, \phi^k(q_0), f\}|\geq 2$. But then, coloring $y^*$ with $f$, it follows that each color of $L_{\psi^k}(w)\setminus\{f\}$ lies in $\mathcal{C}_P^G(\psi^k)$. Since $|L_{\psi^k}(w)\setminus\{f\}|\geq 2$, we have $|\mathcal{C}_P^G(\psi^k)|\geq 2$ and thus $|\mathcal{C}(\phi^k)|\geq 2$, which is false.  \end{claimproof}\end{addmargin}

\vspace*{-8mm}
\begin{addmargin}[2em]{0em}
\begin{subclaim}\label{phi0andphi1MapDistinctSingletonSubcL} $\Lambda_{K_1}(\bullet, \phi^0(q_1), \phi(p_1))$ and $\Lambda_{K_1}(\bullet, \phi^1(q_1), \phi(p_1))$ are distinct singletons. \end{subclaim}

\begin{claimproof}  By Subclaim \ref{AtMostOneColorDownFromK1Q1toU1Star}, $\Lambda_{K_1}(\bullet, \phi^k(q_1), \phi(p_1))$ is a singleton for each $k=0,1$. Suppose that Subclaim \ref{phi0andphi1MapDistinctSingletonSubcL} does not hold. Thus, there is a $c\in L(u_1^{\star})$ such that $\Lambda_{K_1}(\bullet, \phi^k(q_1), \phi(p_1))=\{c\}$ for each $k\in\{0,1\}$. By 3) of Corollary \ref{CorMainEitherBWheelAtM1ColCor}, since $|L(v)\setminus\{\phi(p_0), c\}|\geq 3$, there is a $d\in\Lambda_{H_{\downarrow}}^{p_0vu_1^{\star}}(\phi(p_0), \bullet, c)$. Note that $d\in L(y)\cap L(y^*)$ and, for each $k\in\{0,1\}$, we have $d\not\in\{\phi^k(q_0), \phi^k(q_1)\}$, or else, if either of these do not hold, then there is a $k\in\{0,1\}$ with $|\mathcal{C}(\phi^k)|\geq 3$ for each $k\in\{0,1\}$, which is false. Consider the following cases.

\textbf{Case 1:} $|A_0|=1$

In this case, we let $\phi^0(q_0)=\phi^1(q_0)=f$ for some color $f$. Since $|L(y^*)\setminus\{c, d\}|\geq 3$ and $|L(y)\setminus\{f, \phi(p_0), d\}|\geq 2$, there is an $L$-coloring $\sigma$ of $\{y^*\}\cup (V(D)\setminus\{w, q_1\})$ which uses $f, \phi(p_0), d, c$ on the respective vertices $q_0, p_0, v, u_1^{\star}$, where $L_{\sigma}(y)|\geq 2$. Since $\phi^0(q_1)\neq\phi^1(q_0)$, at least one color of $\{\phi^0(q_1), \phi^1(q_1)\}$ is distinct from $\sigma(y^*)$, so there exists a $k\in\{0,1\}$ such that $\phi^k\cup\sigma$ extends to an $L$-coloring $\tau$ of $V(G)\setminus\{y, w\}$. Since $|L_{\tau}(y)|\geq 2$, we have $\mathcal{C}(\tau)=L_{\tau}(w)$, so $|\mathcal{C}(\phi^k)|\geq\mathcal{C}_G^P(\tau)|\geq 3$, which is false. 

\textbf{Case 2:} $|A_0|=2$

We break this into two subcases.

\textbf{Subcase 2.1} $L(y)\cap L(y^*)\cap A_0\not\subseteq\{c\}$

In this case, there is a $k\in\{0,1\}$ with $\phi^k(q_0)\in L(y)\cap L(y^*)$ and $\phi^k(q_0)\neq c$. By 2) of Subclaim \ref{TwoFactsPhiKQ0Q1}, $\phi^k(q_0)\neq\phi^k(q_1)$. As indicated above, we have $d\neq\phi^k(q_0)$, so there is an extension of $\phi^k$ to an $L$-coloring $\psi$ of $V(P-w)\cup\{v, u_1^{\star}, y^*\}$ using $d, c, \phi^k(q_0)$ on the respective vertices $v, u_1^{\star}, y^*$. In particular, it follows from our choice of colors for $v, u_1^{\star}$ that $\psi$ extends to an $L$-coloring $\psi'$ of $V(G)\setminus\{w, y\}$, and $|L_{\psi'}(y)|\geq 2$, since two neighbors of $y$ are using the same color. But then we have $|\mathcal{C}(\phi^k)|\geq\mathcal{C}(\psi')|$ and $\mathcal{C}(\psi')=L_{\psi'}(w)$. As we have not yet colored $y$, we have $|\mathcal{C}(\phi^k)|\geq |L_{\psi'}(w)|\geq 2$. which is false.

\textbf{Subcase 2.2} $L(y)\cap L(y^*)\cap A_0\subseteq\{c\}$

Recall that $|L(y)|=|L(y^*)|=5$. In this case, since $|A_0|=2$ by the assumption of Case 2, and since $\{\phi^0(q_0), \phi^1(q_0)\}\subseteq L(y)$ by 1) of Subclaim \ref{TwoFactsPhiKQ0Q1}, it follows that there is an $f\in L(y^*)\setminus L(y)$ with $f\neq c$. Since $d\in L(y)\cap L(y^*)$, we have $f\neq d$. By Subclaim \ref{phi0andphi1UseDiffColorsOnQ1}, at least one color of $\{\phi^0(q_1), \phi^1(q_1)\}$ is distinct from $f$, so there exists a $k\in\{0,1\}$ such that $\phi^k$ extends to an $L$-coloring of $V(P-w)\cup\{v, u_1^{\star}, y^*\}$ using $d, c, f$ on the respective vertices $v, u_1^{\star}, y^*$, so $\phi^k$ extends to an $L$-coloring $\psi$ of $V(G)\setminus\{w, y\}$ with $\psi(y^*)=f$. In particular, $|L_{\psi}(y)|\geq 2$, so $\mathcal{C}(\psi)=L_{\psi}(w)$ and thus $|\mathcal{C}(\phi^k)|\geq\mathcal{C}_G^P(\psi)|\geq 3$, which is false. \end{claimproof}\end{addmargin}

It follows from Subclaim \ref{phi0andphi1MapDistinctSingletonSubcL}, together with 1b) of Theorem \ref{BWheelMainRevListThm2}, that $|L(u_1^{\star})|=3$ and $\Lambda_{K_1}(\bullet, \phi^k(q_1), \phi(p_1))=\{\phi^{1-k}(q_1)\}$ for each $k=0,1$. In particular, $\{\phi^0(q_1), \phi^1(q_1)\}\subseteq L(u_1^{\star})$ and $A_1\subseteq L(u_1^{\star})$. We now let $b$ be the lone color of $L(u_1^{\star})\setminus A_1$. 

\vspace*{-8mm}
\begin{addmargin}[2em]{0em}
\begin{subclaim}\label{ATwoColorsContainedInBoth} $A_1\subseteq L(v)\setminus\{\phi(p_0)\}$. Furthermore, for each $k\in\{0,1\}$, $\Lambda_{H_{\downarrow}}^{p_0vu_1^{\star}}(\phi(p_0), \phi^k(q_1), \bullet)=\{b\}$. \end{subclaim}

\begin{claimproof} Suppose that $A_1\not\subseteq L(v)\setminus\{\phi(p_0)\}$ and let $k\in\{0,1\}$, where $\phi^k(q_1)\not\in L(v)\setminus\{\phi(p_0)\}$. Thus, $|L(v)\setminus\{\phi(p_0), \phi^k(q_1)\}|\geq 4$. As indicated above, $\Lambda_{K_1}(\bullet, \phi^{1-k}(q_1), \phi(p_1))=\{\phi^k(q_1)\}$. Possibly $\phi(p_0)=\phi^k(q_1)$, but, in any case, since $H_{\downarrow}$ is not a triangle, it follows from 3) of Corollary \ref{CorMainEitherBWheelAtM1ColCor} that $|\Lambda_{H_{\downarrow}}^{p_0vu_1^{\star}}(\phi(p_0), \bullet, \phi^k(p_1))|\geq 2$. Since $|\mathcal{C}(\phi^{1-k})|=1$, there is a $c\in L_{\phi^{1-k}}(w)\setminus\mathcal{C}_G^P(\phi^{1-k})$, and $\phi^{1-k}$ extends to an $L$-coloring $\psi$ of $V(P\cup K_1)$ using $c$ on $w$ and $\phi^k(q_1)$ on $u_1^{\star}$. Since $v$ only has four neighbors in $\textnormal{Ext}_G(D)$, and  $|\Lambda_{H_{\downarrow}}^{p_0vu_1^{\star}}(\phi(p_0), \bullet, \phi^k(p_1))|\geq 2$, it follows that $\psi$ extends to an $L$-coloring of $G$, which is false, since $c\not\in\mathcal{C}_G^P(\phi^{1-k})$. Thus, $A_1\subseteq L(v)\setminus\{\phi(p_0)\}$.

Now suppose toward a contradiction that there is a $k\in\{0,1\}$ with $\Lambda_{H_{\downarrow}}(\phi(p_0), \phi^k(q_1), \bullet)\neq\{b\}$. Since $L(u_1^{\star})=A_1\cup\{b\}$, we have $\phi^{1-k}(q_1)\in\Lambda_{H_{\downarrow}}(\phi(p_0), \phi^k(q_1), \bullet)$. Since $\Lambda_{K_1}(\bullet, \phi^k(q_1), \phi(p_1))=\{\phi^{1-k}(q_1)\}$, it follows that $\phi^k$ extends to an $L$-coloring $\psi$ of $V(G)\setminus\{w, y, y^*\}$ which uses $\phi^k(q_1)$ on both $v$ and $q_1$. In particular, as two neighbors of $y^*$ are using the same color, $\mathcal{C}(\psi)=L_{\psi}(w)$, so $|\mathcal{C}(\phi^k)|\geq |\mathcal{C}(\psi)|\geq 3$, which is false. \end{claimproof}\end{addmargin}

We now let $S=\bigcap_{k=0,1}\Lambda_{H_{\downarrow}}(\phi(p_0), \bullet, \phi^k(q_1))$.

\vspace*{-8mm}
\begin{addmargin}[2em]{0em}
\begin{subclaim}\label{SetSWorksForBothKAtLeastTwo} $|S|\geq 2$. \end{subclaim}

\begin{claimproof} By Subclaim \ref{ATwoColorsContainedInBoth}, $\Lambda_{H_{\downarrow}}^{p_0vu_1^{\star}}(\phi(p_0), \phi^k(q_1), \bullet)=\{b\}$ for each $k=0,1$, so it follows from 1) of Theorem \ref{EitherBWheelOrAtMostOneColThm} that $H_{\downarrow}$ is a broken wheel with principal path $p_0vu_1^{\star}$. Letting $u$ be the unique neighbor of $p_0$ on the path $H_{\downarrow}-v$, we have $u\neq u_1^{\star}$, since $H_{\downarrow}$ is not a triangle, and it follows from Subclaim \ref{ATwoColorsContainedInBoth} that $L(u)=\{\phi(p_0)\cup A_1$. Thus, for each each $s\in L(v)\setminus(\{\phi(p_0)\}\cup A_1)$, we have $s\not\in L(u)$, so $s\in S$, and $|S|\geq 2$.  \end{claimproof}\end{addmargin}

We now have enough to finish the proof of Claim \ref{NoVertexVReachesAcrossFromPiToU1-iStar}. Since $|A_0|\leq 2$ and $|A_1|\leq 2$, we choose a $c\in L(w)\setminus (A_0\cup A_1)$. Since $\mathcal{C}(\phi^0)\cap\mathcal{C}(\phi_1)=\varnothing$, there is a $k\in\{0,1\}$ with $c\not\in\mathcal{C}_G^P(\phi^k)$. Since $\Lambda_{K_1}(\bullet, \phi^k(q_1), \phi(p_1))=\{\phi^{1-k}(q_1)\}$, $\phi^k$ extends to an $L$-coloring $\psi$ of $V(P\cup K_1)\cup\{y, y^*\}$ using $\phi^{1-k}(q_1)$ on $u_1^{\star}$ and $c$ on $w$. By Subclaim \ref{SetSWorksForBothKAtLeastTwo}, $|S|\geq 2$, and since $v$ only has four neighbors in $\textnormal{dom}(\psi)$, there is a color of $S$ left over for $v$, so $\psi$ extends to an $L$-coloring of $G$, and thus $c\in\mathcal{C}_G^P(\phi^k)$, a contradiction. This completes the proof of Claim \ref{NoVertexVReachesAcrossFromPiToU1-iStar}. \end{claimproof}

Claim \ref{NoVertexVReachesAcrossFromPiToU1-iStar} has the following consequence. 

\begin{Claim}\label{ForEachPiWeExtToK0K1ToG} Let $i\in\{0,1\}$, where $K_i$ is an edge. Then the following hold. 
\begin{enumerate}[label=\arabic*)] 
\itemsep-0.1em
\item For any $y\in V(G\setminus C)$ with more than two neighbors on the path $P_{1-i}$, the neighborhood of $y$ on this path is either $p_iq_iw$ or $wq_{1-i}u_i^{\star}$; AND
\item For any $L$-coloring $\psi$ of $V(P-w)$, $\psi$ extends to an $L$-coloring $\psi'$ of $V(P-w)\cup V(K_{1-i})$, and, for any such $\psi'$, we have $|\mathcal{C}(\psi)|\geq |\mathcal{C}(\psi')|\geq |L(w)\setminus\{\psi(q_0), \psi(q_1)\}|-2\geq 1$. 
\end{enumerate}\end{Claim}

\begin{claimproof} Say $i=0$ for the sake of definiteness and et $y\in V(G\setminus C)$, where $|N(y)\cap V(P_1)|\geq 3$. Since $G$ has no induced 4-cycles and $P_1$ is an induced path, $G[N(y)\cap V(P_1)]$ is a subpath of $P_1$. By Claim \ref{Q0Q1NoNbrExceptwCL}, at most one of $q_0, q_1$ lies in $N(y)$, and, by Claim \ref{NoVertexVReachesAcrossFromPiToU1-iStar}, at most one of $p_0, u_1^{\star}$ lies in $N(y)$, so 1) follows. If 2) does not hold, then, by Lemma \ref{PartialPathColoringExtCL0}, there is a vertex of $G\setminus C$ with at least three neighbors in $P_1-w$, contradicting 1). \end{claimproof}

\subsection{Chords of $C$}\label{SubsecChordsIncToQi}

\begin{Claim}\label{KiNotEdgeBWTriContain} For each $i\in\{0,1\}$, where $u_i^{\star}\neq p_i$, at least one of the following holds.
\begin{enumerate}[label=\emph{\arabic*)}] 
\itemsep-0.1em
\item $K_i$ is not a broken wheel with principal path $u_i^{\star}q_ip_i$; OR
\item $K_i$ is a triangle; OR
\item Letting $x$ be the unique vertex of distance two from $p_i$ on the path $C-z$, $L(p_i)\not\subseteq L(x)$
\end{enumerate}
\end{Claim}

\begin{claimproof} Suppose for the sake of definiteness that $i=1$ and $u_1^{\star}\neq p_1$, i.e $K_1$ is not an edge, and suppose $K_1$ is a broken wheel with principal path $u_1^{\star}q_1p_1$, but $K_1$ is not a triangle. In particular, $x=u_{t-1}$ and $xq_1$ is a chord of $C$. We show that $L(p_1)\not\subseteq L(u_{t-1})$. Suppose that $L(p_1)\subseteq L(u_{t-1})$. Let $H$ be the subgraph of $G$ bounded by outer cycle $(p_0(C\setminus\mathring{P})u_{t-1})q_{1}wq_0$ and $Q$ be the 4-path $p_0q_0wq_1x$ on the outer cycle of $H$. By Claim \ref{HiSatConCl}, there exists a family of $|L(p_0)|+(|L(u_{t-1})|-3)$ different elements of $\Omega_H(Q, [u_{t-1}])$, each using a different color of $x$. As $L(p_1)\subseteq L(u_{t-1})$ and $|L(p_0)|+|L(p_1)|=4$, there exists a $\phi\in\Omega_H(Q, [u_{t-1}])$ with $\phi(u_{t-1})\in L(p_1)$. Let $c=\phi(p_0)$ and $d=\phi(u_{t-1})$, and let $\psi$ be the $L$-coloring of $\{p_0, p_1\}$ using $c,d$ on the respective vertices $p_0, p_1$. For each $i=0,1$, $N(q_i)\cap\textnormal{dom}(\psi)\subseteq\{p_0, p_1\}$, and, since $G$ is a counterexample, $\psi$ extends to two $L$-colorings $\psi_0, \psi_1$ of $V(P-w)$, where $\mathcal{C}(\psi_0)$ and $\mathcal{C}(\psi_1)$ are disjoint sets of size precisely one. For each $k=0,1$, $\psi_k$ does not use $d$ on $q_1$, and, by b) of Claim \ref{AtMostOneChordToC-P}, $K_0$ is an edge. Thus, for each $k=0,1$, $\phi\cup\psi_k$ is a proper $L$-coloring of its domain which uses $d$ on both of $u_{t-1}, p_1$, and there is a color left over for $u_t$, so $\phi\cup\psi_k$ extends to an $L$-coloring $\sigma_k$ of $\textnormal{dom}(\phi)\cup V(P-w)\cup\{u_t\}$. 

\vspace*{-8mm}
\begin{addmargin}[2em]{0em}
\begin{subclaim}\label{EachmathcalCPGNonempty} For each $k=0,1$, $\mathcal{C}(\sigma_k)\neq\varnothing$. \end{subclaim}

\begin{claimproof} Recall that, by Claim \ref{u1-istarNotAdjpi}, $p_0u_1^{\star}\not\in E(G)$, and, by Claim \ref{Q0Q1NoNbrExceptwCL}, $q_0u_1^{\star}\not\in E(G)$ and $q_0q_1\not\in E(G)$. Since the restriction of $\sigma_k$ to $\{p_1, q_1, u_t, u_{t-1}\}$ extends to $L$-color the broken wheel $K_1$, it follows that $\sigma_k$ extends to an $L$-coloring $\tau$ of $V(P-w)\cup V(K_1)$. Since $K_0$ is an edge, the outer cycle of $H_1$ is induced. Since $|L_{\tau}(w)|\geq 3$, it follows from Lemma \ref{PartialPathColoringExtCL0} applied to the 4-path $p_0q_0q_1u_1^{\star}$ on the outer cycle of $H_1$ that there is a $z\in V(H_1\setminus C^{H_1})$ with at least three neighbors in $\{p_0, q_0, q_1, u_1^{\star}\}$, contradicting 1) of Claim \ref{ForEachPiWeExtToK0K1ToG}. \end{claimproof}\end{addmargin}

For each $k=0,1$, $\mathcal{C}(\sigma_k)\subseteq\mathcal{C}(\psi_k)$, as $\sigma_k$ is an extension of $\psi_k$. For each $k=0,1$, let $\sigma_k'$ be the restriction of $\sigma_k$ to $\textnormal{dom}(\sigma_k)\setminus\{p_1, u_t\}$. Each of $\sigma_0', \sigma_1'$ is an extension of $\phi$ to an $L$-coloring of $\textnormal{dom}(\phi)\cup\{q_0, q_1\}$. Note that, for each $k=0,1$, $\mathcal{C}(\sigma_k)$ is precisely equal to $\mathcal{C}_H^Q(\sigma_k')$, so, by Subclaim \ref{EachmathcalCPGNonempty}, each of $\mathcal{C}_H^Q(\sigma_0')$ and $\mathcal{C}_H^Q(\sigma_1')$ is nonempty, and furthermore, for each $k=0,1$, $\mathcal{C}_H^Q(\sigma_k')\subseteq\mathcal{C}_G^P(\psi_k)$. Thus, $\mathcal{C}_H^Q(\sigma_0')$ and $\mathcal{C}_H^Q(\sigma_1')$ are disjoint singletons, which is false, as $\phi\in\Omega_H(Q, [u_{t-1}])$. \end{claimproof}

\begin{Claim}\label{KiEitherUniversalColForTriangle} For each $i\in\{0,1\}$, if $K_i$ is not an edge and $|L(p_i)|\geq 2$, then, letting $u$ be the unique neighbor of $p_i$ on the path $C-q_i$, at least one of the following holds.
\begin{enumerate}[label=\emph{\arabic*)}] 
\itemsep-0.1em
\item There is a $K_i$-universal color of $L(p_i)$; OR
\item $K_i$ is a triangle with $L(p_i)\subseteq L(y)$.
\end{enumerate}
 \end{Claim}

\begin{claimproof} Say $i=1$ without loss of generality. Suppose that $K_1$ is not an edge and there is no almost $K_1$-universal color of $L(p_1)$. By 2) of Corollary \ref{CorMainEitherBWheelAtM1ColCor}, $K_1$ is a broken wheel with principal path $u_1^{\star}q_1p_1$. If there is an $a\in L(p_1)\setminus L(u_t)$, then $a$ is an almost $K_1$-universal color of $L(p_1)$, which contradicts our assumption, so $L(p_1)\subseteq L(u_t)$. If $K_1$ is a triangle, we are done, so suppose that $K_1$ is not a triangle. By Claim \ref{KiNotEdgeBWTriContain}, $L(p_1)\not\subseteq L(u_{t-1})$, and, since $|L(p_1)|\geq 2$, there is an $a\in L(p_1)$ with $L(u_1)\setminus\{a\}\not\subseteq L(u_{t-1})$. By 4) of Theorem \ref{BWheelMainRevListThm2}, $a$ is an almost $K_1$-universal color of $L(p_1)$, contradicting our assumption. \end{claimproof}

\subsection{Dealing with 2-chords of $C$ of the form $p_izw$: part I}\label{VertexZAllofPiQiW}

We now deal with 2-chords of $C$ of the form $p_izw$. To do this, we introduce the notation below, analogous to the notation from the proof of Proposition \ref{InterMedUnobstrucGInduced}. 

\begin{defn}\label{DefHzGzAndPz} \emph{For each $i\in\{0,1\}$ and any $z\in V(G\setminus C)$ which is adjacent to each of $p_i, q_i, w$,.}
\begin{enumerate}[label=\emph{\arabic*)}] 
\itemsep-0.1em
\item\emph{we let $p^z$ be the unique vertex of $N(z)\cap V(C\setminus\mathring{P})$ which is farthest from $p_i$ on the path $C\setminus\mathring{P}$}; AND
\item\emph{We let $G_z$ be the subgraph of $G$ bounded by outer face $(p_i\ldots p^z)z$ and $H_z$ be the subgraph of $G$ bounded by outer cycle $(p^z(C\setminus\mathring{P})u_{1-i}^{\star})q_{1-i}wz$.}
\end{enumerate}
 \end{defn}

This is illustrated in Figure \ref{ZAdjAllThreeWp0q1}, where $i=0$. Possibly $N(z)\cap V(C\setminus\mathring{P})=\{p_0\}$ and $p^z=p_0$, in which case $G_z$ is the edge $p_0p^z$. Note that, in Definition \ref{DefHzGzAndPz}, the index $i$ is uniquely specified by Claim \ref{Q0Q1NoNbrExceptwCL}, so it is not necessary to include the index $i$ in the notation $G_z, H_z, p^z$. Recalling Definition \ref{DefnPPUnObstruc}, we now have the following.

\begin{Claim}\label{EvenNumberObstructionsHiCL} Let $i\in\{0,1\}$, where $K_i$ is an edge. Then there are either zero or two endpoints $p^*$ of $P_{1-i}$ such that $H_{1-i}$ is $(P_{1-i}, p^*)$-obstructed. \end{Claim}

\begin{claimproof} Suppose for the sake of definitness that $i=0$. The outer cycle of $H_1$ is induced, and it follows from Claim \ref{NoVertexVReachesAcrossFromPiToU1-iStar} that $N(p_0)\cap N(u_1^{\star})\subseteq V(C^{H_1})$. By Claim \ref{EachColorPhiP-wCLeq1}, $N(q_0)\cap N(q_1)=\{w\}$, so we can apply Proposition \ref{InterMedUnobstrucGInduced} to the rainbow $(H_1, C^{H_1}, P_1, L)$. We fix an arbitrary $\phi\in\mathcal{F}$. By 2) of Claim \ref{ForEachPiWeExtToK0K1ToG}, $\phi$ extends to an $L$-coloring $\phi'$ of $V(P-w)\cup V(K_1)$. Let $\psi$ be the restriction of $\phi'$ to $P_1-w$. Since $|L_{\psi}(w)|\geq 3$ and $|\mathcal{C}_{H_1}^{P_1}(\psi)|\leq 1$, it follows from P2) of Proposition \ref{InterMedUnobstrucGInduced} that there are either zero or two endpoints $p^*$ of $P_1$ such that $H_1$ is $(P_1, p^*)$-obstructed. \end{claimproof}

\begin{center}\begin{tikzpicture}

\node[shape=circle,draw=black] (p0) at (-4, 0) {$p_0$};
\node[shape=circle,draw=white] (mid) at (-3, 0) {$\ldots$};
\node[shape=circle,draw=black] (pz) at (-1.5, 0) {$p^z$};
\node[shape=circle,draw=white] (mid+) at (0, 0) {$\ldots$};
\node[shape=circle,draw=black] (v1) at (1, 0) {$u_1^{\star}$};
\node[shape=circle,draw=white] (un+) at (2, 0) {$\ldots$};
\node[shape=circle,draw=black] (ut) at (3, 0) {$u_t$};
\node[shape=circle,draw=black] (p1) at (4, 0) {$p_1$};
\node[shape=circle,draw=black] (q0) at (-3,2) {$q_0$};
\node[shape=circle,draw=black] (q1) at (3,2) {$q_1$};
\node[shape=circle,draw=black] (w) at (0,4) {$w$};
\node[shape=circle,draw=black] (z) at (-1.9, 1.9) {$z$};
\node[shape=circle,draw=white] (K1) at (2.8, 1) {$K_1$};
\node[shape=circle,draw=white] (Gz) at (-2.1, 0.8) {$G_z$};
\node[shape=circle,draw=white] (Hz) at (0.1, 0.8) {$H_z$};

 \draw[-] (p1) to (ut);
 \draw[-] (p0) to (q0) to (w) to (q1) to (v1);
\draw[-] (q1) to (p1);
\draw[-]  (p0) to (mid) to (pz) to (mid+) to (v1) to (un+) to (ut);
\draw[-] (p0) to (z) to (q0);
\draw[-] (z) to (w);
\draw[-] (z) to (pz);
\end{tikzpicture}\captionof{figure}{}\label{ZAdjAllThreeWp0q1}\end{center}

\begin{Claim}\label{EverySigmaNo2ChordFromW} Let $i\in\{0,1\}$, where $p_i, q_i, w$ have a common neighbor $z$ in $G\setminus C$. Suppose further that $z$ and $q_{1-i}$ have a common neighbor $y$ with $y\neq w$. Then $y\in V(G\setminus C)$ and $N(w)=\{q_0, q_1\}\cup\{z, y\}$. Furthermore, $y$ is adjacent to $u_{1-i}^{\star}$ and $|V(K_1)|>2$.  \end{Claim}

\begin{claimproof} Suppose for the sake of definitness that $i=0$. By Claim \ref{Q0Q1NoNbrExceptwCL}, $q_1z\not\in E(G)$, and since $G$ has no induced 4-cycles and $G$ is short-separation-free, we have $N(w)=\{q_0, z, y, q_1\}$. Furthermore, $y\in V(G\setminus C)$, as $w$ is incident to no chords of $C$. Now, $N(q_0)=\{p_0, w, z\}$, and $K_0$ is an edge. Since $y$ is adjacent to each of $z, q_1$, it follows that $H_1$ is $(p_0, P_1)$-obstructed. Thus, by Claim \ref{EvenNumberObstructionsHiCL}, $H_1$ is also $(u_1^{\star}, P_1)$-obstructed, so $y\in N(u_1^{\star})$.  To finish, we just need to check that $K_1$ is not an edge. Suppose that $K_1$ is an edge. Thus, $u_1^{\star}=p_1$ and $H_1=G$, and the outer cycle of $G\setminus\mathring{P}$ contains the 3-path $Q=p_0zyp_1$. Since $K_0$ is also an edge, the outer cycle of $G$ is induced. Note that, for any $L$-coloring $\psi$ of $\{p_0, p_1\}$, $|L_{\psi}(y)|\geq 4$, as $p_0\not\in N(y)$. Since $|L(p_0)|+|L(p_1)|=4$, it follows from 1) of Theorem \ref{CornerColoringMainRes} that there is an $L$-coloring $\psi$ of $\{p_0, p_1\}$ which extends to two different elements of $\textnormal{End}(y, Q, G\setminus\mathring{P})$. Let $S$ be a set of two colors of $L_{\psi}(y)$, where, for each $s\in S$, $\psi$ extends to an element of $\textnormal{End}(y, Q, G\setminus\mathring{P})$ using $s$ on $y$. We now choose an arbitrary $\phi\in\mathcal{F}$ whose restriction to $\{p_0, p_1\}$ is $\psi$. Since $|S\setminus\{\phi(q_1)\}|\geq 1$, $\phi$ extends to an $L$-coloring $\phi'$ of $V(P-w)\cup\{y\}$ with $\phi'(y)\in S$, so any extension of $\phi'$ to an $L$-coloring of $\textnormal{dom}(\phi')\cup\{z\}$ extends to $L$-color all of $G-w$. But now, since $z$ only has three neighbors in $\textnormal{dom}(\phi')$, it follows that $\mathcal{C}(\phi')=L_{\phi'}(w)$, so $|\mathcal{C}(\phi)|\geq\mathcal{C}(\phi')|\geq 2$, which is false, as $\phi\in\mathcal{F}$.\end{claimproof}

Applying Claim \ref{EverySigmaNo2ChordFromW}, we have the following claim, which makes up the remainder of Subsection \ref{VertexZAllofPiQiW}.

\begin{Claim}\label{PzYObstructionVert}  Let $i\in\{0,1\}$, where $p_i, q_i, w$ have a common neighbor $z$ in $G\setminus C$. Then $p^z$ and $u_{1-i}^{\star}$ have no common neighbor in $G\setminus C$. \end{Claim}

\begin{claimproof} We suppose without loss of generality that $i=0$ and suppose toward a contradiction that there is a vertex $y\in V(G\setminus C)$ which is adjacent to both of $p^z, q_1$. If necessary, we can choose $y$ so that no 2-chord of $C$ with endpoints $p^z, q_1$ separates $y$ from $w$ (i.e, if such a 2-chord exists, we replace $y$ with the midpoint of that 2-chord). Let $H'$ be the subgraph of $H_z\cup K_1$ bounded by outer cycle $p^z(C\setminus\mathring{P})p_1q_1y$ and let $H''$ be the subgraph of $H_z\cup K_1$ bounded by outer cycle $wzp^zyq_1$. That is, $H'\cup H''=H_z\cup K_1$ and $H'\cap H''=p^zyq_1$. Now, the outer cycle of $H'$ contains the 3-path $R=p_zyq_1p_1$. Let $Y$ be the set of $L$-colorings $\psi$ of $\{p^z, p_1\}$ such that $\psi$ extends to at least two distinct elements of $\textnormal{End}(y, R, H')$. 

\vspace*{-8mm}
\begin{addmargin}[2em]{0em}
\begin{subclaim}\label{SetColSizeAtLeastP1} $|\textnormal{Col}(Y\mid p^z)|\geq |L(p_1)|$ \end{subclaim}

\begin{claimproof} If $p_1\not\in N(y)$, then $|L_{\psi}(y)|\geq 4$ for any $L$-coloring $\psi$ of $\{p^z, p_1\}$, so the subclaim immediately follows from 1) Theorem \ref{CornerColoringMainRes} in that case, since $1\leq |L(p_1)|\leq 3$. On the other hand, if $y\in N(p_1)$, then it follows from Theorem \ref{SumTo4For2PathColorEnds} that there is a family of $|L(p_1)|$ different elements of $\textnormal{End}(p^zyp_1, H'-q_1)$, each using a different color on $p^z$, and each of these  is an $L$-coloring of $\{p^z, p_1\}$ which also lies in $Y$, so we are again done. \end{claimproof}\end{addmargin}

Applying Corollary \ref{GlueAugFromKHCor}, it immediately follows that there is an $L$-coloring $\sigma$ of $\{p_0, p^z\}$ and a $\psi\in Y$ such that $\sigma$ is $(z, G_z)$-sufficient and $\sigma(p^z)=\psi(p^z)=b$ for some color $b$. We now fix $\phi^0, \phi^1\in\mathcal{F}$, where each of $\phi^0, \phi^1$ uses $\sigma(p_0), \psi(p_1)$ on the respective vertices $p_0, p_1$, and $\mathcal{C}(\phi^0)$ and $\mathcal{C}(\phi^1)$ are disjoint singletons. We let $S$ be the set of $s\in L_{\psi}(y)$ such that $\psi$ extends to an element of $\textnormal{End}(y, R, H')$ using $s$ on $y$. By definition, $|S|\geq 2$. Furthermore, since $G$ has no copies of $K_{2,3}$ and $yp^z\in E(G)$, $G_z$ is not just an edge, ie $p_0zp^z$ is a 2-path. Note that, for each $k=0,1$, $\phi^k\cup\psi$ is a proper $L$-coloring of its domain, since $q_0, q_1\not\in N(p^z)$.

\vspace*{-8mm}
\begin{addmargin}[2em]{0em}
\begin{subclaim}\label{EdgeWyWFourNeighbors} $wy\in E(G)$, and, in particular, $N(w)=\{q_0, z, w, q_1\}$.\end{subclaim}

\begin{claimproof} Suppose that $wy\not\in E(G)$. We now choose an arbitrary $\phi\in\{\phi^0, \phi^1\}$. Since neither $q_1$ nor $w$ is adjacent to $p^z$, and $q_1\not\in N(z)$, the outer cycle of $H''$ is an induced 5-cycle. Since $|L_{\phi}(w)|\geq 3$, there is an $r\in L_{\phi}(w)\setminus\mathcal{C}_G^P(\phi)$. Furthermore, as $|S\setminus\{\phi(q_1)\}|\geq 1$, at least one color of $S$ is left for $y$ in $L_{\phi\cup\psi}(y)$. By our choice of colors for $p_0, p^z$, it follows that the union $\phi\cup\psi$ extends to an $L$-coloring $\phi^*$ of $V(G_z\cup P)$, where $\phi^*(w)=r$ and $\phi^*(y)\in S$, so $\phi^*$ also extends to $L$-color $H'$. Since $r\not\in\mathcal{C}_G^P(\phi)$ and the outer cycle of $H''$ is induced, there is an $L$-coloring of the outer 5-cycle of $H''$ which does not extend to $L$-color $\textnormal{Int}_G(H'')$. By Theorem \ref{BohmePaper5CycleCorList}, there is a vertex of $\textnormal{Int}_G(H'')$ adjacent to all five vertices of $C^{H''}$, contradicting our choice of 2-chord $p^zyq_1$ of $C$. We concude that $wy\in E(G)$, as desired. Since $G$ has no induced 4-cycles and $z\not\in N(q_1)$, $y$ is adjacent to all four vertices of $\{p^z, z, w, q_1\}$, and, in particular, $H''$ is outerplanar and $N(w)=\{q_0, z, y, q_1\}$. \end{claimproof}\end{addmargin}

\vspace*{-8mm}
\begin{addmargin}[2em]{0em}
\begin{subclaim}\label{listOfFactsACSubFirst} All of the following hold.
\begin{enumerate}[label=\Alph*)]
\itemsep-0.1em
\item For each $k\in\{0,1\}$, any $d\in L_{\phi^k}(q_0)\setminus\{b\}$ and $s\in S\setminus\{d, \phi^k(q_1)\}$, $\phi^k$ extends to an $L$-coloring of $G-w$ using $d, s$ on the respective vertices $z,y$, and $L(w)\setminus\{\phi^k(q_0), d ,s, \phi^k(q_1)\}|=1$; AND
\item For each $k\in\{0,1\}$, $\phi^k(q_1)\not\in L_{\phi^k}(z)\setminus\{b\}$, and $S\cap (L_{\phi^k}(z)\setminus\{b\})=S\setminus\{\phi^k(q_1)\}$; AND
\item For some $k\in\{0,1\}$, $\phi^k(q_1)\not\in S$.
\end{enumerate}
 \end{subclaim}

\begin{claimproof} We first fix a $k\in\{0,1\}$ and show that $k$ satisfies A) and B). Since $\phi^k\cup\psi$ restricts to an element of $\textnormal{End}(p_0zp^z, G_z)$ and $s\in S$, it immediately follows that $\phi^k$ extends to an $L$-coloring of $G-w$ using $d, s$ on the respective vertices $z,y$, and since $|\mathcal{C}(\phi^k)|=1$, we have $L(w)\setminus\{\phi^k(q_0), d ,s, \phi^k(q_1)\}|=1$. This proves A). Now we show that $k$ satisfies B). We first show that $\phi^k(q_1)\not\in L_{\phi^k}(z)\setminus\{b\}$. Suppose not. Since $|S\setminus\{\phi^k(q_1)\}|\geq 1$, it follows from A) that $\phi^k\cup\psi$ extends to an $L$-coloring $\phi^*$ of $V(G-w)$ using $\phi^k(q_1)$ on each of $z$ and $q_1$. But then, since $\textnormal{deg}(w)|=4$, there are two colors left for $w$, contradicting the second part of A). Thus, $\phi^k(q_1)\not\in L_{\phi^k}(z)\setminus\{b\}$. Now let $s\in S\setminus\{\phi^k(q_1)\}$. To finish the proof of B), it suffices to show that $s\in L_{\phi^k}(z)\setminus\{b\}$. Suppose not. Then $\phi^k\cup\psi$ extends to an $L$-coloring $\psi^*$ of $V(H')\cup V(P-w)$ using $s$ on $y$. Since $s\not\in L_{\phi^k}(z)\setminus\{b\}$ any $\psi^*$ restricts to an element of $\textnormal{End}(p_0zp^z, G_z)$, it follows that $\mathcal{C}(\psi^*)=L_{\psi^*}(w)$, as $z$ has precisely five neighbors in $\textnormal{dom}(\psi^*)$. Thus, $|\mathcal{C}(\phi^k)|\geq |\mathcal{C}(\psi^*)\geq 2$, which is false. 

Now we prove C). Suppose that C) does not hold. Thus, $\{\phi^0(q_1), \phi^1(q_1)\}\subseteq S$. By B), $\phi^0(q_1)=\phi^1(q_1)=s$ for some $s\in S$. As $\phi^0$ and $\phi^1$ are distinct extensions of $\phi$ to $V(P-w)$, we have $\phi^0(q_0)\neq\phi^1(q_0)$. If there exists a $k\in\{0,1\}$ such that $\phi^k(q_0)\not\in L(z)\setminus\{\phi(p_0), b, s\}$, then we leave $z$ uncolored and extend $\phi^k\cup\psi$ to an $L$-coloring $\psi^*$ of $V(H')\cup V(P-w)$ which uses $s$ on $y$. Since $\psi^*$ restricts to an element of $\textnormal{End}(p_0zp^z, G_z)$, we get $\mathcal{C}(\psi^*)=L_{\psi^*}(w)$, and thus $|\mathcal{C}(\phi^k)|\geq |\mathcal{C}(\psi^*)|\geq 2$, which is false.  Thus, we have $\{\phi^0(q_0), \phi^0(q_1)\}\subseteq L(z)\setminus\{\phi(p_0), b, s\}$. But now we just choose an $r\in L(w)\setminus (S\cup\{\phi^0(q_0), \phi^1(q_1)\}$. Applying A), for each $k=0,1$, we extend $\phi^k$ to an $L$-coloring of $G-w$ by coloring $z, y$ with the respective colors $\phi^{1-k}(q_0)$ and $s$, which leaves $r$ for $w$, so $r\in\mathcal{C}_G^P(\phi^0)\cap\mathcal{C}_G^P(\phi^1)$, contradicting our assumption that these sets are disjoint. \end{claimproof}\end{addmargin}

Applying Facts A)-C), we have the following.

\vspace*{-8mm}
\begin{addmargin}[2em]{0em}
\begin{subclaim}\label{Replace1ForSubFactList1} $\phi^0(q_0)=\phi^1(q_1)=d$ for some color $d$, and furthermore, $|S_{\psi}|=2$ and $L(z)=\{\sigma(p_0), b, d\}\cup S$ as a disjoint union. \end{subclaim}

\begin{claimproof} Applying C) of Subclaim \ref{listOfFactsACSubFirst}, we let $k\in\{0,1\}$ with $\phi^k(q_1)\not\in S$. By B), $L_{\phi^k}(z)\setminus\{b\}=S$.  Applying B) again, we get $\phi^{1-k}(q_1)\not\in S$ as well, and $L_{\phi^{1-k}}(z)\setminus\{b\}=S$. In particular, since $\phi^0\cup\psi$ and $\phi^1\cup\psi$ restrict to the same $L$-coloring of $\{p^0, z\}$, we have $\phi^0(q_0)=\phi^1(q_0)=d$ for some color $d$. Since $\phi^0$ and $\phi^1$ are distinct, $\phi^0(q_1)\neq\phi^1(q_1)$. If either $|L(z)|>5$ or $L(z)\neq S\cup\{\phi(p_0), b, d\}$, then, for some $s\in  S$ and any $k\in\{0,1\}$, $\phi^k\cup\psi$ extends to an $L$-coloring $\psi^*$ of $V(P-w)\cup V(H')$ using $s$ on $y$, where $|L_{\psi^*}(z)|\geq 2$. Since $\psi^*$ restricts to an element of $\textnormal{End}(p_0zp^z, G_z)$ and two colors are left for $z$, we get $\mathcal{C}(\psi^*)=L_{\psi^*}(w)$, so $|\mathcal{C}(\phi^k)|\geq 2$, which is false. Thus, $|L(z)|=5$ and $L(z)=S\cup\{a, b, d\}$ as a disjoint union. In particular, $|S|=2$. \end{claimproof}\end{addmargin}

Let $d$ be as in Subclaim \ref{Replace1ForSubFactList1}. 

\vspace*{-8mm}
\begin{addmargin}[2em]{0em}
\begin{subclaim}\label{Replace2ForSubFactList2}  $\{\phi^0(q_1), \phi^1(q_1)\}\cap S=\varnothing$, and, for each $k\in\{0,1\}$, any extension of $\phi^k$ to an $L$-coloring of $G-w$ colors the edge $zy$ with the colors of $S_{\psi}$. \end{subclaim}

\begin{claimproof} By C) of Subclaim \ref{listOfFactsACSubFirst}, $\{\phi^0(q_1), \phi^1(q_1)\}\cap S_{\psi}=\varnothing$. It follows from A) that, for each $k\in\{0,1\}$, $\phi^k\cup\psi$ extends to an $L$-coloring of $G-w$ which colors the edge $zy$ with the colors of $S$, and $\mathcal{C}(\phi^k)=L(w)\setminus (\{d, \phi^k(q_1)\}\cup S)$. Since $\mathcal{C}(\phi^0)$ and $\mathcal{C}(\phi^1)$ are disjoint singletons, we get $\mathcal{C}(\phi^k)=\{\phi^{1-k}(q_1)\}$ for each $k=0,1$, and it immediately follows that, for each $k\in\{0,1\}$, any extension of $\phi^k$ to an $L$-coloring of $G-w$ uses the colors of $S$ on $zy$. \end{claimproof}\end{addmargin}

By Claim \ref{EverySigmaNo2ChordFromW}, $yu_1^{\star}\in E(G)$ as well. Let $J$ be the subgraph of $G$ bounded by outer cycle $(p^z(C\setminus\mathring{P})u_1^{\star})y$. Note that $C^J$ contains the 2-path $p^zyu_1^{\star}$, and $G\setminus\{w, p_0\}=(G_z\cup J\cup K_1)+zyq_1$. Let $d$ be as above, i.e $d=\phi^0(q_0)=\phi^1(q_0)$. For each $k\in\{0,1\}$, since $u_1^{\star}$ is not adajcent to either of $p_0, q_0$, we let $\phi^k_*$ be an extension of $\phi^k$ to an $L$-coloring of $V(P-w)\cup V(K_1)$. By Subclaim \ref{Replace1ForSubFactList1}, $b\in L(z)\setminus\{\phi(p_0), z\}$. We now fix a $b'\in\Lambda(\phi(p_0), b, \bullet)$. 

\vspace*{-8mm}
\begin{addmargin}[2em]{0em}
\begin{subclaim}\label{ListofFactsDToESubCL2} All of the following hold. 
\begin{enumerate}
\itemsep-0.1em
\item [\mylabel{}{D)}] For each $k\in\{0,1\}$, $\Lambda_J(b', \bullet, \phi^k_*(u_1^{\star}))\subseteq\{b, \phi^k(q_1)\}$; AND
\item [\mylabel{}{E)}] $J$ is not a triangle and, for each $k\in\{0,1\}$, $\Lambda_J(b, \bullet, \phi^k_*(u_1^{\star}))\subseteq S$; AND
\item [\mylabel{}{F)}] $J$ is a broken wheel with principal path $p^zyu_1^{\star}$, and $\phi^0_*(u_1^{\star})=\phi^1_*(u_1^{\star})$.
\end{enumerate}
\end{subclaim}

\begin{claimproof} Let $k\in\{0,1\}$. If $\Lambda(b', \bullet, \phi^k_*(u_1^{\star}))\not\subseteq\{b, \phi^k(q_1)\}$, then $\phi^*_k$ extends to an $L$-coloring of $G-w$ which colors $z$ with $b$, contradicting Subclaim \ref{Replace2ForSubFactList2} . This proves D). Now suppose toward a contradiction that $J$ is a triangle. Thus, for each $k\in\{0,1\}$, we have $\Lambda(b', \bullet, \phi^k_*(u_1^{\star}))=\varnothing$, or else we contradict D), so $b'=\phi^0_*(u_1^{\star})=\phi^1_*(u_1^{\star})$. Now we color $p^z$ with $b$ instead. Since $b\neq b'$, we get $\Lambda_J(b, \bullet, \phi^k_*(u_1^{\star}))=L(y)\setminus\{b, b'\}$ for each $k=0,1$. Since $\phi^0(q_1)\neq\phi^1(q_1)$, there exists an $\ell\in\{0,1\}$ such that $L(y)\setminus\{b, b'\phi^{\ell}(q_1)\}\neq S$. On the other hand, since $\sigma\in\textnormal{End}(p_0zp^z, G_z)$, it follows from Subclaim \ref{Replace1ForSubFactList1} that $S\subseteq\Lambda_{G_z}(\phi(p_0), \bullet, b)$, so $\phi^{\ell}\cup\psi$ extends to an $L$-coloring of $G-w$ in which $y$ is not colored by a color of $S$, contradicting Subclaim \ref{Replace2ForSubFactList2}. Thus, $J$ is not a triangle. Furthermore, given a $k\in\{0,1\}$, if $\Lambda_J(b, \bullet, \phi^k_*(u_1^{\star}))\setminus S$ is nonempty, then $\phi^k\cup\psi$ extends to an $L$-coloring of $G-w$ using a color of $S$ on $z$ and using a color on $y$ which does not lie in $S$, contradicting Subclaim \ref{Replace2ForSubFactList2}. This proves E).

It follows from E) that there are at least two $L$-colorings of $\{p^z, y, u_1^{\star}\}$ which do not extend to $L$-color $J$. By 1) of Theorem \ref{EitherBWheelOrAtMostOneColThm}, $J$ is a broken wheel with principal path $p^zyu_1^{\star}$, where $|V(J)|\geq 4$. Let $u$ be the unique neighbor of $p^z$ on the path $G^z-y$. Now suppose toward a contradiction that $\phi^0_*(u_1^{\star})\neq\phi^1_*(u_1^{\star})$. Thus, there exist colors $r_0, r_1\in L(y)$, where $r_0\neq r_1$ and, for each $k\in\{0,1\}$, $r_k\in L(y)\setminus (S\cup\{b, \phi^k_*(u_1^{\star})\})$. By E), $L(u)=\{b, r_0, r_1\}$, so $L(u)\cap S=\varnothing$, as $b\not\in S$. Furthermore, since $\phi^0_*(u_1^{\star})\neq\phi^1_*(u_1^{\star})$ and $|S|=2$, there exists a $\ell\in\{0,1\}$ and an $s\in S$ with such that $s\not\in\{b', \phi^{\ell}_*(q_1)\}$. But since $L(u)\cap S=\varnothing$, we have $s\in\Lambda_J(b', \bullet, \phi^{\ell}_*(u_1^{\star}))$, and since $S\cap\{b, \phi^0(q_1), \phi^1(q_1)\}=\varnothing$, we contradict D). This proves F). \end{claimproof}\end{addmargin}

Applying F) of Subclaim \ref{ListofFactsDToESubCL2}, we let $\phi^0_*(u_1^{\star})=\phi^1_*(u_1^{\star})=s$ for some color $s$.  Since $\phi^0(q_1)\neq\phi^1(q_1)$, it follows from D) that $\Lambda_J(b', \bullet, s)\subseteq\{b\}$. Let $T=L(y)\setminus\{b, b', s\}$. Now we have $T\cap\Lambda_J(b', \bullet, s)=\varnothing$. It follows from 3i) of Corollary \ref{CorMainEitherBWheelAtM1ColCor} that $|T|=2$, so $\{b, b', s\}$ is a subset of $L(y)$ of size three, and $b\neq s$. By E), $\Lambda_J(b, \bullet, s)$ is a proper subset of of $L(y)\setminus\{b, s\}$, and since $T\cap\Lambda_J(b', \bullet, s)=\varnothing$ and $b'\neq b$, it  follows from 3ii) of Corollary \ref{CorMainEitherBWheelAtM1ColCor} that $b=s$ and $b\in T$, which is false. This completes the proof of Claim \ref{PzYObstructionVert}.   \end{claimproof}

\subsection{Dealing with 2-chords of $C$ of the form $p_izw$: part II}\label{2ChordIntermSecondPartSec}

\begin{Claim}\label{TwoColoringsABDiffColQ1-I} Let $i\in\{0,1\}$, where $p_i, q_i, w$ have a common neighbor $z$ in $G\setminus C$. Suppose further that $z$ and $q_{1-i}$ have a common neighbor $y$ with $y\neq w$. Then, for each $(a,b)\in L(p_0)\times L(p_1)$, we have $\phi_{ab}^0(q_{1-i})\neq \phi_{ab}^1(q_{1-i})$. \end{Claim}

\begin{claimproof}  Say $i=0$ without loss of generality, so $K_0$ is an edge, and, by Claim \ref{EverySigmaNo2ChordFromW}, $y\not\in V(C)$ and $N(w)=\{q_0, z, y, q_1\}$, and furthermore, $yu_1^{\star}\in E(G)$ and $K_1$ is not an edge. Let $x_y$ be the unique vertex of $N(y)\cap V(C\setminus\mathring{P})$ which is closest to $p_0$ on the path $C\setminus\mathring{P}$. Possibly $x_y=u_1^{\star}$. Since $G$ is $K_{2,3}$-free, $x_y\neq p_0$. Let $J$ be the subgraph of $H_1\setminus\mathring{P}$ bounded by outer cycle $p_0(C\setminus\mathring{P})x_yyz$ and let $J'$ be the subgraph of $H_1\setminus\mathring{P}$ bounded by outer face $x_y(C\setminus\mathring{P})u_1^{\star}y$. Possibly $J'$ is an edge. In any case, $J\cup J'=H_1\setminus\mathring{P}$ and $J\cap J'=yx_y$. The outer cycle of $J$ contains the 3-path $p_0zyx_y$, and every chord of the outer cycle of $J$ is incident to $z$. This is illustrated in Figure \ref{RPathAlmostLastCase}, where the path $R$ is in bold. We note that $z\not\in N(x_y)$, or else $x_y=p^z$, contradicting Claim \ref{PzYObstructionVert}.  Since $G$ has no induced 4-cycles, it follows from the definition of $x_y$ and $p^z$ that $p^zx_y$ is not an edge of $G$, so $d(p^z, x_y)\geq 2$.

\begin{center}\begin{tikzpicture}

\node[shape=circle,draw=black] (p0) at (-3.5, 0) {$p_0$};
\node[shape=circle,draw=black] (u1star) at (1.5, 0) {$u_1^{\star}$};
\node[shape=circle,draw=white] (p1-) at (2.5, 0) {$\ldots$};
\node[shape=circle,draw=black] (p1) at (3.5, 0) {$p_1$};
\node[shape=circle,draw=black] (q0) at (-2.5,2) {$q_0$};
\node[shape=circle,draw=black] (q1) at (1.5,2) {$q_1$};
\node[shape=circle,draw=black] (w) at (0,3.5) {$w$};
\node[shape=circle,draw=black] (z) at (-1.4, 2) {$z$};
\node[shape=circle,draw=black] (xy) at (-0.5, 0) {$x_y$};
\node[shape=circle,draw=white] (xy+) at (0.5, 0) {$\ldots$};
\node[shape=circle,draw=white] (p0+) at (-2, 0) {$\ldots$};
\node[shape=circle,draw=white] (J) at (-1.2, 1) {$J$};
\node[shape=circle,draw=white] (J') at (0.5, 1) {$J'$};
\node[shape=circle,draw=white] (K1) at (2, 0.8) {$K_1$};
\node[shape=circle,draw=black] (y) at (0.5, 2) {$y$};
\draw[-] (p0) to (q0) to (w) to (q1);
\draw[-]  (p0) to (p0+) to (xy) to (xy+) to (u1star) to (p1-) to (p1);
\draw[-] (w) to (z);
\draw[-] (y) to (q1) to (p1);
\draw[-] (q0) to (z);
\draw[-] (y) to (w);
\draw[-] (q1) to (u1star) to (y);
\draw[-, line width=1.9pt] (p0) to (z) to (y) to (xy);

\end{tikzpicture}\captionof{figure}{}\label{RPathAlmostLastCase}\end{center}

Now suppose toward a contradiction that Claim \ref{TwoColoringsABDiffColQ1-I} does not hold. Thus, there is an $(a,b)\in L(p_0)\times L(p_1)$ such that $\phi_{ab}^0(q_1)=\phi_{ab}^1(q_1)=c$ for some color $c$. Furthermore, for each $k\in\{0,1\}$, $\phi_{ab}^k$ extends to an $L$-coloring $\psi^k$ of $V(P-w)\cup V(K_1)$, where $\psi^0, \psi^1$ use the same color on $u_1^{\star}$, i.e a color of $\Lambda_{K_1}(\bullet, c, b)$. Let $\psi^0(q_1)=\psi^1(q_1)=c'$ for some $c'\in L(u_1^{\star})$. We now let $A=\{\psi^0(q_0), \psi^1(q_0)\}$. Since $\phi_{ab}^0$ and $\phi_{ab}^1$ are distinct, $|A|=2$.

\vspace*{-8mm}
\begin{addmargin}[2em]{0em}
\begin{subclaim}\label{yp0NoCommExZSub} The three vertices $y, p_0, x_y$ have no common neighbor. \end{subclaim}

\begin{claimproof} Suppose toward a contradiction that $y, p_0, x_y$ have a common neighbor $y'$. Since $yp_0\not\in E(G)$ and $G$ has no induced 4-cycles, $zy'\in E(G)$ as well, and, since $d(p^z, x_y)\geq 2$, $y'\not\in V(C)$. In particular, $G$ contains the 5-cycle $p_0q_0wyy'$, and $\textnormal{Int}(F)$ is a wheel with central vertex $z$ adjacent to all the vertices of $F$, so $N(z)=V(F)$ and $p^z=p_0$. Furthermore, since $y'in N(p_0)\cap (x_y)$, we have $x_y\neq u_1^{\star}$, or else we contradict Claim \ref{NoVertexVReachesAcrossFromPiToU1-iStar}, so $J'$ is not just an edge. Since $|L(y)\setminus\{c, c'\}|\geq 3$, it follows from 2ii) of Corollary \ref{CorMainEitherBWheelAtM1ColCor} that there is an $s\in L(y)\setminus\{c, c'\}$ with $\Lambda_{J'}^{x_yyu_1^{\star}}(\bullet, s, c')|\geq 2$. Now, $J-z$ is bounded by outer cycle $p_0y'yx_y(C\setminus\mathring{P})p_0)$. Let $L^*$ be a list-assignment for $V(J-z)$ in which $p_0, y$ are precolored with the respective vertices $a, s$, where $L^*(x_y)=\{s\}\cup\Lambda_{J'}^{x_yyu_1^{\star}}(\bullet, s, c')$ and otherwise $L^*=L$. Since $L^*(y')|=5$ and $|L^*(x_y)|\geq 3$, it follows from Lemma \ref{PartialPathColoringExtCL0} that $J-z$ is $L^*$-colorable, so it follows from our choice of $s$ that, for each $k\in\{0,1\}$, $\psi^k$ extends to an $L$-coloring $\pi^k$ of $G\setminus\{w, z\}$ such that $\pi^0(y)=\pi^1(y)=s$ and $\pi^0(y')=\pi^1(y')=s'$ for some color $s'$. If there is a $k\in\{0,1\}$ such that $\pi^k(q_0)\not\in L(z)\setminus\{a, s', s\}$, then $|\mathcal{C}(\pi^k)|\geq |\mathcal{C}(\phi_{ab}^k)|\geq 2$, which is false. Thus, $A\subseteq L(z)\setminus\{a, s, s'\}$, and, since $|A|=2$, it follows that, for each $k\in\{0,1\}$, $\pi^k$ extends to an $L$-coloring of $G-w$ in which the endpoints of $q_0z$ are colored with the colors of $A$. Choosing an $r\in L(w)\setminus (A\cup\{s, c\})$, we then have $r\in\mathcal{C}(\pi^0)\cap\mathcal{C}(\pi^1)$, so $r\in\mathcal{C}(\phi_{ab}^0)\cap\mathcal{C}(\phi_{ab}^1)$, which false, as $\mathcal{C}(\phi_{ab}^0)$ and $\mathcal{C}(\phi_{ab}^1)$ are disjoint. \end{claimproof}\end{addmargin}

\vspace*{-8mm}
\begin{addmargin}[2em]{0em}
\begin{subclaim}\label{ColorThreeVerticesAllofJUnionJ'} Any $L$-coloring of $\{p_0, y, u_1^{\star}\}$ extends to $L$-color all of $H_1\setminus\mathring{P}$. \end{subclaim}

\begin{claimproof} Let $\sigma$ be an $L$-coloring of $\{p_0, y, u_1^{\star}\}$. Since $d(p_0, u_1^{\star})\geq d(p^z, u_1^{\star})\geq 2$, $\sigma$ extends to an $L$-coloring $\sigma'$ of $V(J')\cup\{p_0\}$. We just need to check that $\sigma'$ extends to $L$-color $J$ as well. Suppose not. Since $z\not\in N(x_y)$, we have $|L_{\sigma'}(z)|\geq 3$, and since every chord of the outer cycle of $J$ is incident to $z$, it follows from Lemma \ref{PartialPathColoringExtCL0} that there is a vertex of $J\setminus C^J$ adjacent to all three of $p_0, x_y, y$, contradicting Subclaim \ref{yp0NoCommExZSub}. \end{claimproof}\end{addmargin}

Applying Subclaim \ref{ColorThreeVerticesAllofJUnionJ'}, we have the following.

\vspace*{-8mm}
\begin{addmargin}[2em]{0em}
\begin{subclaim}\label{AL(y)CC'Disjoint} $A\cap (L(y)\setminus\{c, c'\})=\varnothing$ \end{subclaim}

\begin{claimproof} Suppose not. Thus, there is a $k\in\{0,1\}$ with $psi^k(q_0)\in L(y)\setminus\{c, c'\}$. By Subclaim \ref{ColorThreeVerticesAllofJUnionJ'}, there is an $L$-coloring $\tau$ of $J\cup J'$ using the colors $a, \psi^k(q_0), c'$ on the respective vertices $p_0, y, u_1^{\star}$. Note that $\tau(z)\neq\psi^k(q_0)$, so $\psi^k\cup\tau$ is a proper $L$-coloring of $G-w$. Since two neighbors of $w$ are using the same color, we have $|\mathcal{C}(\phi^k_{ab})|\geq\mathcal{C}(\psi^k\cup\tau)|\geq 2$, which is false. \end{claimproof}\end{addmargin}

It follows from 2) of Claim \ref{ForEachPiWeExtToK0K1ToG} that $|L(w)|=5$ and $A\cup\{c\}$ is a subset of $L(w)$ of size three. It follows from Subclaim \ref{AL(y)CC'Disjoint} that $L(y)\setminus\{c, c'\}$ is disjoint to $A\cup\{c\}$, and since $|L(y)\setminus\{c, c'\}|\geq 3$, there is an $f\in L(y)\setminus\{c, c'\}$ with $f\not\in L(w)$. Possibly $f=a$, but since $yp_0\not\in E(G)$, it follows from Subclaim \ref{ColorThreeVerticesAllofJUnionJ'} that there is an $L$-coloring $\tau$ of $J\cup J'$ using $a, f, c'$ on the respective vertices $p_0, y, u_1^{\star}$. Since $|A|=2$, there is a $k\in\{0,1\}$ such that $\psi^k(q_0)\neq\tau(z)$, so $\psi^k\cup\tau$ is a proper $L$-coloring of $G-w$. Since $f\not\in L(w)$, we have $|\mathcal{C}(\phi_{ab}^k)|\geq\mathcal{C}(\psi^k\cup\tau)|\geq 2$, which is false. \end{claimproof}

Applying Claim \ref{TwoColoringsABDiffColQ1-I} and recalling Definition \ref{DefnPPUnObstruc}, we now prove the main result of Subsection \ref{2ChordIntermSecondPartSec}

\begin{Claim}\label{H1PP0UnobstructCL} For each $i\in\{0,1\}$, if $K_i$ is an edge, then $H_{1-i}$ is neither $(Q, p_i)$-obstructed nor $(Q, u_{1-i}^{\star})$-obstructed. \end{Claim}

\begin{claimproof}  Say $i=0$ for the sake of definiteness, and suppose toward a contradiction that $K_0$ is an edge, but there is an endpoint $p^*$ of $P_1$ such that $H_1$ is $(P_1, p^*)$-obstructed. By Claim \ref{EvenNumberObstructionsHiCL}, $H_1$ is both $(P_1, p_0)$-obstructed and $(P_1, u_1^{\star})$-obstructed. Since $N(p_0)\cap N(u_1^{\star})\subseteq V(C)$ and $N(q_0)\cap N(q_1)=\{w\}$, it follows that there exists an edge $zy$ of $V(H_1\setminus C^{H_1})$, where $N(z)\cap V(P_1)=\{p_0, q_0, w\}$ and $N(y)\cap V(P_1)=\{w, q_1, u_1^{\star}\}$. By Claim \ref{EverySigmaNo2ChordFromW}, $K_1$ is not an edge, so $u_1^{\star}\neq p_1$. Note that $H_1\setminus\mathring{P}$ is bounded by outer cycle $p_0(C\setminus\mathring{P})u_1^{\star}$, and this cycle contains the 3-path $R=p_0zyu_1^{\star}$. Since $yp_0\not\in E(G)$ it follows that, for any $L$-coloring $\psi$ of $\{p_0, u_1^{\star}$, $|L_{\psi}(y)|\geq 4$. Let $Y$ be the set of $L$-colorings $\psi$ of $\{p_0, u_1^{\star}\}$ such that $\psi$ extends to at least two different elements of $\textnormal{End}(y, R, H_1\setminus\mathring{P})$. By 1) of Theorem \ref{CornerColoringMainRes}, since $1\leq |L(p_0)|\leq 3$, there is a family of $|L(p_0)|$ different elements of $Y$, each using a different color on $u_1^{\star}$. By Corollary \ref{GlueAugFromKHCor}, there is a $\psi\in Y$ and and a $\psi'\in\textnormal{End}(u_1^{\star}q_1p_1, K_1)$ with $\psi(u_1^{\star})=\psi'(u_1^{\star})$. Let $S$ be a subset of $L(y)\setminus\{\psi(u_1^{\star})$ of size two, where, for each $s\in S$, $\psi$ extends to an element of $\textnormal{End}(y, R, H_1\setminus\mathring{P})$ using $s$ on $y$. 

We now let $a=\psi(p_0)$ and $b=\psi'(p_1)$. By Claim \ref{TwoColoringsABDiffColQ1-I}, there exists a $k\in\{0,1\}$ such that $\phi^k_{ab}(q_1)\neq\psi'(u_1^{\star})$, so $\phi^k_{ab}\cup\psi\cup\psi'$ is a proper $L$-coloring of $V(P-w)\cup\{u_1^{\star}\}$, and, since $\psi'$ is $(q_1, K_1)$-sufficient, it follows that $\phi^k_{ab}\cup\psi\cup\psi'$ extends to an $L$-coloring $\sigma$ of $V(P-w)\cup V(K_1)$. Since $|S\setminus\{\sigma(y)\}|\geq 1$, $\sigma$ extends to an $L$-coloring $\sigma'$ of $\textnormal{dom}(\sigma)\cup\{y\}$ with $\sigma'(y)\in S$. Note that $z\not\in N(u_1^{\star})$, as $G$ contains no copies of $K_{2,3}$, so $z$ only has three neighbors in $\textnormal{dom}(\sigma')$ and $|L_{\sigma'}(z)|\geq 2$. Since there are two colors left over for $z$, it follows from our choice of $\psi$ that any extension of $\sigma'$ to an $L$-coloring of $\textnormal{dom}(\sigma')\cup\{w\}$ extends to $L$-color all of $G$, so $\mathcal{C}(\sigma')=L_{\sigma'}(w)$ and thus $\mathcal{C}(\phi_{ab}^k)|\geq |\mathcal{C}(\sigma')|\geq 2$, which is false. \end{claimproof}

\subsection{Completing the proof of Lemma \ref{EndLinked4PathBoxLemmaState}}\label{FinalSubSecLemmaBox}

By Claim \ref{AtMostOneChordToC-P}, at least one of $K_0, K_1$ is an edge, so suppose without loss of generality that $K_0$ is an edge. Thus, the outer cycle of $H_1$ is induced. By Claim \ref{H1PP0UnobstructCL}, for each endpoint $p^*$ of $Q_1$, $H_1$ is $(Q_1, p^*)$-unobstructed.

\begin{Claim}\label{NothingMapsTwoColorsifCounter} Either $K_1$ is an edge, or, for each $\phi\in\mathcal{F}$, $\Lambda_{K_1}(\bullet, \phi(q_1), \phi(p_1))|=1$. \end{Claim}

\begin{claimproof} Suppose that $K_1$ is a not an edge and let $phi\in\mathcal{F}$ Let $T=\Lambda_{K_1}(\bullet, \phi(q_1), \phi(p_1)$ and suppose toward a contradiction that $|T|\geq 2$. Note that $T\subseteq L(u_1^{\star})\setminus\{\phi'(q_1)\}$. Since $|\mathcal{C}(\phi)|=1$, there exist distinct $r_0, r_1\in L_{\phi}(w)\setminus\mathcal{C}(\phi)$. For each $k=0,1$, the $L$-coloring $\sigma_k=(\phi(p_0), \phi(q_0), r_k, \phi(q_1))$ of $p_0q_0wq_1$ does not extend to an $L$-coloring of $u_1^{\star}$ which uses a color of $T$ of $H_1$. Since the outer cycle of $H_1$ is induced and $H_1$ is $(Q, p_0)$-unobstructed, it follows from P1) of Proposition \ref{InterMedUnobstrucGInduced} that each of $\sigma_0$ and $\sigma_1$ is $(L, P_1, p_0)$-blocked, so $\{\phi(q_0), r_0\}$ and $\{\phi(q_0), r_1\}$ are the same set of size two, which is false. \end{claimproof}

\begin{Claim}\label{K1ListGeq2TriangleSubCLLp0Lp1} Either $K_1$ is an edge, or $K_1$ is a broken wheel with principal path $u_1^{\star}q_1p_1$. Furthermore, if $K_1$ is not an edge and $|L(p_1)|\geq 2$, then $K_1$ is a triangle. \end{Claim}

\begin{claimproof} Suppose that $K_1$ is not an edge. As $\mathcal{F}\neq\varnothing$, it follows from Claim \ref{NothingMapsTwoColorsifCounter}, together with 2i) of Corollary \ref{CorMainEitherBWheelAtM1ColCor}, that $K_1$ is a broken wheel with principal path $u_1^{\star}q_1p_1$ and no color of $L(p_1)$ is almost $K_1$-universal. The rest follows from Claim \ref{KiEitherUniversalColForTriangle}. \end{claimproof}

By 2) of Claim \ref{ForEachPiWeExtToK0K1ToG}, each element of $\mathcal{F}$ extends to an $L$-coloring of $V(P-w)\cup V(K_1)$, so, for each $\phi\in\mathcal{F}$, for each $(a,b)\in L(p_0)\times L(p_1)$ and $k\in\{0,1\}$, we define an $L$-coloring $\pi_{ab}^k$ of $Q_1-w$, where $\pi_{ab}^k$ and $\phi_{ab}^k$ are equal on their common domain, and, if $K_1$ is not an edge, then the $L$-coloring $(\pi_{ab}^k(u_1^{\star}), \phi_{ab}^k(q_1), b)$ of $u_1^{\star}q_1p_1$ extends to $L$-color $K_1$. Note that $\mathcal{C}_{H_1}^{Q_1}(\pi_{ab}^k)\subseteq\mathcal{C}(\phi_{ab}^k)$. In particular, for each $(a,b)\in L(p_0)\times L(p_1)$, $\mathcal{C}_{H_1}^{Q_1}(\pi_{ab}^0)$ and $\mathcal{C}_{H_1}^{Q_1}(\pi_{ab}^1)$ are disjoint sets of size at most one. We let $\mathcal{F}'=\{\pi_{ab}^k: (a,b)\in L(p_0)\times L(p_1)\ \textnormal{and}\ k\in\{0,1\}\}$. Note that, if $K_1$ is an edge, then $\mathcal{F}'=\mathcal{F}$. It follows from Claim \ref{NothingMapsTwoColorsifCounter} that the color $\pi_{ab}^k(u_1^{\star})$ is uniquely specified by $\phi_{ab}^k$, so, if $K_1$ is not an edge, then, for each $(a,b)\in L(p_0)\times L(p_1)$ such that $\Lambda_{K_1}(\bullet, \phi_{ab}^0(q_1), b)\cap\Lambda_{K_1}(\bullet, \phi_{ab}^1(q_1), b)\neq\varnothing$, $\pi_{ab}^0$ and $\pi_{ab}^1$ use the same color on $u_1^{\star}$. 

\begin{Claim}\label{TwoPhisDontUseSameColCL} For each $(a,b)\in L(p_0)\times L(p_1)$, $\phi^0_{ab}(q_1)\neq\phi^0(q_1)$. \end{Claim}

\begin{claimproof} As indicated above, for each endpoint $p^*$ of $Q_1$, $H_1$ is $(Q_1, p^*)$-unobstructed. Let $(a,b)\in L(p_0)\times L(p_1)$ and suppose that $\phi^0_{ab}(q_1)=\phi^0(q_1)$. As indicated above, we have $\pi^0_{ab}(u_1^{\star})=\pi^1_{ab}(u_1^{\star})$, and since $\phi_{ab}^0, \phi_{ab}^1$ are distinct and $\pi^0_{ab}(p_0)=\pi^1_{ab}(p_0)=a$, it follows that $\pi^0_{ab}$ and $\pi^1_{ab}$ differ precisley on $q_0$. Since $\mathcal{C}_{H_1}(\pi_{ab}^0)$ and $\mathcal{C}_{H_1}^{Q_1}(\pi_{ab}^1)$ are disjoint sets of size at most one, this contradicts P3) of Proposition \ref{InterMedUnobstrucGInduced} applied to the rainbow $(H_1, C^{H_1}, Q_1, L)$. \end{claimproof}

We now have the following.

\begin{Claim}\label{EitherZeroOrOneBlocked} There exists an $(a,b)\in L(p_0)\times L(p_1)$ such that, for each $\pi\in\{\pi^0_{ab}, \pi^1_{ab}\}$ and each $s\in L_{\pi}(w)\setminus\mathcal{C}_{H_1}^{Q_1}(\pi)$, the $L$-coloring $(\pi(p_0), \pi(q_0), s, \pi(q_1), \pi(u_1^{\star}))$ of $V(Q_1)$ is either $(L, Q_1, p_0)$-blocked or $(L, Q_1, u_1^{\star})$-blocked. \end{Claim}

\begin{claimproof} Suppose the claim does not hold. We first show that there is a chord of $C$. 

\vspace*{-8mm}
\begin{addmargin}[2em]{0em}
\begin{subclaim} $K_1$ is not an edge. \end{subclaim}

\begin{claimproof} Suppose that $K_1$ is an edge. Thus, $G$ is an induced cycle and $H_0=H_1=G$, and $u_1^{\star}=p_1$. Furthermore, for each $i\in\{0,1\}$, $G$ is $(P, p_i)$-unobstructed. By the minimality of $G$, $\Omega_G^{\geq 2}(P)=\varnothing$. For any $(a,b)\in L(p_0)\times L(p_1)$ and $k\in\{0,1\}$, $\pi^k_{ab}=\phi^k_{ab}$ and $\mathcal{C}_{H_1}^{Q_1}(\pi_{ab}^k)|=1$, so it follows from P2) of Proposition \ref{InterMedUnobstrucGInduced} that any $(a,b)\in L(p_0)\times L(p_1)$ satisfies Claim \ref{EitherZeroOrOneBlocked}, contradicting our assumption. \end{claimproof}\end{addmargin}

\vspace*{-8mm}
\begin{addmargin}[2em]{0em}
\begin{subclaim}\label{ForEachOmegaSufficientSubCL} For any $L$-coloring $\psi$ of $\{u_1^{\star}, p_1\}$, if $\psi$ is $(K_1, q_1)$-sufficient, then no element of $\Omega_{H_1}^{\geq 2}(Q_1)$ colors $u_1^{\star}$ with $\psi(u_1^{\star})$. \end{subclaim}

\begin{claimproof}  Suppose there is a $\psi'\in\Omega_{H_1}(Q_1)$ using the color $\psi(u_1^{\star})$ on $u_1^{\star}$. Thus, $\psi'\cup\psi$ is a proper $L$-coloring of $\{p_0, u_1^{\star}, p_1\}$. Let $a=\psi'(p_0)$ and $c=\psi(q_1)$, and consider the $L$-colorings $\phi_{ac}^0$ and $\phi_{ac}^1$ of $P-w$. If there is a $k\in\{0,1\}$ such that $\phi_{ac}^k\neq\psi(u_1^{\star})$, then, by our choice of $\psi$ the $L$-coloring $(\psi(u_1^{\star}), \phi_{ac}^k(q_1), c)$ of $u_1^{\star}q_1p_1$ extends to $L$-color $K_1$. If that holds, then $|\mathcal{C}(\phi^k_{ac})|\geq 2$, since $\psi'\in\Omega_{H_1}(Q_1)$. Since $|\mathcal{C}(\phi^k_{ac})|=1$, we have $\phi^0_{ac}(q_1)=\phi^1_{ac}(q_1)=\psi(u_1^{\star})$, contradicting Claim \ref{TwoPhisDontUseSameColCL}. \end{claimproof}\end{addmargin}

\vspace*{-8mm}
\begin{addmargin}[2em]{0em}
\begin{subclaim}\label{UniformColorUsedAA'} For any $a, a'\in L(p_0)$ and $b\in L(p_1)$, we have $\{\pi_{ab}^0(q_1), \pi_{ab}^1(q_1)\}=\{\pi_{a'b}(q_1), \pi_{a'b}(q_1)\}$, and this set has size two. \end{subclaim}

\begin{claimproof} Let $T=\{c\in L(q_1)\setminus\{b\}: |\Lambda_{K_1}(\bullet, c, b)|=1\}$. By 2 ii) of Corollary \ref{CorMainEitherBWheelAtM1ColCor}, $|T|\leq 2$. Combining Claims \ref{NothingMapsTwoColorsifCounter} and \ref{TwoPhisDontUseSameColCL}, we then obtain the subclaim. \end{claimproof}\end{addmargin}

For each $(a,b)\in L(p_0)\times L(p_1)$, the labels of the two elements of $\mathcal{F}$ using $a,b$ on the respective vertices $p_0, p_1$ are arbitrary, so, applying Subclaim \ref{UniformColorUsedAA'}, we suppose without loss of generality that, for any $b\in L(p_1)$ and $k\in\{0,1\}$, the set $\{\pi_{ab}^k(q_1): a\in L(p_0)\}$ has size one. Thus, it follows from Claim \ref{NothingMapsTwoColorsifCounter} that, for any $b\in L(p_1)$ and $k\in\{0,1\}$, $\{\pi_{ab}^k(q_1): a\in L(p_0)\}$ also has size one.

\vspace*{-8mm}
\begin{addmargin}[2em]{0em}
\begin{subclaim} $|L(p_1)|>1$. \end{subclaim}

\begin{claimproof} Suppose not. Thus, $|L(p_0)|=3$. Let $b$ be the lone color of $L(p_1)$. For each $k\in\{0,1\}$, there is a color $s_k\in L(q_1)$ such that, for each $a\in L(p_0)$, $\pi_{ab}^k(u_1^{\star})=s_k$. For each $k\in\{0,1\}$, let $L_k$ be a list-assignment for $H_1$, where $L_k(u_1^{\star})=\{s_k\}$ and otherwise $L_k=L$. For each $k=0,1$, $(H_1, C^{H_1}, Q_1, L_k)$ is an end-linked rainbow, as $|L_k(p_0)|=3$. Since Claim \ref{EitherZeroOrOneBlocked} does not hold, there exists a $k\in\{0,1\}$ such that the end-linked rainbow $(H_1, C^{H_1}, Q_1, L_k)$ satisfies ii) of Proposition \ref{InterMedUnobstrucGInduced}. Thus, for some $k\in\{0,1\}$, there is a $\psi\in\Omega_{H_1}(Q_1)$ using $s_k$ on $u_1^{\star}$. Let $a=\psi(p_0)$. Since $\Lambda_{K_1}(\bullet, \pi_{ab}^k(q_1), b)=\{s_k\}$, we get $|\mathcal{C}(\phi_{ab}^k)|\geq 2$, which is false. \end{claimproof}\end{addmargin}

Since $|L(p_1)|>1$ and $K_1$ is not an edge, it follows from Claim \ref{K1ListGeq2TriangleSubCLLp0Lp1} that $K_1$ is a triangle. Since $|L(u_1^{\star})|\geq 3$, $(H_1, C^{H_1}, Q_1, L)$ is an end-linked rainbow. Now, it follows from P2) of Proposition \ref{InterMedUnobstrucGInduced}, together with the assumption of Claim \ref{EitherZeroOrOneBlocked}, that there is a $\psi^-\in\Omega_{H_1}^{\geq 2}(Q_1)$. Since $|L(p_1)|>1$, we have $|L(p_1)\setminus\{\psi^-(u_1^{\star})\}|\geq 1$, and since $K_1$ is a triangle, there is a $(K_1, q_1)$-sufficient $L$-coloring of $\{u_1^{\star}, p_1\}$ using $\psi^-(u_1^{\star})$ on $u_1^{\star}$, contradicting Subclaim \ref{ForEachOmegaSufficientSubCL}. This proves Claim \ref{EitherZeroOrOneBlocked}. \end{claimproof}

Let $(a,b)$ be as in Claim \ref{EitherZeroOrOneBlocked} and let $\pi^0=\pi^0_{ab}$ and $\pi^1=\pi^1_{ab}$. For each $i\in\{0,1\}$, we now define a set $T_i$ with $|T_i|=2$ as follows: Letting $p^*\in\{p_0, u_1^{\star}\}$ be the unique neighbor of $q_i$ on the path $Q_i$ and $x$ be the unique neighbor of $p^*$ on the path $C\setminus P$, if there is a $z\in V(G\setminus C)$ adjacent to all three of $p^*, q_i, w$, then we take $T_i$ to be a two-element subset of $L(z)\setminus L(x)$, and there is no such $z$, then we take $T_i$ to be an arbitrary 2-element subset of $L(q_i)$. 

\begin{Claim}\label{ForeachPairKT0T1} There exists a $k\in\{0,1\}$ such that either $\pi^k(q_0)\not\in T_0$ or $\pi^k(q_1)\not\in T_1$. \end{Claim}

\begin{claimproof} Suppose that, for each $k\in\{0,1\}$, $\pi^k(q_0)\in T_0$ and $\pi^k(q_1)\in T_1$. As $|T_0\cup T_1|\leq 4$, we choose an $r\in L(w)\setminus (T_0\cup T_1)$. For each $k\in\{0,1\}$, $(a, \pi^k(q_0), r, \pi^k(q_1), \pi^k(u_1^{\star}))$ is a proper $L$-coloring of $V(Q_1)$ which is neither $(L, Q_1, p_0)$-blocked nor $(L, Q_1, u_1^{\star})$-blocked, as $r\not\in T_0\cup T_1$. Thus, by Claim \ref{EitherZeroOrOneBlocked}, $r$ lies in both $\mathcal{C}_{H_1}^{Q_1}(\pi^0)$ and $\mathcal{C}_{H_1}^{Q_1}(\pi^1)$, which is false, as these sets are disjoint. \end{claimproof}

Applying Claim \ref{ForeachPairKT0T1}, we suppose without loss of generality that there is a $k\in\{0,1\}$ such that $\pi^k(q_0)\not\in T_0$. (Note that, once we fix a $k\in\{0,1\}$, the case with $T_1$ instead is indeed symmetric, even though $K_1$ is possibly not an edge). Let $S=(L(w)\setminus\{\pi^k(q_0), \pi^k(q_1)\})\setminus\mathcal{C}_{H_1}^{Q_1}(\pi^k)$. Note that $|S|\geq 2$. For each $s\in S$, let $\pi_s^k$ be the $L$-coloring $(\pi^k(p_0), \pi^k(q_0), s, \pi^k(q_1), \pi^k(u_1^{\star})$ of $Q_1$. For each $s\in S$, $\pi_s^k$ does not extend to an $L$-coloring of $H_1$, and furthermore, $\pi_s^k$ is not $(L, Q_1, p_0)$-blocked, since $\pi^k(q_0)\not\in T_0$. By Claim \ref{EitherZeroOrOneBlocked}, for each $s\in S$, $\pi_s^k$ is $(L, Q, u_1^{\star})$-blocked. Letting $s, s'$ be distinct elements of $S$, it follows that $\{s, \pi^k(q_1)\}$ and $\{s', \pi^k(q_1)\}$ are the same list of size two, which is false. This completes the proof of Lemma \ref{EndLinked4PathBoxLemmaState}. \end{proof}


\begin{thebibliography}{25}
\thispagestyle{fancy}

\bibitem{lKthForGoBoHm6}
T. Böhme, B. Mohar and M. Stiebitz, Dirac's map-color theorem for choosability
   \emph{J. Graph Theory} Volume 32, Issue 4 (1999), 327-339

\bibitem{5ChooseCrossFarPap}
   Z. Dvo\v{r}\'ak, B. Lidick\'{y} and B. Mohar, 5-choosability of graphs with crossings far apart,
  \emph{J. Combin. Theory Ser. B} 123 (2017), 54-96

\bibitem{JNHolepunchPaperICitation}
J. Nevin, Extensions of Thomassen's Theorem to Paths of Length At Most Four: Part I, \emph{arXiv: 2207.12128}

\bibitem{AllPlanar5ThomPap}
  C. Thomassen, Every Planar Graph is 5-Choosable
  \emph{J. Combin. Theory Ser. B} 62 (1994), 180-181

\bibitem{ExpManyColChooseThomPap}
  C. Thomassen, Exponentially many 5-list-colorings of planar graphs
  \emph{J. Combin. Theory Ser. B} 97 (2007), 571-583
\end{thebibliography}
\end{document}